\documentclass[11pt]{article}
\usepackage{amsthm,amsmath,amssymb,bbm}
\usepackage{color}
\def\bbe{\mathbb E}
\def\bbp{\mathbb P}
\def\bbr{\mathbb R}
\def\bbn{\mathbb N}
\def\bbone{{\mathbbm 1}}

\def\ugam{\Gamma}
\def\lam{\lambda}
\def\caln{{\cal N}}

\def\calh{{\cal H}}
\def\cala{{\cal A}}
\theoremstyle{plain}
\newtheorem{thm}{Theorem}[section]
\newtheorem{exmp}{Example}[section]
\newtheorem{cor}{Corollary}[section]
\newtheorem{prop}{Proposition}[section]
\newtheorem{lem}{Lemma}[section]
\theoremstyle{definition}
\newtheorem{rem}{Remark}[section]
\numberwithin{equation}{section}

\begin{document}
\title{On Stein's Method for Infinitely Divisible Laws With Finite First Moment}
\author{Benjamin Arras\thanks{Universit\'e Pierre et Marie Curie, Laboratoire Jacques-Louis Lions, 4 place Jussieu, 
75005 Paris; arrasbenjamin@gmail.com.}\; and Christian Houdr\'e\thanks{Georgia Institute of Technology, 
School of Mathematics, Atlanta, GA 30332-0160; houdre@math.gatech.edu. 
I would like to thank the Institute for Mathematical Sciences of the National University of Singapore for its hospitality 
and support as well as the organizers of  the 2015 Workshop on New Directions in Stein's Method.  
My participation in this workshop regenerated my interests on these topics and ultimately 
led to the present paper.   
Research supported in part by the grants \# 246283 and \# 524678 from the Simons Foundation.  
This material is also based, in part, upon work supported by the 
National Science Foundation under Grant No. 1440140, while this author was in residence at the Mathematical 
Sciences Research Institute in Berkeley, California, during the Fall semester of 2017.  
Finally, many thanks are due to the Universit\'e Pierre et Marie Curie for its 
hospitality while part of this research was carried out.  
\newline\indent  Keywords:  Infinite Divisibility, Self-decomposability, Stein's Method, Stein-Tikhomirov's Method, Weak Limit Theorems, 
Rates of Convergence, Kolmogorov Distance, 
Smooth Wassertein Distance.
\newline\indent MSC 2010: 60E07, 60E10, 60F05.}}

\maketitle
\begin{abstract}
We present, in a unified way, a Stein methodology for infinitely 
divisible laws (without Gaussian component) having finite first moment. Based on a correlation representation, we obtain a characterizing non-local Stein operator which boils down to classical Stein operators in specific examples. Thanks to this characterizing operator, we introduce various extensions of size bias and zero bias distributions and prove that these notions are closely linked to infinite divisibility. Combined with standard Fourier techniques, these extensions also allow obtaining explicit rates of convergence for compound Poisson approximation in particular towards the symmetric $\alpha$-stable distribution. Finally, in the setting of non-degenerate self-decomposable laws, by semigroup techniques, we solve the Stein equation induced by the characterizing non-local Stein operator and obtain quantitative bounds in weak limit theorems for sums of independent random variables going back to the work of Khintchine and L\'evy.
\end{abstract}

\section{Introduction}
Since its inspection in the normal setting (see \cite{Stein}) Stein's
method of approximation has enjoyed tremendous successes in both theory
and applications.  Starting with Chen's \cite{C} initial extension to the
Poisson case the method has been developed for various
distributions such as compound Poisson, geometric, negative binomial,
exponential, Laplace. (We refer the reader to Chen,
Goldstein and Shao~\cite{CGS} or Ross~\cite{R} for good introductions to the method,
as well as more precise and complete references.)

The methodology developed for the distributions just mentioned is
often ad hoc, the fundamental equations changing from one law 
to another and it is therefore not always easy to see their common 
underlying thread/approach.  There is, however, a class of random variables 
for which a common
methodology is possible.  The class we have in mind is the infinitely
divisible one, and it is the purpose of these notes to study
Stein's method in this context.  Our results will, in particular, 
provide a common framework for all the examples mentioned above 
in addition to presenting new ones.

As far as the content of the paper is concerned, the next section
introduces the basic infinitely divisible terminology and some examples. 
The third one provides a functional characterization of infinitely 
divisible laws from which distance estimates follows: this is the content of Theorem \ref{thm3.1}. Various
comparisons with previously known situations are then made. In particular, we respectively extend the notions of size-bias and zero-bias distributions in Corollary \ref{cor:ExtSizeBias} and in Proposition \ref{local} to infinitely divisible distributions with L\'evy measure satisfying certain moment assumptions.
Section 4 shows how the new characterizations obtained in the previous section 
lead, via Fourier methods, and for infinitely divisible 
sequences, rates of convergence results in either Kolmogorov or smooth Wasserstein 
distance, in particular for compound Poisson approximations.  
In Section \ref{sec:SE},  the solution to the fundamental equation is  
presented and its properties are studied when the target limit law belongs to the class of non-degenerate self-decomposable distributions. 
In Section \ref{sec:AP}, the developed Stein methodology is applied to obtain some quantitative 
approximation results for classical weak limit 
theorems for sums of independent random variables leading to Theorem \ref{thm:GWLT} which is the main result of the section and is complemented by some more explicit corollaries. 
Finally, we conclude this manuscript by addressing further extensions 
of our ideas and results for potential future work.  

\section{Preliminaries}

Let $X\sim ID(b,\sigma^2,\nu)$ be an infinitely divisible random 
variable, i.e., let $\varphi$ the characteristic function of $X$ be given
for all $t\in\bbr$ by
\begin{equation}\label{eq2.1}
\varphi(t)=
e^{itb-\sigma^2 \frac{t^2}2+\int^{+\infty}_{-\infty}
(e^{itu}-1-itu{\bbone}_{|u|\le 1})\nu(du)},
\end{equation}
for some $b\in\bbr$, $\sigma\ge 0$ and a positive Borel measure $\nu$ on 
$\bbr$ such that $\nu (\{0\})=0$ and 
$\int^{+\infty}_{-\infty}(1\wedge u^2)\nu (du)<+\infty$. 
The measure 
$\nu$ is called the L\'evy measure of $X$, and $X$ is said to be without
Gaussian component (or to be purely Poissonian or purely non-Gaussian)
whenever $\sigma^2=0$. (We refer the reader to Sato \cite{S}, for a 
good introduction to infinitely divisible laws and L\'evy processes.) 
The representation \eqref{eq2.1} is the standard one we will mainly be using
throughout these notes with the (unique) generating triplet
$(b,\sigma^2,\nu)$. However, different types of representation for the
characteristic function are also possible. Two of these are presented
next.  First, if $\nu$ is such that $\int_{|u|\le 1} |u|\nu (du)<+\infty$,
then \eqref{eq2.1} becomes
\begin{equation}\label{eq2.2}
\varphi(t)=e^{itb_0-\sigma^2\frac{t^2}2+\int^{+\infty}_{-\infty}
(e^{itu}-1)\nu (du)},
\end{equation}
where $b_0 = b-\int_{|u|\le 1}u\nu (du)$ is called the {\it drift} of $X$.
This representation of the characteristic function is cryptically
expressed as $X\sim ID(b_0,\sigma^2,\nu)_0$. Second, if $\nu$ is
such that $\int_{|u|>1} |u|\nu(du)<+\infty$, then
\eqref{eq2.1} becomes
\begin{equation}\label{eq2.3}
\varphi(t)=e^{itb_1-\sigma^2\frac{t^2}2 + \int^{+\infty}_{-\infty}
(e^{itu}-1-itu)\nu (du)},
\end{equation}
where $b_1=b+\int_{|u|>1} u\nu (du)$ is called the {\it center} of $X$.
This last representation is now cryptically written as 
$X\sim ID(b_1,\sigma^2,\nu)_1$. In fact, $b_1=\bbe X$  as, for any
$p>0$, $\bbe |X|^p<+\infty$ is equivalent to $\int_{|u|>1}|u|^p\nu (du)
<+\infty$.  Also, for any $r>0$, $\bbe e^{r|X|}<+\infty$ is equivalent to
$\int_{|u|>1} e^{r|u|}\nu (du)<+\infty$.

Various choices of generating triplets $(b,\sigma^2,\nu)$ provide various
classes of infinitely divisible laws.  The triplet $(b,0,0)$
corresponds to a degenerate random variable, $(b,\sigma^2,0)$ to a normal
one with mean $b$ and variance $\sigma^2$, the choice $(\lambda,0,
\lambda\delta_1)$, where $\lambda >0$ and where $\delta_1$ is the 
Dirac measure at 1, corresponds to a Poisson random variable
with parameter $\lambda$.  For $\nu$ finite, with the choice 
$b_0=b-\int_{|u|\le 1} u\nu (du)=0$, $\sigma^2=0$ and further setting
$\nu (du)=\nu (\bbr)\nu_0 (du)$, where $\nu_0$ is a Borel probability
measure on $\bbr$,  \eqref{eq2.2} becomes 
\begin{equation}\label{eq2.4}
\varphi(t)=e^{\nu(\bbr)\int^{+\infty}_{-\infty} (e^{itu}-1)
\nu_0 (du)},
\end{equation}
i.e., $X$ is compound Poisson: $X\sim {\rm CP}(\nu(\bbr), \nu_0)$. 
Next, let $X\sim \caln Bin^0
(r,p)$, i.e., let the support of $X$ be the non-negative 
integers and let 
$$\bbp (X=k)=\frac{\ugam (r+k)}{\ugam(r)k!} p^r(1-p)^k,\quad
k=0,1,2,\dots$$
where $r>0$ and $0<p<1$. Then, $X\sim ID(b,0,\nu)$ with 
$\nu (du)=r\sum^\infty_{k=1}k^{-1}{q^k}\delta_k(du)$ and 
$b_0=b-\int_{|u|\le 1} u\nu (du)=0$, i.e., $b=rq$, and so $\bbe X = rq/p$, 
where as usual
$q=1-p$. If instead, $X\in\caln Bin(r,p)$, i.e., if
$$\bbp (X=k)=\frac{\ugam (r+k-1)}{\ugam (r)(k-1)!} p^r(1-p)^{k-1},
\quad k=1,2,\dots$$
then $X\sim ID(b,0,\nu)$ with $b=1+rq$ and 
$\nu (du)=r\sum^\infty_{k=1}
k^{-1}{q^k}\delta_k(du)$ and so $\bbe X = r/p$.  
If $X$ has a Gamma distribution with parameters
$\alpha > 0$ and $\beta >0$, i.e., if $X$ has density 
${\beta^\alpha}{\ugam(\alpha)}^{-1} 
x^{\alpha-1}e^{-\beta x}{\bbone}_{(0,+\infty)}(x)$, then 
$X\sim ID (b,0,\nu)$ with $\nu (du)=\alpha{e^{-\beta u}}u^{-1}
\bbone_{(0,+\infty)}(u)du$ and 
$b_0=0$, i.e., $b=\int^1_0 \alpha e^{-\beta u} du=\alpha(1-e^{-\beta})/\beta$. 
If $X$ is the standard Laplace distribution with
density $e^{-|x|}/2$, $x\in\bbr$, then $X\sim ID(b,0,\nu)$
where ${\nu (du)} = {e^{-|u|}}{|u|^{-1}}du$, $u\ne 0$ and
$b_0=0$, i.e., $b=\int_{|u|\le 1} u{e^{-|u|}}{|u|^{-1}} du=0$. More
generally, if $X$ has a two-sided exponential distribution with 
parameters $\alpha >0$ and $\beta >0$, i.e., if $X$ has density
${\alpha\beta}{(\alpha +\beta)^{-1}} (e^{-\alpha x}\bbone_{[0+\infty)} (x) 
+e^{\beta x}\bbone_{(-\infty,0)}(x))$, $x\in \bbr$, then once more 
$X\sim ID (b,0,\nu)$ with ${\nu (du)}/{du} ={e^{-\alpha u}}u^{-1} 
\bbone_{(0,+\infty)} (u) - {e^{\beta u}}u^{-1} 
\bbone_{(-\infty,0)} (u)$ and $b_0=0$, i.e., 
$b=\int_{|u|\le 1} u\nu (du)=\alpha^{-1} (1-e^{-\alpha})-
\beta^{-1} (1-e^{-\beta})$.  Finally, if $X$ is a stable vector, say, on $\bbr^d$, 
then $\nu$ is
given by $\nu (B)=\int_{S^{d-1}}\sigma(d\xi)\int^\infty_0\bbone_B
(r\xi)\frac{dr}{r^{1+\alpha}}$, $0<\alpha <2$, where $\sigma$ 
(the spherical component of $\nu$) is a finite positive measure on 
the unit sphere $S^{d-1}$ of $\bbr^d$ ($S^0 = \{-1, 1\}$) and $B$ a Borel set of $\mathbb{R}^d$. 
Now, $X$ is symmetric 
if and only if $\sigma$ is symmetric in which case the characteristic function of $X$ becomes
$\varphi (t)=e^{-C_\alpha\int_{S^{d-1}}|\langle t,\xi\rangle|^\alpha
\sigma (d\xi)}$, where $C_\alpha=\frac{\sqrt\pi \ugam ((2-\alpha)/2)}
{\alpha 2^\alpha\ugam ({(1+\alpha)}/2)}$.  Moreover, $X$ is
rotationally invariant if and only if $\sigma$ is uniform on 
$S^{d-1}$ and then $\varphi(t)=e^{-C_{\alpha,d} \|t\|^\alpha}$,
$t\in\bbr^d$, where $C_{\alpha ,d}=C_\alpha\int_{S^{d-1}}|\langle
\frac t{\|t\|},\xi\rangle|^\alpha\sigma (d\xi)$ does not depend on $t$.
In particular, if $\sigma$ is uniform on $S^{d-1}$, 
$C_{\alpha, d}=\ugam (d/2)\ugam ((2-\alpha)/2)/\alpha 2^\alpha
\ugam ((d+\alpha)/2)$.  Therefore, if $X$ is an $\alpha$-stable random variable, 
its L\'evy measure is given by
\begin{align}\label{StableLM}
\nu(du):= \left(c_1\frac{1}{u^{\alpha+1}}\bbone_{(0,+\infty)}(u)
+c_2 \frac{1}{|u|^{1+\alpha}}\bbone_{(-\infty,0)}(u)\right)du.
\end{align}
where $c_1,c_2\geq 0$ are such that $c_1+c_2>0$. The symmetric case 
corresponds to $c_1=c_2$ 
and $b=0$, which we write as $X\sim S\alpha S$.  The class of infinitely divisible distributions is vast and 
also includes, among others, Student's t-distribution, the Pareto distribution, 
the F-distribution, the Gumbel distribution. Besides these classical examples let us mention 
that any log convex density on $(0,+\infty)$ is infinitely divisible and so are many 
classes of log-concave measures (see \cite[Chapter 10]{S}).

{\it  In the rest of this text, the terminology 
 L\'evy measure is used to denote a positive Borel measure on $\bbr$ 
 which is atomless at the origin and 
 which integrates out the function $f(x) = \min(1, x^2)$. Moreover, as in \cite{S}, for a real function $f$, increasing means $f(s)\leq f(t)$ for $s<t$, and decreasing means $f(s)\geq f(t)$ for $s<t$. When the equality is not allowed, we say strictly increasing or strictly decreasing. Finally, we denote by $\ln$ the natural logarithm.}

\section{Characterization and Coupling}\label{sec:CC}

We are now ready to present our first result which characterizes ID laws
via a functional equality.  This functional equality involves Lipschitz
functions and we have to agree on what is meant by ``Lipschitz."
Below, the functions we consider need not be defined on the whole of
$\bbr$ but just on a subset of $\bbr$ containing $R_X$, the range of
$X\sim ID(b,0,\nu)$, and $R_X+S_\nu$, where $S_\nu$ is the support of 
$\nu$.  For example, if $X$ is a Poisson random variable, a Lipschitz
function (with Lipschitz constant 1) is then defined on $\bbn$ and is
such that $|f(n+1)-f(n)|\le 1$, for all $n\in\bbn$. 

Now, as well known, a Lipschitz function $f$ defined on a subset $S$
of $\bbr$ can be extended, without increasing its Lipschitz semi-norm
where, as usual, the Lipschitz semi-norm of $f$ is 
$\|f\|_{Lip}=\sup_{x\ne y}|f(x)-f(y)|/|x-y|$.
(This can be done in various ways, e.g., for any $x\in\bbr$, let
$\tilde f(x)=\inf_{z\in S} (f(z)+|x-z|)$. Then, for any $y\in \bbr$,
$\tilde f(x)\le \inf_{z\in S}(f(z)+|x-y|+|y-z|)=|x-y|+\tilde f(y)$.  Another extension is given via 
$\bar f(x)=\sup_{z\in S} (f(z)-|x-z|)$.)
Now, in the integral representations and as integrands, $f$ and $\tilde f$ are indistinguishable.
Therefore, and since we do not wish to distinguish between, say, discrete
and continuous random variables, in the sequel, Lipschitz will be understood
in the classical sense, i.e., $f\in Lip$ with Lipschitz constant $C>0$, if $|f(x)-f(y)|\le C|x-y|$,
for all $x,y\in\bbr$, and $f$ could then be viewed as the Lipschitz
extension $\tilde f$.  Throughout the text, the space of real-valued 
Lipschitz functions defined on some domain $D$ is denoted by $Lip(D)$,  
while the space of bounded Lipschitz ones is denoted by $BLip(D)$.  
Endowed with the norm $\|\cdot\|_{BLip}= \max(\|\cdot\|_{\infty}, \|\cdot\|_{Lip})$, 
$BLip(D)$ is a Banach space, with $\|f\|_\infty=\sup_{x\in\mathbb{R}}|f(x)|$. Finally, we denote the closed 
unit ball of $Lip(D)$ by $Lip(1)$ and similarly $BLip(1)$ denotes the closed unit ball 
of $BLip(D)$.

\begin{thm}\label{thm3.1}
Let $X$ be a random variable such that $\bbe|X|<+\infty$. Let $b\in\bbr$
and let $\nu$ be a positive Borel measure on $\bbr$ such that 
$\nu (\{0\})=0$, $\int^{+\infty}_{-\infty} (1\wedge u^2)\nu (du)<
+\infty$ and $\int_{|u|>1}|u|\nu(du)<+\infty$. Then,
\begin{equation}\label{eq3.1}
\bbe\left(Xf(X)-bf(X)-\int^{+\infty}_{-\infty}(f(X+u)-f(X)\bbone_{|u|\le 1})
u\nu (du)\right)=0,
\end{equation}
for all bounded Lipschitz function $f$ if and only if $X\sim ID(b,0,\nu)$.
\end{thm}

\begin{proof} Note at first that, by the assumption on $\nu$ and $f$, 
the left-hand side of \eqref{eq3.1} is well defined and note also that throughout 
the proof, interchanges of integrals and expectations are perfectly
justified.  The direct part of the statement is, in fact, a particular
case of a covariance representation obtained in \cite[Proposition 2]{HPAS}.  Indeed, if
$X\sim ID(b,0,\nu)$ and if $f$ and $g$ are two bounded Lipschitz
functions, then
\begin{equation}\label{eq3.2}
Cov (f(X),g(X))=\int^1_0 \bbe_z\int^{+\infty}_{-\infty}
(f(X_z+u)-f(X_z))(g(Y_z+u)-g(Y_z)) \nu (du)dz,
\end{equation}
where $(X_z,Y_z)$ is a two-dimensional ID vector with characteristic function defined by
$\varphi_z(t,s)=(\varphi(t)\varphi(s))^{1-z}\varphi(t+s)^z$, for all $z\in [0,1]$, 
for all $s,t\in\mathbb{R}$ and where $\varphi$ is the characteristic function of $X$.
%a two-dimensional ID vector with parameter $(b,b)$ and 
%L\'evy measure $z\nu_1(du,dv)+(1-z)\nu_0 (du,dv)$, where 
%$\nu_0 (du,dv)=\nu (du)\delta_0 (dv)+\delta_0 (du)\nu (dv)$ and where
%$\nu_1(du,dv)$ is the measure $\nu$ restricted to the diagonal of $\bbr\times\bbr$.
Since $X_z=_dY_z=_d X$ where $=_d$ stands for equality in distribution, taking $g(y)=y$ (which is possible by first taking $g_R(y)=y\bbone_{|y|\leq R}+R\bbone_{y\geq R}-R\bbone_{y\leq -R}$ for $R>0$ and then passing to the limit), 
\eqref{eq3.2} becomes
\begin{equation}\label{eq3.3}
\bbe Xf(X)-\bbe X\bbe f(X)=\bbe\int^{+\infty}_{-\infty}(f(X+u)-f(X))
u\nu (du).
\end{equation}
To pass from \eqref{eq3.3} to \eqref{eq3.1}, just note that since 
$\bbe |X|<+\infty$, differentiating the characteristic function of $X$,
shows that $\bbe X=b+\int_{|u| > 1}u\nu (du)$.
To prove the converse part of the equivalence, i.e., that \eqref{eq3.1},
when valid for all $f$ bounded Lipschitz, implies that 
$X\sim ID (b,0,\nu)$, it is enough to apply \eqref{eq3.1} to sines and
cosines or equivalently to complex exponential functions and then to identify
the corresponding characteristic function. For any $s\in \bbr$,
let $f(x)=e^{isx}$, $x\in \bbr$. \eqref{eq3.1} becomes
\begin{equation}\label{eq3.4}
\bbe Xe^{isX} -b\bbe e^{isX} = \bbe e^{isX}\int^{+\infty}_{-\infty}
(e^{isu} -\bbone_{|u|\le 1})u\nu (du).
\end{equation}
Let $\varphi (s)=\bbe e^{isX}$, then \eqref{eq3.4} rewrites as
\begin{equation}\label{eq3.5}
\varphi'(s) = i\varphi (s)\left( b+\int^{+\infty}_{-\infty}
(e^{isu} -\bbone_{|u|\le 1})u\nu (du)\right).
\end{equation}
Integrating out the real and imaginary parts of \eqref{eq3.5} leads, for any $t\ge 0$, to:  
\begin{align*}
\varphi (t)&= \exp\bigg(itb+i\int^t_0\int^{+\infty}_{-\infty} 
(e^{isu} -\bbone_{|u|\le 1})u\nu (du)ds\bigg)\\
&= \exp\bigg(itb+ i\int^{+\infty}_{-\infty}\int^t_0 (e^{isu} -\bbone_{|u|\le 1}
)u\,ds\,\nu (du)\bigg)\\
&= \exp\bigg(itb+\int^{+\infty}_{-\infty} (e^{itu} -1-itu\bbone_{|u|\le 1})
\nu (du)\bigg).
\end{align*}
A similar computation for $t\le 0$ finishes the proof.    
\end{proof}

\begin{rem}\label{rem3.2}

(i) Both the statement and the proof of Theorem~\ref{thm3.1} carry over
to $X\sim ID(b,\sigma^2,\nu)$. The corresponding version of
\eqref{eq3.1} which characterizes $X$ is then
\begin{equation}\label{eq3.6}
\bbe (Xf(X)-bf(X)-\sigma^2f'(X)-\int^{+\infty}_{-\infty} (f(X+u)-f(X)
\bbone_{|u|\le 1})u\nu (du))=0.
\end{equation}
In particular, if $\nu=0$, \eqref{eq3.6} is the well known 
characterization of the normal law with mean $b=\bbe X$ and 
variance $\sigma^2$.

(ii) There are other ways to restate Theorem~\ref{thm3.1} for $X$ 
such that $\bbe |X|<+\infty$. For example, 
if $X\sim ID(b,0,\nu)$, then
\begin{equation}\label{eq3.7}
Cov (X, f(X))=\bbe \int^{+\infty}_{-\infty} (f(X+u)-f(X))u\nu (du).
\end{equation}
Conversely, if \eqref{eq3.7} is satisfied for all bounded Lipschitz 
functions $f$, then $X\sim ID(b,0,\nu)$, where $b=\bbe X-
\int^{+\infty}_{-\infty} u\bbone_{|u|>1} \nu (du)$.  In case
$\int_{|u|\le 1}|u|\nu (du)<+\infty$, a further characterizing 
representation is
\begin{equation}\label{eq3.8}
\bbe X f(X)-\left(b-\int_{|u|\le 1}u\nu (du)\right)\bbe f(X)=
\bbe\int^{+\infty}_{-\infty} f(X+u)u\nu (du),
\end{equation}
or equivalently,
\begin{equation}\label{eq3.9}
\bbe X f(X)-b_0
\bbe f(X)= \bbe\int^{+\infty}_{-\infty} f(X+u)u\nu (du),
\end{equation}
i.e.,
\begin{equation}\label{eq3.10}
\bbe X f(X)-\left(\bbe X-\int^{+\infty}_{-\infty}u\nu (du)\right)
\bbe f(X)= \bbe\int^{+\infty}_{-\infty} f(X+u)u\nu (du).
\end{equation}
\end{rem}
\noindent
Let us now specialize \eqref{eq3.1}, \eqref{eq3.8} and \eqref{eq3.9} to various cases, 
some known, some new.

\begin{exmp}
(i) Of course, if $X\sim ID(\lam,0,\lam\delta_1)$, i.e., when 
$X$ is a Poisson random variable with parameter $\lam = \bbe X>0$, then 
\eqref{eq3.1} becomes the familiar
\begin{equation}\label{eq3.11}
\bbe Xf(X)=\bbe X\bbe f(X+1).
\end{equation}
More generally, if $\nu (du)=c\delta_1(du)$, then
$X\sim ID(b=\bbe X,0,c\delta_1)$. 

(ii) If $X\sim\caln Bin^0(r,p)$, then, as indicated before, $b_0= 0$, $\nu(du) = 
r\sum_{k=1}^{+\infty}q^k\delta_k(du)/k$, with $q=1-p$, 
and so \eqref{eq3.9} becomes
\begin{align}\label{eq3.12}
\bbe Xf(X)&=r\bbe \sum^\infty_{k=1}f(X+k)q^k \nonumber\\
&=rq\bbe f(X+1) + r\bbe \sum^\infty_{k=2}f(X+k)q^k \nonumber\\
&=rq\bbe f(X+1) + \sum_{k=2}^\infty\sum_{j=0}^\infty f(j+k)r\frac{\ugam (r+j)}
{\ugam(r)j!} p^rq^jq^k \nonumber\\
&=rq\bbe f(X+1) + \sum_{\ell=2}^\infty f(\ell)p^rq^\ell\frac{r}{\ugam(r)}
\sum_{k=0}^{\ell -2}\frac{\ugam (r+k)}{k!} \nonumber\\
&=rq\bbe f(X+1) + \sum_{\ell=2}^\infty f(\ell)p^rq^\ell\frac{1}{\ugam(r)}
\frac{\ugam (r+\ell -1)}{(\ell-2)!} \nonumber\\
&=rq\bbe f(X+1) + q\bbe Xf(X+1), 
\end{align}
since $\ugam(t+1)= t\ugam(t), t>0$.  Now, \eqref{eq3.12} is exactly 
the negative binomial characterizing identity obtained in \cite{BP}.   

(iii) If $X\sim \caln Bin(r,p)$, \eqref{eq3.1} becomes
\begin{align}\label{eq3.13}
\bbe Xf(X)&=\bbe f(X)+r\bbe\sum^\infty_{k=1} 
f(X+k)q^k\nonumber \\
&=\bbe f(X) + \sum_{k=1}^\infty\sum_{j=1}^\infty f(j+k)r\frac{\ugam (r+j-1)}
{\ugam(r)(j-1)!} p^rq^{j-1}q^k \nonumber\\
&=\bbe f(X) + \sum_{\ell=2}^\infty f(\ell)p^rq^{\ell-1}\frac{r}{\ugam(r)}
\sum_{k=1}^{\ell -1}\frac{\ugam (r+k-1)}{(k-1)!} \nonumber\\
&=\bbe f(X) + \sum_{\ell=2}^\infty f(\ell)p^rq^{\ell-1}\frac{1}{\ugam(r)}
\frac{\ugam (r+\ell -1)}{(\ell-2)!} \nonumber\\
&=\bbe f(X) + q \bbe((r+X-1)f(X+1)),   
\end{align}
which, in view of the previous example, 
is exactly the expected characterizing equation since $X-1\sim \caln Bin^0(r,p)$.

(iv) If $X\sim {\rm CP}(\nu(\bbr), \nu_0)$, then \eqref{eq3.1} or \eqref{eq3.8}--\eqref{eq3.10} becomes 
\begin{equation}\label{cp}
\bbe Xf(X)=\bbe\int^{+\infty}_{-\infty} f(X+u)u\nu (du)=
\nu (\bbr)\bbe\int^{+\infty}_{-\infty}
f(X+u)u\nu _0(du),
\end{equation}
and \eqref{cp} is the characterizing identity for the compound Poisson law given  
in \cite{BCL}.   

(v) If $X$ is the standard Laplace distribution 
with density $e^{-|x|}/2, x\in \bbr$.  Then, $\nu(du)/du=\exp(-|u|)/|u|$, $b=0$, $\int_{-1}^{1}u\nu(du) = 0$, 
and \eqref{eq3.1} or \eqref{eq3.8}-\eqref{eq3.10} becomes
\begin{align}\label{exp}
\bbe Xf(X)&=\bbe\int^{+\infty}_{-\infty} f(X+u){\rm sign}(u){e^{-|u|}}du \nonumber \\
&=2\bbe f(X+\operatorname{sign}(L)L)\nonumber\\
&=\bbe  \int_{0}^{+\infty}(f(X+u)-f(X-u))e^{-u}du\nonumber\\
& =\bbe( f(X+Y) - f(X-Y)),
\end{align}
where $L$ is a standard Laplace random variable independent of $X$, while $Y$ is a standard exponential 
random variable independent of $X$.

(vi) If $X$ is a Gamma random variable with parameters $\alpha > 0$ and $\beta > 0$, then,  see \cite{Luk94},  
for $f$ ``nice'',  
\begin{align}\label{eq:laguerre}
\bbe((\beta X-\alpha)f(X))=\bbe Xf'(X).
\end{align}
But, $X$ is infinitely divisible with $\nu(du)
=\alpha \bbone_{(0,+\infty)}(u)\exp(-\beta u)/udu$, 
and it follows from \eqref{eq3.1} that
\begin{align}\label{GamThm3.1}
\bbe Xf(X)&=\alpha\frac{1-e^{-\beta}}{\beta}\bbe f(X)\nonumber\\
&\quad +\alpha\, \bbe\int_0^\infty (f(X+u)-f(X)\bbone_{|u|\leq 1})e^{-\beta u}du.
\end{align}
Equivalently from \eqref{eq3.8}--\eqref{eq3.10}, since $b_0=0$, and since $\bbe X = \alpha/\beta$, 
\begin{align*}
\bbe X f(X)&=\bbe \int_{-\infty}^{+\infty} f(X+u)u\nu(du)\\
&=\alpha \bbe \int_0^{+\infty}f(X+u)e^{-\beta u}du\\
&=\frac{\alpha}{\beta}\bbe f(X+Y)\\
&=\bbe X\bbe f(X+Y),
\end{align*}
where $Y$ is an exponential random variable, with parameter $\beta$, independent of $X$.
Thus, Theorem~\ref{thm3.1} implies the existence of an additive size bias (see e.g. \cite[Section 2]{CGS}) distribution for the gamma distribution. 
Moreover, it says that the only probability measure which has an additive 
exponential size bias distribution is the gamma one. 

(vii) To complement this very partial list, let us consider an example, where the 
literature is sparse (for the symmetric case, see \cite{AMPS17,X17}), namely the stable case. 
At first, let $X$ be a symmetric $\alpha$-stable random variable with
$\alpha \in (1,2)$, 
i.e., let $X\sim S\alpha S$. Then, $b=0$ and \eqref{eq3.1} becomes 
\begin{align}
\bbe Xf(X)&=\bbe\int^{+\infty}_{-\infty} \big(f(X+u)-f(X)\bbone_{\{|u|\leq 1\}}\big)u\nu(du)\nonumber \\
&= c\bigg(-\bbe\int_{-\infty}^0\big(f(X+u)-f(X)\bbone_{\{|u|\leq 1\}}\big)\frac{du}{(-u)^{\alpha}}\nonumber\\
&\qquad \qquad \qquad +\bbe\int_{0}^{+\infty}\big(f(X+u)-f(X)\bbone_{\{|u|\leq 1\}}\big)\frac{du}{u^{\alpha}}\bigg)\nonumber\\
&= c\, \bbe \int_0^{+\infty}\big( f(X+u)-f(X-u)\big)\frac{du}{u^{\alpha}}\nonumber,
\end{align}
and, therefore, the previous integral is a fractional operator acting on the test function $f$. 
Let us develop this point a bit more by adopting the notation of  \cite[Section 5.4]{SKM93}.  The Marchaud fractional 
derivatives, of order $\beta$, of (a sufficiently nice function) $f$ are defined by
\begin{align*}
&\mathbf{D}^\beta_+(f)(x):=\dfrac{\beta}{\Gamma(1-\beta)}\int_{0}^{+\infty}\dfrac{f(x)-f(x-u)}{u^{1+\beta}}du,\\
&\mathbf{D}^\beta_{-}(f)(x):=\dfrac{\beta}{\Gamma(1-\beta)}\int_{0}^{+\infty}\dfrac{f(x)-f(x+u)}{u^{1+\beta}}du.
\end{align*}
Note that the above operators are well defined for bounded Lipschitz functions as soon as $\beta\in (0,1)$. 
Then, in a more compact form
\begin{align}\label{eq:StableFrac}
\bbe Xf(X)&= C_\alpha\bbe (\mathbf{D}^{\alpha-1}_+(f)(X)-\mathbf{D}^{\alpha-1}_{-}(f)(X)),
\end{align}
where $C_\alpha= c\Gamma(2-\alpha)/(\alpha-1)$.  
Now, for $X\sim S\alpha S$,  \cite[Proposition 3.2]{ARW00} or \cite[Theorem 4.1]{X17} put forward the following characterizing equation
\begin{align}\label{eq:fraclapchar}
\bbe Xf'(X)=\alpha\bbe \Delta^{\frac{\alpha}{2}}f(X),
\end{align}
where $\Delta^{{\alpha}/{2}}$ is the fractional Laplacian defined via
\begin{align*}
\Delta^{\frac{\alpha}{2}}f(x):= d_\alpha\int_{\mathbb{R}}\dfrac{f(x+u)-f(x)}{|u|^{1+\alpha}}du, 
\end{align*}
and where $d_\alpha = \Gamma(1+\alpha)\sin (\pi\alpha)/(2\pi\cos(\alpha\pi/2))$.  

Taking $f'$ (nice enough) as a test function in the right-hand side of \eqref{eq:StableFrac}, leads to
\begin{align*}
 \mathbf{D}^{\alpha-1}_+(f')(x)-\mathbf{D}^{\alpha-1}_{-}(f')(x)&
 :=\dfrac{\alpha-1}{\Gamma(2-\alpha)}\int_0^{+\infty}\frac{f'(x+u)-f'(x-u)}{u^{\alpha}}du.  
\end{align*}
Moreover, 
\begin{align}\label{rep:FracLap}
\Delta^{\frac{\alpha}{2}}f(x)&:= d_\alpha\int_{\mathbb{R}}\dfrac{f(x+u)-f(x)}{|u|^{1+\alpha}}du \nonumber \\
&=d_\alpha\bigg(\int_0^{+\infty} \int_0^u f'(x+t)dt\frac{du}{u^{1+\alpha}}
-\int_{-\infty}^0 \int_{u}^0 f'(x+t)dt \frac{du}{(-u)^{\alpha+1}}\bigg) \nonumber \\
&=\frac{d_\alpha}{\alpha}\bigg(\int_0^{+\infty}f'(x+u)\frac{du}{u^{\alpha}}-\int_{0}^{+\infty}f'(x-v)\frac{dv}{v^\alpha}\bigg)\nonumber \\
&=\frac{d_\alpha}{\alpha}\int_0^{+\infty}(f'(x+u)-f'(x-u))\frac{du}{u^{\alpha}}, 
\end{align}
showing the equivalence of the two characterizing identities \eqref{eq:StableFrac} and \eqref{eq:fraclapchar} for $X\sim S\alpha S$.  
For the general stable case with L\'evy measure given, with $c_1\ne c_2$, by \eqref{StableLM}, then in a straightforward manner, 
\begin{align}\label{characstable}
\bbe Xf(X):=c_{2,\alpha}\bbe (\mathbf{D}^{\alpha-1}_+(f)(X))-c_{1,\alpha}\bbe (\mathbf{D}^{\alpha-1}_{-}(f)(X))
+\dfrac{(c_1-c_2)}{\alpha-1}\bbe f(X),
\end{align}
where 
\begin{align}\label{stableconstants}
c_{1,\alpha}=c_1\dfrac{\Gamma(2-\alpha)}{\alpha-1}, \qquad c_{2,\alpha}=c_2\dfrac{\Gamma(2-\alpha)}{\alpha-1}.
\end{align}

(viii)  Another class of infinitely divisible distributions which is of interest in 
a Malliavin calculus framework is the class of second order Wiener chaoses.  As well known, if $X$ belongs to this class and 
if $=_d$ denotes equality in distribution, then
\begin{align*}
X=_d\sum_{k=1}^{+\infty}\lambda_k (Z_k^2-1), 
\end{align*}
where $(Z_k)_{k\geq 1}$ is a sequence of iid standard normal random variables and where the sequence of reals 
$(\lambda_k)_{k\geq 1}$ is square summable.  Equivalently, the characteristic function of $X$  is given by 
\begin{align*}
\varphi(t)=\exp\left(\int_{-\infty}^{+\infty}(e^{itu}-1-itu)\nu(du)\right), 
\end{align*}
where 
\begin{align}\label{chaoslevy}
\frac{\nu(du)}{du}= \left(\sum_{\lambda\in \Lambda_+}\dfrac{e^{-u/(2\lambda)}}{2u}\right)\bbone_{(0,+\infty)}(u)
+\left(\sum_{\lambda\in \Lambda_-}\dfrac{e^{-u/(2\lambda)}}{2(-u)}\right)\bbone_{(-\infty,0)}(u), 
\end{align}
with $\Lambda_+=\{\lambda_k: \lambda_k> 0\}$ and $\Lambda_-=\{\lambda_k: \lambda_k< 0\}$. 
Thus, $X\sim ID(b,0,\nu)$ with $b=-\int_{|u|>1}u\nu(du)$ and $\nu$ as in \eqref{chaoslevy}. 
The corresponding characterizing equation is therefore:  
\begin{align}
\bbe X f(X)=\bbe \int_{-\infty}^{+\infty} (f(X+u)-f(X))u\nu(du), 
\end{align}
since also $\bbe X=0$.
\end{exmp}
\noindent
As a first corollary to Theorem~\ref{thm3.1}, the following characterizing identities result 
extends the notion of additive size bias distribution to infinitely divisible probability measure with 
finite non-zero mean and such that $\int_{-1}^{+1}|u|\nu(du)<+\infty$.

\begin{cor}\label{cor:ExtSizeBias}
Let $X$ be a nondegenerate random variable such that $\bbe|X|<+\infty$. Let $\nu$ be a L\'evy measure such that 
\begin{align}\label{HP:ExtSizeBias}
\int_{-\infty}^{+\infty}|u|\nu(du)<+\infty,    
\end{align}
and let $b_0 = b - \int_{-1}^{1}u\nu(du)$, $b\in\mathbb{R}$. Assume further that
\begin{align}
m_0^\pm = \max(\pm b_0, 0) + \int_\bbr \tilde{\nu}_\pm(du)\ne 0
\end{align}  
with $\tilde{\nu}(du) = u\nu(du)$. Then, 
\begin{align}\label{eq:ExtSizeBias}
\bbe X f(X)=m_0^+\bbe f(X+Y^{+})-m_0^-\bbe f(X+Y^{-}),
\end{align}
for all bounded Lipschitz functions $f$, where the random variables $Y^{+},Y^{-}$ and $X$ are 
independent with $Y^{+}$ and $Y^{-}$ having respective law
\begin{align}
\mu_{Y^\pm}(du) = \frac{b_0^\pm}{m_0^\pm}\delta_0(du) + \frac{\tilde{\nu}_\pm(du)}{m_0^\pm},
%&\mu_{Y^{-}}(du)=\frac{b_0^-}{m_0^-}\delta_0(du)+ \frac{\tilde{\nu}_-(du)}{m_0^-},
\end{align}
if and only if 
$X\sim ID(b,0,\nu)$.
\end{cor}

\begin{proof}
Let $f$ be a bounded Lipschitz function. By \eqref{eq3.9},
\begin{align}\label{another}
\bbe Xf(X)=b_0\bbe f(X)+\bbe\int_{-\infty}^{+\infty}f(X+u)u\nu(du).  
\end{align}
Now since since $\bbe |X|<+\infty$ 
and thanks to \eqref{HP:ExtSizeBias}, $\tilde{\nu}(du)=u\nu(du)$ is a finite signed measure and 
so its Jordan decomposition is given by 
$$
\tilde\nu(du) = \tilde\nu_+(du) - \tilde\nu_{-}(du) = u\bbone_{(0, +\infty)}(u)\nu(du) - (-u)\bbone_{(-\infty, 0)}(u)\nu(du).  
$$ 
Therefore, \eqref{another} becomes 
\begin{align*} 
\bbe Xf(X)=b_0^+\bbe f(X)-b_0^-\bbe f(X)+\bbe\!\int_0^{+\infty}\!\!\!\!\!\!f(X+u){\tilde{\nu}}_{+}(du) 
-\bbe\!\int_{-\infty}^0\!\!\!\!\!f(X+u)\tilde{\nu}_{-}(du).
\end{align*}
Now,  
\begin{align*}
m_0^+&=b_0^+ +\int_\mathbb{R}\tilde{\nu}_+(du)= b_0^+ + \int_0^{+\infty} u\nu(du), \\
m_0^-&=b_0^-+\int_\mathbb{R}\tilde{\nu}_-(du) = b_0^-+\int_{-\infty}^0(-u)\nu(du), 
\end{align*}
($b_0 = b_0^+ - b_0^-$  and $m_0^+ - m_0^-= \bbe X$) and, therefore, introducing the random variables $Y^+$ 
and $Y^-$ proves the direct implication. 
The converse implication follows directly from 
Theorem~\ref{thm3.1} or by first taking $f(x)=e^{i t x}, t \in \bbr$, in \eqref{eq:ExtSizeBias} and then, 
as previously done in the proof of the theorem, by solving a differential equation.
\end{proof}

\begin{rem}\label{rem:ExtSizeBias}
%\begin{itemize}
(i) If $m_0^{-}=0$ and $m_0^{+}\ne 0$, Proposition \ref{cor:ExtSizeBias} still holds with identity \eqref{eq:ExtSizeBias} being replaced by
\begin{align*}
\bbe X f(X)=m_0^+\bbe f(X+Y^{+})
\end{align*}
for all bounded Lipschitz functions $f$. A similar proposition holds true when $m_0^{-}\ne 0$ and $m_0^{+}=0$.
(ii) When $X\sim ID(b, 0, \nu)$ is non-negative, then necessarily the support of its L\'evy measure is 
in $(0,+\infty)$, the condition $\int_0^1ud\nu(u)<+\infty$ is automatically 
satisfied and $b_0 \ge 0$ (see \cite[Theorem 24.11]{S}).  In this context, for $\bbe X < +\infty$, 
$b_0=\bbe X-\int_0^{+\infty}u\nu(du)$, $m_0= m_0^+ =\bbe X$ 
and so, when $\bbe X > 0$, the characterizing equation \eqref{eq:ExtSizeBias} becomes, 
\begin{align}
\bbe Xf(X)=\bbe X \bbe f(X+Y),
\end{align}
where $Y$ is a random variable independent of $X$ whose law is given by 
\begin{align}
\mu_Y(du)=\frac{b_0}{\bbe X}\delta_0(du)+ \frac{u{\bbone}_{(0,+\infty)}(u)}{\bbe X}\nu(du).
\end{align}
This agrees with the standard notion of size bias distribution for finite mean non-negative 
random variable (see e.g. \cite{arratia2013size}) and recovers and extends a result there. \\
{\it The pair $(Y^+, Y^-)$ in \eqref{eq:ExtSizeBias} will be called the additive size-bias pair associated with $X$.}\\
(iii) For $X\geq 0$ with finite first moment, there is a natural relationship between size bias 
distribution and equilibrium distribution with respect to $X$ (see \cite{PR11}). 
Corollary~\ref{cor:ExtSizeBias} also leads naturally to an extension of this relationship. 
Namely, for $X$ as in the corollary and $f$ bounded Lipschitz, 
\begin{align*}
\bbe f(X)-f(0)
&=\bbe Xf'(UX),\\
&=m_0^+\bbe f' (U(X+Y^+)) -m_0^-\bbe f'(U(X+Y^-)),
\end{align*}
where $U$ is a uniform random variable on $[0,1]$ independent of $X$, $Y^+$ and $Y^-$.\\
%\end{itemize}
(iv) Let us consider a non-trivial example for which the assumption $\eqref{HP:ExtSizeBias}$ is not satisfied. 
Let $X$ be a second order Wiener chaos random variable such that, for all $k\geq 1$, $\lambda_k> 0$, with  
further $\sum_{k\geq 1} \lambda_k=+\infty$. Then, for the L\'evy measure $\nu_N$ given via 
\begin{align*}
\nu_N(du)=\bbone_{(0,+\infty)}(u)\sum_{k=1}^N\dfrac{\exp\left(-\frac{u}{2\lambda_k}\right)}{2u}du,
\end{align*}
we have
\begin{align*}
\int_{-1}^{+1}|u|\nu_N(du)=\sum_{k=1}^N\lambda_k\left(1-\exp\left(-\frac{1}{2\lambda_k}\right)\right) 
\geq \left(1-\exp\left(-\frac{1}{2\lambda_{\max}}\right)\right) \sum_{k=1}^N \lambda_k \underset{N\rightarrow +\infty}{\longrightarrow}+\infty
\end{align*}
where $\lambda_{\max}=\max_{k\geq 1} \lambda_k$. So, by monotone convergence, $\int_{-1}^{+1}|u|\nu(du)=+\infty$.  
In particular, the Rosenblatt distribution belongs to this class of second Wiener chaos random variables, since asymptotically $\lambda_k\underset{k\rightarrow+\infty}{\sim}C_D k^{D-1}$, 
for some $C_D>0$ and $D\in (0,1/2)$ (see Theorem 3.2 of \cite{VT13}).
%\end{itemize}
\end{rem}
\noindent

As seen in some of the examples presented  above, characterizing identities can 
involve local operators, e.g., derivatives, 
while our generic characterization is 
non-local,  involving difference operators.   
Let us explain, next, how to pass from one to the other also encouraging the reader    
to contemplate how this passage is linked to the notion of zero bias 
distribution (see \cite{GR97}) with an additive structure.

\begin{rem}\label{localnonlocal} 
Let us present a general methodology for $X\sim ID (b,0,\nu)$ such that $X\geq 0$ and $0< \bbe X<+\infty$ 
to pass from the 
non-local characterization of Theorem~\ref{thm3.1} to a local characterization. 
Again, since $X\geq 0$, then necessarily the support of $\nu$ is 
in $(0,+\infty)$, $\int_0^1u\nu(du)<+\infty$ and  $b_0\geq 0$ (see \cite{S}).   Hence, from the finite mean assumption 
and for all $v > 0$, 
$\eta(v) = \int_{v}^{+\infty}u\nu(du) < +\infty$.  Therefore, denoting by $\mu$ the law of $X$, 
for any bounded Lipschitz function $f$, 
\begin{align}\label{repder}
Cov(X,f(X))&= \bbe \int_0^{+\infty}(f(X+u)-f(X))u\nu(du)\nonumber \\
&=\int_0^{+\infty}\int_0^{+\infty}\left(\int_0^u f'(x+v)dv\right)u\nu(du)\mu(dx)\nonumber \\
&=\int_0^{+\infty}\int_0^{+\infty}f'(x+v)\eta(v)dv\mu(dx)\nonumber \\
&=\int_0^{+\infty} f'(y)(\eta\ast \mu)(dy), 
\end{align}
where $\eta\ast \mu$ is the convolution of the law $\mu$ with the positive 
Borel measure $\eta(dv) = \eta(v)dv$.  In case 
$\eta\ast \mu$ is absolutely continuous with respect to the Lebesgue measure, denoting its Radon-Nikod\'ym 
derivative by $h$, then $h(y)=\int_0^y\eta(y-v)\mu(dv)$, and \eqref{repder} becomes 
\begin{align}\label{repder2}
Cov(X,f(X)) =\int_0^{+\infty} f'(y)h(y)dy.  
\end{align}
In particular, when $X$ has an exponential distribution, then $h(y)=ye^{-y}$ and \eqref{repder2} becomes the classical relation
\begin{align}\label{caracexpo}
Cov(X,f(X))= \bbe Xf'(X).  
\end{align}
In general, $\eta\ast \mu$ is not a probability law, it is a positive measure, not necessarily finite, since 
$\int_0^{+\infty}\eta(v)dv = \int_0^{+\infty} u^2\nu(du)$.  
In case $X$ is nondegenerate with $\bbe X^2 <+\infty$, i.e., 
$\int_1^{+\infty} u^2\nu(du) < +\infty$,   
\eqref{repder} can be rewritten as 
\begin{align}\label{repder3}
Cov(X,f(X)) = \eta((0, +\infty)) \bbe f'(X+Y),   
\end{align}
where $\eta((0, +\infty)) = \int_0^{+\infty} u^2\nu(du) < +\infty$, and $Y$, with 
law $\eta/\eta((0, +\infty))$, is independent of $X$. In view of our previous corollary, it is a simple matter to modify the above arguments when  
the condition $X\ge 0$ is omitted.  The corresponding result is given by the following proposition, whose proof 
is briefly sketched and 
whose statement is, again, also related to the notion of zero-bias distribution (see \cite{GR97}).  
\end{rem}

%\textcolor{red}{
\begin{prop}\label{local} 
Let $X$ be a nondegenerate random variable such that $\bbe X^2<+\infty$. Let $b\in\bbr$, 
and let $\nu\ne 0$ be a L\'evy measure such that 
\begin{align}\label{HP:ExtZeroBias}
\int_{|u|>1} u^2\nu(du)<+\infty.
\end{align}
Then, 
\begin{align}\label{eq:ExtZeroBias}
Cov(X, f(X))=\left(\int_{-\infty}^{+\infty}u^2\nu(du)\right)\bbe f^\prime (X+Y),
\end{align}
for all bounded Lipschitz functions $f$, 
where the function $\eta$ is defined by
\begin{align*}
\eta(v):=\eta_+(v)\bbone_{(0,+\infty)}(v)+\eta_-(v)\bbone_{(-\infty,0)}(v), 
\end{align*}
with $\eta_+$ and $\eta_-$ respectively defined on $(0,+\infty)$ and on $(-\infty,0)$ via 
\begin{align*}
\eta_+(v)=\int_{v}^{+\infty}u\nu(du), \quad\quad \eta_-(v)=\int_{-\infty}^v (-u)\nu(du), 
\end{align*}
%and with moreover 
%\begin{align}
%\eta(\mathbb{R}):=\int_{\mathbb{R}}u^2\nu(du),
%\end{align}
and where the random variables $X$ and $Y$ are independent, with the law of $Y$ given by 
\begin{align*}
\mu_{Y}(du)=\dfrac{\eta(u)}{\int_{-\infty}^{+\infty}u^2\nu(du)}du,
%\mu_{Y^{-}}(du)=\dfrac{\eta_-(u)\bbone_{(-\infty,0)}(u)}{\eta_-((-\infty, 0))}du,
\end{align*}
if and only if 
$X\sim ID(b,0,\nu)$.
\end{prop}
%}

%\textcolor{red}{
\begin{proof}
Let us first sketch the proof of the direct implication. If $\mu$ denotes the law of $X$, then from Theorem \ref{thm3.1}, 
and our hypotheses, it follows that for any bounded Lipschitz function, 
\begin{align}
Cov(X,f(X))=&\bbe \int_{-\infty}^{+\infty}(f(X+u)-f(X))u\nu(du)\nonumber \\
%=&\int \int_\mathbb{R}(f(x+u)-f(x))u\nu(du)d\mu_X(dx)\\
=&\int_{-\infty}^{+\infty} \int_0^{+\infty}(f(x+u)-f(x))u\nu(du)\mu(dx)\nonumber\\
&\qquad +\int_{-\infty}^{+\infty}  \int_{-\infty}^{0}(f(x+u)-f(x))u\nu(du)\mu(dx)\nonumber\\
=&\int_{-\infty}^{+\infty}  \int_0^{+\infty}\left(\int_0^u f'(x+v)dv\right)u\nu(du)d\mu(dx)\nonumber\\
&\qquad +\int_{-\infty}^{+\infty}  \int_{-\infty}^{0}\left(\int_u^0 f'(x+v)dv\right)(-u)\nu(du)\mu(dx)\nonumber\\
=&\int_{-\infty}^{+\infty}  \int_0^{+\infty}f'(x+v)\eta_+(v)dv\mu(dx)\nonumber\\
&\qquad +\int_{-\infty}^{+\infty}  \int_{-\infty}^{0}f'(x+v)\eta_-(v)dvd\mu(dx)\nonumber\\
=&\int_{-\infty}^{+\infty}  \int_{-\infty}^{+\infty}  f'(x+v)\bigg(\eta_+(v)\bbone_{(0,+\infty)}(v)\nonumber\\
&\qquad \qquad \qquad \qquad \qquad \qquad +\eta_-(v)\bbone_{(-\infty,0)}(v)\bigg)dv\mu(dx)\nonumber\\
=&\int_{-\infty}^{+\infty} \int_{-\infty}^{+\infty}  f'(x+v) \eta(v)dv\mu(dx).  
\end{align}
The conclusion then easily follows from the very definition of $Y$ and the assumption \eqref{HP:ExtZeroBias}.  
The converse implication is a direct consequence of the converse part of Theorem \ref{thm3.1} or follows, as before, 
by taking $f(x)=e^{itx}, t\in \bbr$ in \eqref{eq:ExtZeroBias}.  
\end{proof}
%}

%\textcolor{red}{
\begin{rem}\label{rem:Zerobias}
%\begin{itemize}
%\item 
(i) The previous proposition can, in particular, be applied to the two sided exponential distribution with parameters $\alpha>0$ 
and $\beta>0$. In this case, the L\'evy measure is given by 
$\nu(du)/du=e^{-\alpha u}/u\bbone_{(0,+\infty)}(u)-e^{\beta u}/u\bbone_{(-\infty,0)}(u)$. 
Then, the condition \eqref{HP:ExtZeroBias} is readily satisfied,  and   
the law of $Y$ has the following density
\begin{align}
f_{Y}(t)=\frac{\alpha^2\beta^2}{\alpha^2+\beta^2}\bigg(\frac{e^{-\alpha t}}{\alpha}\bbone_{(0,+\infty)}(t)
+\frac{e^{\beta t}}{\beta}\bbone_{(-\infty,0)}(t)\bigg).
\end{align}
%\item If $X$ is a symmetric $\alpha$-stable distribution then, we have, with $\alpha\in (1,2)$:
%\begin{align}
%\int_{\{u>|t|\}}u\nu(du)=c\int_{u>|t|}\frac{du}{u^{\alpha}}=\frac{c}{\alpha-1}\frac{1}{|t|^{\alpha-1}}.
%\end{align}
%Thus, none of the conditions of Corollary \ref{cor:ZeroBias} are satisfied.
(ii) As done in Corollary \ref{cor:ExtSizeBias}, Proposition \ref{local} extends the notion of zero-bias distribution to all 
infinitely divisible nondegenerate distributions with finite variance. {\it The random variable $Y$ in \eqref{eq:ExtZeroBias} 
will be called the extended zero-bias distribution associated with $X$.}\\
(iii) Another possible writing for \eqref{eq:ExtZeroBias}, more in line with \eqref{eq:ExtSizeBias}, is
\begin{align}
Cov(X, f(X))=\eta_+\left((0,+\infty)\right)\bbe f^\prime (X+Y^+)
+\eta_-\left((-\infty,0)\right)\bbe f^\prime (X+Y^-),
\end{align}
where $Y^+$ and $Y^-$ have respective law
\begin{align*}
\mu_{Y^{+}}(du)&=\dfrac{\eta_{+}(u)\bbone_{(0,+\infty)}(u)}{\eta_+((0,+\infty))}du,\\
\mu_{Y^{-}}(du)&=\dfrac{\eta_{-}(u)\bbone_{(-\infty,0)}(u)}{\eta_-((-\infty, 0))}du,
\end{align*}
and where
\begin{align*}
\eta_+\left((0,+\infty)\right)=\int_0^{+\infty}u^2\nu(du),\quad \eta_-\left((-\infty,0)\right):=\int_{-\infty}^{0}u^2\nu(du).
\end{align*}
%\end{itemize}
\end{rem}
\noindent
It is important to note that the stable distributions with $\alpha\in (1,2)$ do not satisfy either the assumptions 
of Corollary \ref{cor:ExtSizeBias} or those of Proposition \ref{local}.  
Nevertheless, our next result which is a mixture of the two previous ones 
characterizes infinitely divisible distributions with finite first moment, and in particular, the stable ones. For this purpose, we introduce the following functions respectively well-defined on $(0,1)$ and on $(-1,0)$,
\begin{align*}
\eta_+(v)=\int_{v}^{1}u\nu(du), \quad\quad \eta_-(v)=\int_{-1}^v (-u)\nu(du).
\end{align*}
Note that since $\nu$ is a L\'evy measure, for all $v\in (0,1)$, 
\begin{align*}
v\eta_+(v)\leq \int_{v}^{1}u^2\nu(du)\leq \int_{0}^{1}u^2\nu(du)<+\infty,
\end{align*}
and similarly for $ \eta_-$.

\begin{prop}\label{prop:final}
Let $X$ be a nondegenerate random variable such that $\bbe |X|<+\infty$. Let $b\in\bbr$
and let $\nu$ be a L\'evy measure such that
\begin{align}\label{HP:finite1}
0 \ne\int_{|u|>1}|u|\nu(du)<+\infty,\quad \int_{-1}^{1}u^2\nu(du)\ne 0.
\end{align}
Then, 
\begin{align}\label{eq:ExtZeroSizeBias}
Cov(X, f(X))=&\left(\int_{-1}^{1}u^2\nu(du)\right)\bbe f'(X+U)+m\bbe f(X+V_+)-m\bbe f(X+V_-), 
\end{align}
for all bounded Lipschitz functions $f$, where $m_{\pm}$ and $m$ are defined by
\begin{align*}
m_+=\int_{1}^{+\infty}u\nu(du),\quad m_-=\int_{-\infty}^{-1}(-u)\nu(du),\quad m=m_++m_-,
\end{align*}
and where the random variables $X$, $U$, $V_+$ and $V_-$ are independent, with the laws of $U$, $V_+$ and $V_-$ 
respectively given by 
\begin{align*}
&\mu_{U}(du)=\dfrac{\eta_+(u)\bbone_{(0,1)}(u)+\eta_-(u)\bbone_{(-1,0)}(u)}{\int_{-1}^{+1}u^2\nu(du)}du,\\
&\mu_{V_+}(du)=\frac{m_-}{m}\delta_0(u)+\frac{u}{m}\bbone_{(1,+\infty)}(u)\nu(du),\\
&\mu_{V_-}(du)=\frac{m_+}{m}\delta_0(u)+\frac{-u}{m}\bbone_{(-\infty,-1)}(u)\nu(du), 
\end{align*}
if and only if 
$X\sim ID(b,0,\nu)$.
\end{prop}
\begin{proof}
First, let $X\sim ID(b,0,\nu)$, and denote its law by $\mu$.  
Then, from Theorem \ref{thm3.1}, for any bounded Lipschitz function,
\begin{align*}
Cov(X,f(X))&= \bbe \int_{-\infty}^{+\infty}(f(X+u)-f(X))u\nu(du)\nonumber\\
&=\bbe \int_{|u|\leq 1}(f(X+u)-f(X))u\nu(du)+\bbe \int_{|u|> 1}(f(X+u)-f(X))u\nu(du).  
\end{align*}
Now, performing steps similar to thoses of Proposition \ref{local} and Corollary \ref{cor:ExtSizeBias} 
for,  respectively, the first and second terms of the previous sum. For the first one,
\begin{align*}
\bbe \int_{|u|\leq 1}(f(X+u)-f(X))u\nu(du)&=\bbe \int_{0}^1(f(X+u)-f(X))u\nu(du)+\bbe \int_{-1}^0(f(X+u)-f(X))u\nu(du)\nonumber\\
&=\int_{-\infty}^{+\infty}  \int_0^{1}\left(\int_0^u f'(x+v)dv\right)u\nu(du)d\mu(dx)\nonumber\\
&\qquad +\int_{-\infty}^{+\infty}  \int_{-1}^{0}\left(\int_u^0 f'(x+v)dv\right)(-u)\nu(du)\mu(dx)\nonumber\\
&=\int_{-\infty}^{+\infty}  \int_0^{1}f'(x+v)\eta_+(v)dv\mu(dx)\nonumber\\
&\qquad +\int_{-\infty}^{+\infty}  \int_{-1}^{0}f'(x+v)\eta_-(v)dvd\mu(dx)\nonumber\\
&=\int_{-\infty}^{+\infty}  \int_{-\infty}^{+\infty}  f'(x+v)\bigg(\eta_+(v)\bbone_{(0,1)}(v)\nonumber\\
&\qquad \qquad \qquad \qquad \qquad  \qquad \qquad +\eta_-(v)\bbone_{(-1,0)}(v)\bigg)dv\mu(dx)\nonumber\\
&=\int_{-\infty}^{+\infty} \int_{-\infty}^{+\infty}  f'(x+v) \eta(v)dv\mu(dx).  
\end{align*}
For the second term,
\begin{align*}
\bbe \int_{|u|> 1}(f(X+u)-f(X))u\nu(du)&=\bbe \int_{|u|> 1}f(X+u)u\nu(du)-\bbe f(X) \int_{|u|> 1}u\nu(du)\nonumber\\
&=\bbe \int_{1}^{+\infty}f(X+u)u\nu(du)-\bbe \int_{-\infty}^{-1}f(X+u)(-u)\nu(du)\nonumber\\
&\qquad\qquad \qquad \qquad -\bbe f(X) \int_{|u|> 1}u\nu(du)\nonumber\\
&=m\left(\bbe f(X+V_+)-\bbe f(X+V_-)\right).  
\end{align*}
The conclusion then easily follows from the very definition of $U$, $V_+$ and $V_-$ and the assumption \eqref{HP:finite1}.  
The converse implication is a direct consequence of the converse part of Theorem \ref{thm3.1} or follows, as before, 
by taking $f(\cdot)=e^{it\cdot}, t\in \bbr$ in \eqref{eq:ExtZeroSizeBias}.  
\end{proof}

\section{General Upper Bounds by Fourier Methods}
The Fourier methodology developed 
in \cite{AMPS17} to study the Stein-Tikhomirov method (see \cite{Tk81}) provides rates of convergence in Kolmogorov or in smooth-Wasserstein distance 
for sequences $(X_n)_{n\ge 1}$ converging  
towards $X_\infty$.   This approach  leads to 
quantitative estimates when $X_\infty$ is either a second order Wiener chaos, 
or the Dickman distribution or the symmetric $\alpha$-stable one.  
Corollary \ref{cor:ExtSizeBias},  or Proposition \ref{local} or, even, the stable characterizing identities of the previous section allow extensions of the aforementioned estimates to classes of infinitely divisible sequences. The forthcoming results are general and intersect some of those on the Dickman distribution presented in \cite{AMPS17}.
Before stating them, recall that the smooth Wasserstein distance, $d_{{W}_r}$, for $r\geq 0$, is given by
\begin{align}
d_{{W}_r}(X,Y)=\underset{h\in \mathcal{H}_r}{\sup}|\bbe h(X)-\bbe h(Y)|,
\end{align}
where $ \mathcal{H}_r$ is the set of continuous functions which are $r$-times continuously differentiable 
such that $\|h^{(k)}\|_\infty\leq 1$, for all $0\leq k\leq r$, where $h^{(0)}=h$ and 
$h^{(k)}$ is the $k$-th derivative of $h$, while very classicaly, the Kolmogorov distance is
\begin{align*}
d_K(X,Y):=\underset{x\in\mathbb{R}}{\sup}\, |\mathbb{P}(X\leq x)-\mathbb{P}(Y\leq x)|.  
\end{align*}
Note that, by an approximation argument, $d_K(X,Y)\leq d_{{W}_0}(X,Y)$, and that by another approximation argument (see e.g. Appendix $A$ of \cite{AMPS17} or Lemma \ref{lem: repsmwa	s} of the Appendix), the smooth Wasserstein distance $d_{{W}_r}$, for $r\geq 1$ admits the following representation
\begin{align}\label{eq:repcinc}
d_{{W}_r}(X,Y)=\underset{h\in C^{\infty}_c(\mathbb{R})\cap  \mathcal{H}_r}{\sup}|\bbe\, h(X)-\bbe\, h(Y)|,
\end{align}
where $C^{\infty}_c(\mathbb{R})$ denotes the space of compactly supported, infinitely differentiable functions on $\mathbb{R}$. Moreover, the smooth Wasserstein distances and the classical Wasserstein distances are ordered in the following way
\begin{align}
d_{W_r}(X,Y) \leq d_{W_1}(X,Y)\leq W_1(X,Y)\leq W_p(X,Y)
\end{align}
for all $r\geq 1$ and where, for all $p\geq 1$ and for all random variables $X,Y$ with finite absolute $p$-th moment
\begin{align}\label{eq:wasserp}
W^p_p(X,Y)= \inf\bbe |X-Y|^p
\end{align}
with the infimum taken over the set of probability measures on $\mathbb{R}\times \mathbb{R}$ with marginals given by the law of $X$ and the law of $Y$. Recall also that convergence in Wasserstein-$p$ distance is equivalent to convergence in law and convergence of the $p$-th absolute moments (see e.g. \cite[Theorem 6.9]{Vbook2}).
\begin{thm}\label{thm4.1}
Let $X_n \sim ID(b_n,0,\nu_n)$, $n\geq 1$, be a sequence of nondegenerate random variables converging 
in law towards $X_\infty \sim ID(b_\infty,0,\nu_\infty)$ (nondegenerate), with also $\bbe |X_n|<+\infty$, $\bbe |X_\infty|<+\infty$ and
\begin{align}\label{HP:Reg0}
&\int_{-1}^{+1}|u|\nu_n(du)<+\infty, \quad\quad \int_{-1}^{+1}|u|\nu_\infty(du)<+\infty,
\end{align}
$n\geq 1$. Further, for all $t\in\mathbb{R}$, let
\begin{align}\label{HP:PolynBound}
|\varphi_\infty(t)|\int_0^{|t|}\dfrac{ds}{|\varphi_\infty(s)|} \leq C_\infty |t|^{p_\infty},
\end{align}
where $\varphi_\infty$ is the characteristic function of $X_\infty$ and where $C_\infty>0$, $p_\infty\geq 1$.  
Let the law of $X_\infty$ be absolutely continuous with respect to the Lebesgue measure and 
have a bounded density. Then,
\begin{align*}
d_K(X_n,X_\infty)\leq C'_\infty \Delta_n^{\frac{1}{p_\infty+2}},
\end{align*}
where
\begin{align*}
\Delta_n=& |(m^n_0)^+-(m^{\infty}_0)^+|+|(m^n_0)^--(m^{\infty}_0)^-| \\
&\qquad \qquad \qquad +(m^\infty_0)^+\bbe |Y_n^+-Y_\infty^+|+(m^\infty_0)^-\bbe |Y_n^--Y_\infty^-|.
\end{align*}
where $(m^n_0)^{\pm},Y_n^{\pm}$ and 
$(m^{\infty}_0)^{\pm},Y_\infty^{\pm}$ are the quantities defined in 
Corollary \ref{cor:ExtSizeBias} respectively associated with $X_n$ and $X_\infty$.
\end{thm}

\begin{proof} 
%We start with the proof of (i). 
From Corollary \ref{cor:ExtSizeBias} applied to $X_n$ and $X_\infty$, let
\begin{align*}
&\Delta_n^{\pm}(t):=(m^n_0)^{\pm}\varphi_{Y^{\pm}_n}(t)-(m^{\infty}_0)^{\pm}\varphi_{Y^{\pm}_\infty}(t),\\
%&\Delta_n^-(t):=(m^n_0)^-\varphi_{Y^-_n}(t)-(m^{\infty}_0)^-\varphi_{Y^-_\infty}(t)\\
&S_\infty(t):=(m_0^{\infty})^+\varphi_{Y^+_\infty}(t)-(m_0^{\infty})^-\varphi_{Y^-_\infty}(t),\\
&\varepsilon_n(t):=\varphi_n(t)-\varphi_\infty(t)
\end{align*}
where $\varphi_{Y^{\pm}_n}$ and $\varphi_{Y^{\pm}_\infty}$ are the characteristic functions of $Y^{\pm}_n$ and $Y^{\pm}_\infty$. Now, thanks to the identity (\ref{eq:ExtSizeBias}) applied to the test functions $f(\cdot) =e^{i t .}$,
\begin{align*}
&\frac{1}{i}\dfrac{d}{dt}\big(\varphi_\infty(t)\big)=\varphi_\infty(t)S_\infty(t),\\
&\frac{1}{i}\dfrac{d}{dt}\big(\varphi_n(t)\big)=\varphi_n(t)S_\infty(t)+\varphi_n(t)\Delta_n^+(t)-\varphi_n(t)\Delta_n^-(t).
\end{align*}
Subtracting these last two expressions gives, recalling also that the characteristic function of an ID law never vanishes,
\begin{align*}
\dfrac{d}{dt}\big(\varepsilon_n(t)\big)=\frac{\varepsilon_n(t)}{\varphi_\infty(t)}\frac{d}{dt}(\varphi_\infty(t))+i\varphi_n(t)(\Delta_n^+(t)-\Delta_n^-(t)),
\end{align*}
since,
\begin{align*}
S_\infty(t)=\frac{1}{i\varphi_\infty(t)}\frac{d}{dt}(\varphi_\infty(t)).
\end{align*}
Then, straightforward computations imply that for all $t\geq 0$:
\begin{align*}
\varepsilon_n(t)=i\varphi_\infty(t)\int_0^t \frac{\varphi_n(s)}{\varphi_\infty(s)}\big((\Delta_n^+(s)-\Delta_n^-(s))\big)ds,
\end{align*}
and similarly for $t\leq 0$. Let us next bound the difference, $\Delta_n^+(s)-\Delta_n^-(s)$. First,
\begin{align*}
|\Delta_n^+(s)-\Delta_n^-(s)|\leq I+II+III+IV,
\end{align*}
where
\begin{align*}
&I:=|(m^n_0)^+-(m^{\infty}_0)^+|,\\
&II:=|(m^n_0)^--(m^{\infty}_0)^-|,\\
&III:=(m^{\infty}_0)^+ |\varphi_{Y^+_n}(s)-\varphi_{Y^+_\infty}(s)| \leq (m^{\infty}_0)^+ |s| \bbe |Y^+_n-Y^+_{\infty}|,\\
&IV:=(m^{\infty}_0)^- |\varphi_{Y^-_n}(s)-\varphi_{Y^-_\infty}(s)|\leq (m^{\infty}_0)^- |s| \bbe |Y^-_n-Y^-_{\infty}|.
\end{align*}
Hence,
\begin{align*}
|\varepsilon_n(t)|&\leq |\varphi_\infty(t)|\int_0^t \dfrac{ds}{|\varphi_\infty(s)|}
\left(|(m^n_0)^+-(m^{\infty}_0)^+|+|(m^n_0)^--(m^{\infty}_0)^-|\right)\nonumber \\
&\qquad +|\varphi_\infty(t)|\int_0^t \dfrac{|s| ds}{|\varphi_\infty(s)|}\left((m^{\infty}_0)^+ \bbe |Y^+_n-Y^+_{\infty}|
+ (m^{\infty}_0)^- \bbe |Y^-_n-Y^-_{\infty}|\right).
\end{align*}
Then, using \eqref{HP:PolynBound}, together with the definition of $\Delta_n$, leads to:
\begin{align}\label{eq:epsilon}
|\varepsilon_n(t)|&\leq C_\infty (|t|^{p_\infty}+|t|^{p_\infty+1})\Delta_n.
\end{align}
Since the law of $X_\infty$ has a bounded density applying the classical Esseen inequality 
(see e.g. \cite[Theorem 5.1]{Pet95}) gives, for all $T>0$,
\begin{align}\label{ineq:Ess}
d_K(X_n,X_\infty)\leq C_1\int_{-T}^T \dfrac{|\varepsilon_n(t)|}{|t|}dt+C_2\dfrac{\|h\|_\infty}{T},
\end{align}
where $C_1$ and $C_2$ are positive (absolute) constants, while $\|h\|_\infty$ is the essential supremum of the 
density $h$ of the law of $X_\infty$.  Next, plugging \eqref{eq:epsilon} into \eqref{ineq:Ess},
\begin{align*}
d_K(X_n,X_\infty) \leq C'_1\bigg(T^{p_\infty}+T^{p_\infty+1}\bigg)\Delta_n+\frac{C_2^\prime}{T}.
\end{align*}
The choice $T=(1/\Delta_n)^{\frac{1}{p_\infty+2}}$ concludes the proof.
\end{proof}

\begin{rem}

(i) Let us briefly discuss the growth condition on the limiting characteristic function $\varphi$, 
namely, the requirement that for all $t\in \mathbb{R}$,
\begin{align}\label{ineq:PolyBouPhi}
L(\varphi)(t):=|\varphi(t)|\int_0^{|t|}\frac{ds}{|\varphi(s)|}\leq C|t|^p,
\end{align}
for some $C>0$ and $p\geq 1$. When the limiting law is the standard normal one, 
the functional $L(\varphi)$ is the Dawson integral associated with the normal distribution. 
It decreases to $0$ at infinity, and for $t\in \mathbb{R}$,
\begin{align}
L(\varphi)(t):=e^{\frac{-t^2}{2}}\int_0^{|t|}e^{\frac{s^2}{2}}ds\leq \dfrac{2|t|}{1+t^2}.
\end{align}
%It is called the generalized Dawson function in \cite{AMPS17} for general limiting law. 
Different behaviors are possible for this (generalized Dawson) functional, see \cite{AMPS17}. 
As detailed below, in a general gamma setting, \eqref{ineq:PolyBouPhi} holds true with $p=1$ while in the stable case (see Lemma $10$ of Appendix $B$ 
in \cite{AMPS17}), for $1<\alpha<2$, and $t>0$,
\begin{align}\label{GenDawStab}
L(\varphi)(t)\leq \left(t^{1-\alpha}/c+C e^{-ct^\alpha}\right),
\end{align}
where $C=\int_{0}^{c^{-{1/\alpha}}}\!\!e^{c s^\alpha}ds$ and $c=c_1+c_2$. In particular, $C\leq e/c^{{1/\alpha}}$.
Moreover, for $t$ small, \eqref{GenDawStab} can be replaced by,
\begin{align*}
L(\varphi)(t)\leq |t|.
\end{align*}
Then, for some constant $C'>0$ only depending on $\alpha$ and $c$, and for all $t\in\mathbb{R}$,
\begin{align}\label{ineq:FineBoundStable}
L(\varphi)(t)\leq C'\dfrac{|t|}{1+|t|^{\alpha}}.
\end{align}

For the Dickman distribution as considered in \cite{AMPS17}, a linear growth can also be obtained from 
the corresponding characteristic function.\\
\\
(ii) As well known,  
\begin{align}\label{ineq:KW1}
d_K(X_n, X_\infty) \le \sqrt{2\|h\|_\infty W_1 (X_n, X_\infty)},
\end{align}
where again $\|h\|_\infty$ is the supremum norm of $h$, the bounded density of the law of $X_\infty$, and where $W_1$ 
denotes the 1-Wasserstein distance as in \eqref{eq:wasserp} which also admits the following representation
$$W_1(X,Y)=\break\underset{h\in Lip(1)}{\sup}|\bbe h(X) -\bbe h(Y)|,$$
for $X,Y$ random variables with finite first moment. Therefore, to go beyond the bounded density case, e.g., to consider 
discrete limiting laws, it is natural to 
explore convergence rates in (smooth) Wasserstein.  Under uniform (exponential) integrability, such issues 
can be tackled.  For example,  instead of the bounded density assumption, let, for some $\lambda>0$ and $\alpha\in (0,1]$,
\begin{align}\label{eq:TailSeq}
\underset{n\geq 1}{\sup}\, \bbe\, e^{\lambda |X_n|^\alpha}<+\infty, \quad i.e., \quad 
\underset{n\geq 1}{\sup}\, \int_{|u|>1} e^{\lambda |u|^\alpha}\nu_n(du)<+\infty,
%\quad\int_{|u|>1} e^{\lambda_\infty |u|^\alpha}\nu_\infty(du)<+\infty.
\end{align}
then,
\begin{align}\label{convW}
d_{{W}_{p_\infty+2}}(X_n,X_\infty) \leq C_{\infty}' \Delta_n |\ln \Delta_n|^{\frac{1}{2\alpha}},
\end{align} 
where $\Delta_n$ is as in the previous theorem.  
The proof of \eqref{convW} uses the pointwise estimate \eqref{eq:epsilon} 
%\begin{align}
%|\varepsilon_n(t)|&\leq C_\infty (|t|^{p_\infty}+|t|^{p_\infty+1})\Delta_n,
%\end{align}
combined with the assumption $\eqref{eq:TailSeq}$ and with the statement and conclusion of \cite[Theorem 1]{AMPS17}.  
\end{rem}
\noindent
Proposition \ref{local} also provides quantitative upper bounds on the Kolmogorov 
distance. This is the content of the next proposition whose statement is similar to that of Theorem \ref{thm4.1}.

\begin{prop}\label{thm4.2}
Let $X_n\sim ID(b_n, 0,\nu_n)$, $n\geq 1$, be a sequence of nondegenerate random variables converging in law towards 
$X_\infty\sim ID(b_\infty, 0,\nu_\infty)$ (nondegenerate) and such that $\bbe |X_n|^2<+\infty$, $\bbe |X_\infty|^2<+\infty$, $n\geq 1$.  
Let also, for all $t\in\mathbb{R}$,
\begin{align}\label{HP:thm4.2}
|\varphi_\infty(t)|\int_0^{|t|} \dfrac{ds}{|\varphi_\infty(s)|}\leq C_\infty |t|^{p_\infty},
\end{align}
where $\varphi_\infty$ is the characteristic function of $X_\infty$ and where $C_\infty>0$, $p_\infty\geq 1$.  
Let the law of $X_\infty$ be absolutely continuous with respect to the Lebesgue measure and have a bounded density, then
\begin{align}
d_K(X_n,X_\infty)\leq C'_\infty \Delta_n^{\frac{1}{p_\infty+3}},
\end{align}
where
\begin{align}
\Delta_n=|\eta_n-\eta_\infty|+|\bbe X_n-\bbe X_\infty|+\bbe |Y_n-Y_\infty|,
\end{align}
where $Y_n$ and 
$Y_\infty$ are the random variables defined in
Proposition \ref{local} respectively associated with $X_n$ and $X_\infty$ respectively and where,
\begin{align*}
\eta_n:=\int_{-\infty}^{+\infty} u^2\nu_n(du),\quad \quad \eta_\infty:=\int_{-\infty}^{+\infty} u^2\nu_\infty(du).
\end{align*}
\end{prop}

\begin{proof}
The proof of this proposition is very similar to the proof of Theorem \ref{thm4.1} and so it is only sketched. 
From Proposition~\ref{local} applied to $X_n$ and $X_\infty$, let:
\begin{align*}
m_n&:=\bbe X_n,\quad \quad m_\infty:=\bbe X_\infty,\\
\Delta_n(t)&:=t(\eta_\infty\varphi_{Y_\infty}(t)-\eta_n\varphi_{Y_n}(t))+i(m_n-m_\infty),\\
R_\infty(t)&:=-t\eta_\infty\varphi_{Y_\infty}(t)+im_\infty,\\
\varepsilon(t)&:=\varphi_n(t)-\varphi_\infty(t).
\end{align*}
Next, thanks to the identity \eqref{eq:ExtZeroBias} applied to the test functions $f(\cdot)=e^{it.}$,
\begin{align*}
\dfrac{d}{dt}\big(\varphi_n(t)\big)&=R_\infty(t)\varphi_n(t)+\varphi_n(t)\Delta_n(t),\\
\dfrac{d}{dt}\big(\varphi_\infty(t)\big)&=R_\infty(t)\varphi_\infty(t).
\end{align*}
Subtracting the last two expressions, 
\begin{align*}
\dfrac{d}{dt}\big(\varepsilon_n(t)\big)=\dfrac{\varepsilon(t)}{\varphi_\infty(t)}\dfrac{d}{dt}(\varphi_\infty(t))+\varphi_n(t)\Delta_n(t).
\end{align*}
Then, straightforward computations imply that for all $t\geq 0$,
\begin{align*}
\varepsilon_n(t)=\varphi_\infty(t)\int_0^t \frac{\varphi_n(s)}{\varphi_\infty(s)}\Delta_n(s)ds,
\end{align*}
and similarly for $t\leq 0$. Hence, using \eqref{HP:thm4.2} together with standard computations,
\begin{align*}
|\varepsilon_n(t)|\leq C_\infty(|t|^{p_\infty}+|t|^{p_\infty+1}+|t|^{p_\infty+2})\Delta_n.  
\end{align*}
Finally, proceeding as in the end of the proof of Theorem~\ref{thm4.1} concludes the proof. 

\end{proof}
\noindent
\begin{rem}\label{rem:Comparison}
(i) Since the random variables $(Y_n^{\pm},Y_\infty^{\pm})$ (resp.~$(Y_n,Y_\infty)$) in Theorem~\ref{thm4.1} 
(resp.~Proposition \ref{thm4.2}) are independent of $(X_n,X_\infty)$, one can choose any of their couplings. 
In particular, the definition of $\Delta_n$ of Theorem \ref{thm4.1} can be replaced by
\begin{align}\label{eq:DeltaVarW1+}
\Delta_n=&\, |(m^n_0)^+-(m^{\infty}_0)^+|+|(m^n_0)^--(m^{\infty}_0)^-|\nonumber \\
&\qquad +(m^\infty_0)^+W_1(Y_n^+,Y_\infty^+)+(m^\infty_0)^-W_1(Y_n^-,Y_\infty^-).
\end{align}
Similarly the quantity $\Delta_n$ of  Proposition \ref{thm4.2} can be replaced by, 
\begin{align}\label{eq:DeltaVarW1}
\Delta_n=|\eta_n-\eta_\infty|+|\bbe X_n-\bbe X_\infty|+W_1(Y_n,Y_\infty).
\end{align}

(ii) Recall that the Wasserstein-$1$ distance between two random variables $X$ and $\tilde{X}$ 
both having finite first moment,  and respective law $\mu$ and $\tilde{\mu}$ can also be represented as,
\begin{align}\label{eq:repW1}
W_1(X,\tilde{X})=\int_{-\infty}^{+\infty}|F_\mu(t)-F_{\tilde{\mu}}(t)|dt,
\end{align}
where $F_\mu$ and $F_{\tilde{\mu}}$ are the respective cumulative distribution functions of $\mu$ and $\tilde{\mu}$. 
Combining the above with Proposition \ref{local}, \eqref{eq:DeltaVarW1} becomes,
\begin{align}\label{eq:DeltaVarW12}
\Delta_n=&\, |\eta_n-\eta_\infty|+|\bbe X_n-\bbe X_\infty|+\int_{-\infty}^0 \left|\int_{-\infty}^t (-v)(t-v)\left(\frac{\nu_n(dv)}{\eta_n}
-\frac{\nu_\infty(dv)}{\eta_\infty}\right)\right|dt\nonumber\\
&+\int_0^{+\infty}\bigg|\int_0^{+\infty}v (v\wedge t)\left(\frac{\nu_n(dv)}{\eta_n}-\frac{\nu_\infty(dv)}{\eta_\infty}\right) +\int_{-\infty}^0 v^2\left(\frac{\nu_n(dv)}{\eta_n}-\frac{\nu_\infty(dv)}{\eta_\infty}\right)\bigg|dt.
\end{align}

(iii) Next, for the second order chaoses $X_n=\sum_{k=1}^{+\infty}\lambda_{n,k}(Z_k^2-1)/2$, $n\geq 1$ and $X_\infty=\sum_{k=1}^{+\infty}\lambda_{\infty,k}(Z_k^2-1)/2$, with $\lambda_{n,k}>0$ and $\lambda_{\infty,k}>0$, for all $k\geq 1$, $\Delta_n$ in \eqref{eq:DeltaVarW12} becomes
\begin{align}
\Delta_n=2\int_0^{+\infty}\left|\sum_{k=1}^{+\infty}(\lambda_{\infty,k}^2 e^{-\frac{t}{2\lambda_{\infty,k}}}-\lambda_{n,k}^2 e^{-\frac{t}{2\lambda_{n,k}}})\right|dt.
\end{align}
Similar computations can be done using \eqref{eq:DeltaVarW1+} and \eqref{eq:repW1}.\\
\\
(iv) Again, the Kolmogorov distance can be replaced by a smooth Wasserstein one.   
(Replacing also the bounded density assumption.)  Indeed, if for some $\lambda>0$ and $\alpha\in (0,1]$,
%\begin{align}\label{eq:TailSeq2}
$\underset{n\geq 1}{\sup}\, \bbe\, e^{\lambda |X_n|^\alpha}<+\infty$, 
%\quad \bbe\, e^{\lambda_\infty |X_\infty|^\alpha}<+\infty
%\end{align}
then 
\begin{align*}
d_{{W}_{p_\infty+3}}(X_n,X_\infty) \leq C_{\infty}' \Delta_n |\ln \Delta_n|^{\frac{1}{2\alpha}}, 
\end{align*}
as easily seen by simple modifications of the techniques presented above.  

(v) Any sequence of infinitely divisible distributions converging in law has a limiting distribution 
which is itself infinitely divisible, e.g., \cite[Lemma 7.8]{S}. It is thus natural to ask for 
conditions for such convergence 
as well as for quantitative versions of it.  In this regard, \cite[Theorem 8.7]{S} provides
necessary and sufficient conditions ensuring the 
weak convergence of sequences of infinitely divisible distributions.  
Namely, it requires that, as $n\to +\infty$, 
\begin{align*}
\beta_n=b_n+\int_{-\infty}^{+\infty}u\left(c(u)-\bbone_{|u|\leq 1}\right)\nu_n(du)\longrightarrow \beta_\infty,
\end{align*}
and that,
\begin{align*}
\qquad \qquad u^2c(u)d\nu_n \Longrightarrow u^2c(u)d\nu_\infty,
\end{align*}
for some bounded continuous function $c$ from $\mathbb{R}$ to $\mathbb{R}$ such that $c(u)=1+o(|u|)$ 
,as $|u|\rightarrow 0$, and $c(u)= O(1/|u|)$ as $|u|\rightarrow+\infty$. Therefore, Theorem \ref{thm4.1} 
and Proposition \ref{thm4.2} provide quantitative versions of these results. 
%\\}
\end{rem}

\noindent
The previous results do not encompass the case of the stable distributions since neither \eqref{HP:Reg0} nor $\bbe |X_\infty|^2<+\infty$ are satisfied. To obtain 
quantitative convergence results toward more 
general ID distributions, let us present a result for some classes of self-decomposable laws. {\it Below and elsewhere, we follow \cite{S}, and use the terminology increasing or decreasing in a non-strict manner.} 
\begin{prop}\label{propSD}
Let $X_n\sim ID(b_n, 0,\nu_n)$, $n\geq 1$, be a sequence of nondegenerate random variables converging in law towards 
$X_\infty\sim ID(b_\infty, 0,\nu_\infty)$ (nondegenerate) such that $\bbe |X_n|<+\infty$, $n\geq 1$, $\bbe |X_\infty|<+\infty$, and,
\begin{align*}
&\nu_n(du):= \frac{\psi_{1,n}(u)}{u}\bbone_{(0,+\infty)}(u)du+\frac{\psi_{2,n}(-u)}{(-u)}\bbone_{(-\infty,0)}(u)du, \\
&\nu_\infty(du):= \frac{\psi_{1,\infty}(u)}{u}\bbone_{(0,+\infty)}(u)du+\frac{\psi_{2,\infty}(-u)}{(-u)}\bbone_{(-\infty,0)}(u)du,
\end{align*}
where $\psi_{1,n},\psi_{2,n},\psi_{1,\infty}$ and $\psi_{2,\infty}$ are decreasing functions 
on $(0,+\infty)$.
Let also, for all $t\in\mathbb{R}$,
\begin{align}\label{HP:thm4.3}
|\varphi_\infty(t)|\int_0^{|t|} \dfrac{ds}{|\varphi_\infty(s)|}\leq C_\infty |t|^{p_\infty},
\end{align}
where $\varphi_\infty$ is the characteristic function of $X_\infty$ and where $C_\infty>0$, $p_\infty\geq 1$.  
Finally, let the law of $X_\infty$ be absolutely continuous with respect to the Lebesgue measure and have a bounded density. Then,
\begin{align*}
d_K(X_n,X_\infty)\leq C'_\infty (\Delta_n)^{\frac{1}{p_\infty+2}}, 
\end{align*}
where
\begin{align*}
\Delta_n=&|\bbe X_n-\bbe X_\infty|+\int_0^{1}|u||\psi_{1,n}(u)-\psi_{1,\infty}(u)|du+\int_{1}^{+\infty} |\psi_{1,n}(u)-\psi_{1,\infty}(u)|du\\
&\quad+\int_{-1}^0|u||\psi_{2,n}(-u)-\psi_{2,\infty}(-u)|du+\int_{-\infty}^{-1} |\psi_{2,n}(-u)-\psi_{2,\infty}(-u)|du.
\end{align*}
\end{prop}
\begin{proof}
Again, this proof  is very similar to the proof of Theorem \ref{thm4.1} and so it is only sketched. Let,
\begin{align*}
&m_n:=\bbe X_n,\quad\quad m_\infty:=\bbe X_\infty,\\
&\Delta_n(t):=m_n-m_\infty+\int_{-\infty}^{+\infty}(e^{itu}-1)u\nu_n(du)-\int_{-\infty}^{+\infty}(e^{itu}-1)u\nu_\infty(du),\\
&S_\infty(t):=m_\infty+\int_{-\infty}^{+\infty}(e^{itu}-1)u\nu_\infty(du),\\
&\varepsilon_n(t):=\varphi_n(t)-\varphi_\infty(t).
\end{align*}
Applying the identity \eqref{eq3.7} to $X_n$ and $X_\infty$ with $f(\cdot)=e^{i t.}$ gives
\begin{align*}
&\dfrac{d}{dt}(\varphi_n(t))=i\Delta_n(t)\varphi_n(t)+iS_\infty(t)\varphi_n(t),\\
&\dfrac{d}{dt}(\varphi_\infty(t))=iS_\infty(t)\varphi_\infty(t), 
\end{align*}
and thus,  
\begin{align*}
\dfrac{d}{dt}(\varepsilon_n(t))=\frac{\varphi'_\infty(t)}{\varphi_\infty(t)}\varepsilon_n(t)+i\Delta_n(t)\varphi_n(t).  
\end{align*}
Therefore, for all $t\geq 0$,
\begin{align*}
\varepsilon_n(t)=i\varphi_\infty(t)\int_0^t \frac{\varphi_n(s)}{\varphi_\infty(s)}\Delta_n(s)ds,
\end{align*}
and similarly for $t\leq 0$.  Let us now bound the quantity $\Delta_n(\cdot)$.
\begin{align*}
|\Delta_n(s)|&\leq |m_n-m_\infty|+\left|\int_{-\infty}^{+\infty}(e^{isu}-1)u\nu_n(du)-\int_{-\infty}^{+\infty}(e^{isu}-1)u\nu_\infty(du)\right|\\
&\leq  2(1+|s|)\bigg(|m_n-m_\infty|++\int_0^{1}|u||\psi_{1,n}(u)-\psi_{1,\infty}(u)|du+\int_{1}^{+\infty} |\psi_{1,n}(u)-\psi_{1,\infty}(u)|du\\
&\quad+\int_{-1}^0|u||\psi_{2,n}(-u)-\psi_{2,\infty}(-u)|du+\int_{-\infty}^{-1} |\psi_{2,n}(-u)-\psi_{2,\infty}(-u)|du\bigg),\\
&\leq 2 (1+|s|) \Delta_n.
\end{align*}
This implies,
\begin{align*}
|\varepsilon_n(t)| \leq C_\infty (|t|^{p_\infty}+|t|^{p_\infty+1})\Delta_n.
\end{align*}
To conclude the proof, proceed as in the end of the proof of Theorem \ref{thm4.1}.
\end{proof}
\noindent
\begin{rem}\label{paretoinfdiv}
Recalling \eqref{ineq:FineBoundStable}, note that the stable distributions do satisfy the assumptions of Proposition \ref{propSD}.  
However, the very specific properties  of their L\'evy measure entail detailled computations in order to reach a precise rate 
of convergence. To illustrate how our methodology can be adapted to obtain a rate of convergence towards 
a stable law, 
we present an example pertaining to the domain of normal attraction of the symmetric 
$\alpha$-stable distribution.  
Let $1<\alpha<2$ and set $c=(1-\alpha)/(2\Gamma(2-\alpha)\cos(\alpha\pi/2))$ and 
$\lambda=(2c)^{{1}/{\alpha}}$.  Then, 
denote by $f_1(x):=\frac{\alpha}{2\lambda}(1+\frac{|x|}{\lambda})^{-\alpha-1}$ the density of the 
Pareto law with parameters $\alpha>0$ and $\lambda>0$.  As well known, this random variable 
is infinitely divisible, see \cite[Chapter IV, Example 11.6]{SV03}, and belongs to the domain of normal attraction of the symmetric $\alpha$-stable distribution, \cite{Pet95}.   Our version of the Pareto density differs from the one considered in \cite{CW93,KK00,X17} 
given by $f(x):=\alpha\lambda^{\alpha}/(2|x|^{\alpha+1})\bbone_{|x|>\lambda}$,  
which is not infinitely divisible.  Indeed, since it is a symmetric distribution, its characteristic function is real-valued, 
and by standard computations,
\begin{align*}
\varphi(s)=1-s^\alpha+\alpha\int_0^1 \dfrac{1-\cos(\lambda y s)}{y^{\alpha+1}}dy \le 1-s^\alpha+\frac{\alpha s^2\lambda^2}{2(2-\alpha)}, 
\end{align*}
$s>0$.  
%Bounding $\alpha\int_0^1 \dfrac{1-\cos(\lambda y s)}{y^{\alpha+1}}dy$ by $\alpha s^2\lambda^2/(2(2-\alpha))$ implies for $s>0$,
%\begin{align}
%\varphi_1(s)\leq 1-s^\alpha+\frac{\alpha s^2\lambda^2}{2(2-\alpha)}.
%\end{align}
Now, it is not difficult to see that the above right-hand side can take negative values, e.g. for $\alpha=3/2$ and $s=2$, contradicting 
infinite divisibility.
\end{rem}

\begin{prop}\label{prop:DNA}
Let $(\xi_i)_{i\geq 1}$ be a sequence of iid random variables such that $\xi_1\sim f_1$. 
For $1 < \alpha < 2$, let $X_n=\sum_{i=1}^n\xi_i/n^{\frac{1}{\alpha}}$, $n\geq 1$, and let $X_\infty\sim S\alpha S$ have characteristic function 
$\varphi_\infty(t)=\exp(-|t|^\alpha)$, $t\in\mathbb{R}$. Then,
\begin{align}
d_K(X_n,X_\infty)\leq \frac{C}{n^{\frac{2}{\alpha}-1}}, 
\end{align}
for some $C>0$ which depends only on $\alpha$.
\end{prop}
\begin{proof}
Let $\varphi_n$ be the characteristic function of $X_n$, $n\geq 1$.  
Adopting the notation of the proof of Proposition \ref{propSD}, for all $t\geq 0$,
\begin{align*}
\varepsilon_n(t)=i\varphi_\infty(t)\int_0^t \frac{\varphi_n(s)}{\varphi_\infty(s)}\Delta_n(s)ds,
\end{align*}
and similarly for $t\leq 0$. But, thanks to the  identity \eqref{eq3.7} applied to $X_n$ and $X_\infty$ with $f(\cdot)=e^{i s.}$,
\begin{align*}
i\Delta_n(s)=\dfrac{\varphi_n'(s)}{\varphi_n(s)}-\dfrac{\varphi_\infty'(s)}{\varphi_\infty(s)}.
\end{align*}
Moreover, for $s>0$,
\begin{align*}
&\dfrac{\varphi_\infty'(s)}{\varphi_\infty(s)}=-\alpha s^{\alpha-1}.
\end{align*}
By standard computations, for $s>0$,
\begin{align*}
\varphi_1(s)
%&=\alpha\int_0^{+\infty} \cos(\lambda s x)\frac{dx}{(1+x)^{\alpha+1}}\\
&=\alpha \int_1^{+\infty} \cos(\lambda s (y-1))\frac{dy}{y^{\alpha+1}}\\
%&=\cos(\lambda s) \int_1^{+\infty}\cos(\lambda s y)\frac{\alpha dy}{y^{\alpha+1}}+\sin(\lambda s) \int_1^{+\infty}\sin(\lambda s y)\frac{\alpha dy}{y^{\alpha+1}}\\
&=\cos(\lambda s)\left(1-s^\alpha+\int_0^1\frac{1-\cos(\lambda s y)}{y^{\alpha+1}}\alpha dy\right)+\sin(\lambda s)\bigg(\int_1^{+\infty}\dfrac{\sin(\lambda s z)}{z^{\alpha+1}}\alpha dz\bigg)\\
%\dfrac{\lambda \alpha s}{\alpha-1}-s^\alpha\lambda^\alpha \sin(\pi\alpha/2)\dfrac{\Gamma(2-\alpha)}
%{\alpha-1}\nonumber\\
%&\quad\quad+\int_0^1 \frac{s\lambda z-\sin(\lambda s z)}{z^{\alpha+1}}\alpha dz\bigg)\\
&=1-s^\alpha+\psi_1(s)+\psi_2(s)+\psi_3(s),
\end{align*}
where 
\begin{align*}
&\psi_1(s)=\int_0^1\frac{1-\cos(\lambda s y)}{y^{\alpha+1}}\alpha dy,\\
&\psi_2(s)=(\cos(\lambda s)-1)\bigg(\int_1^{+\infty}\dfrac{\cos(\lambda s z)}{z^{\alpha+1}}\alpha dz\bigg),
\end{align*}
and,
\begin{align*}
\psi_3(s)=\sin(\lambda s)\bigg(\int_1^{+\infty}\dfrac{\sin(\lambda s z)}{z^{\alpha+1}}\alpha dz\bigg). 
\end{align*}
%$\psi_1(s)=\int_0^1\frac{1-\cos(\lambda s y)}{y^{\alpha+1}}\alpha dy$, $\psi_2(s)=(\cos(\lambda s)-1)(1-s^\alpha+\psi_1(s))$ and $\psi_3(s)=\sin(\lambda s)((\lambda \alpha s)/(\alpha-1)-s^\alpha\lambda^\alpha \sin(\pi\alpha/2)\Gamma(2-\alpha)/(\alpha-1)+\int_0^1 \frac{s\lambda z-\sin(\lambda s z)}{z^{\alpha+1}}\alpha dz)$. 
Then,
\begin{align*}
\dfrac{\varphi_n'(s)}{\varphi_n(s)}&=n^{1-\frac{1}{\alpha}}\frac{\varphi_1'\left(\frac{s}{n^\frac{1}{\alpha}}\right)}{\varphi_1\left(\frac{s}{n^\frac{1}{\alpha}}\right)}\\
&=\dfrac{-\alpha s^{\alpha-1}+n^{1-\frac{1}{\alpha}}\psi_1'\left(\frac{s}{n^{\frac{1}{\alpha}}}\right)+n^{1-\frac{1}{\alpha}}\psi_2'\left(\frac{s}{n^{\frac{1}{\alpha}}}\right)+n^{1-\frac{1}{\alpha}}\psi_3'\left(\frac{s}{n^{\frac{1}{\alpha}}}\right)}{1-\frac{s^\alpha}{n}+\psi_1\left(\frac{s}{n^{\frac{1}{\alpha}}}\right)+\psi_2\left(\frac{s}{n^{\frac{1}{\alpha}}}\right)+\psi_3\left(\frac{s}{n^{\frac{1}{\alpha}}}\right)}.
\end{align*}
implying that,
\begin{align*}
\Delta_n(s)&=\dfrac{-\alpha s^{\alpha-1}+n^{1-\frac{1}{\alpha}}\psi_1'\left(\frac{s}{n^{\frac{1}{\alpha}}}\right)+n^{1-\frac{1}{\alpha}}\psi_2'\left(\frac{s}{n^{\frac{1}{\alpha}}}\right)+n^{1-\frac{1}{\alpha}}\psi_3'\left(\frac{s}{n^{\frac{1}{\alpha}}}\right)}{1-\frac{s^\alpha}{n}+\psi_1\left(\frac{s}{n^{\frac{1}{\alpha}}}\right)+\psi_2\left(\frac{s}{n^{\frac{1}{\alpha}}}\right)+\psi_3\left(\frac{s}{n^{\frac{1}{\alpha}}}\right)}+\alpha s^{\alpha-1}\\
&=\dfrac{n^{1-\frac{1}{\alpha}}\psi_1'\left(\frac{s}{n^{\frac{1}{\alpha}}}\right)+n^{1-\frac{1}{\alpha}}\psi_2'\left(\frac{s}{n^{\frac{1}{\alpha}}}\right)+n^{1-\frac{1}{\alpha}}\psi_3'\left(\frac{s}{n^{\frac{1}{\alpha}}}\right)-\alpha \frac{s^{2\alpha-1}}{n}+\alpha s^{\alpha-1}(\psi_1\left(\frac{s}{n^{\frac{1}{\alpha}}}\right)+\psi_2\left(\frac{s}{n^{\frac{1}{\alpha}}}\right)+\psi_3\left(\frac{s}{n^{\frac{1}{\alpha}}}\right))}{1-\frac{s^\alpha}{n}+\psi_1\left(\frac{s}{n^{\frac{1}{\alpha}}}\right)+\psi_2\left(\frac{s}{n^{\frac{1}{\alpha}}}\right)+\psi_3\left(\frac{s}{n^{\frac{1}{\alpha}}}\right)}.
\end{align*}
Before, bounding the quantity $\varepsilon_n$ let us provide bounds for the functions $\psi_1$, $\psi_2$ and $\psi_3$ and their derivatives. For $s>0$,
\begin{align}\label{ineq:Psi123}
|\psi_1(s)|\leq C_1 s^2, \quad |\psi_2(s)|\leq C_2 s^2,\quad |\psi_3(s)|\leq C_3 s^2, 
\end{align}
and,
\begin{align*}
|\psi'_1(s)|\leq C_4 s,\quad |\psi'_2(s)|\leq C_5 (s+s^2), \quad |\psi'_3(s)|\leq C_6 s, 
\end{align*}
for some strictly positive constants, $C_i$, $i=1,...,6$, depending only on $\alpha$.  Therefore, for $t>0$,
\begin{align*}
|\varepsilon_n(t)|&\leq |\varphi_\infty(t)|\int_0^t \frac{\left|\varphi^{n-1}_1\left(\frac{s}{n^{\frac{1}{\alpha}}}\right)\right|}{|\varphi_\infty(s)|}\bigg| n^{1-\frac{1}{\alpha}}\psi_1'\left(\frac{s}{n^{\frac{1}{\alpha}}}\right)+n^{1-\frac{1}{\alpha}}\psi_2'\left(\frac{s}{n^{\frac{1}{\alpha}}}\right)+n^{1-\frac{1}{\alpha}}\psi_3'\left(\frac{s}{n^{\frac{1}{\alpha}}}\right)-\alpha \frac{s^{2\alpha-1}}{n}\\
&\quad\quad+\alpha s^{\alpha-1}\left(\psi_1\left(\frac{s}{n^{\frac{1}{\alpha}}}\right)+\psi_2\left(\frac{s}{n^{\frac{1}{\alpha}}}\right)+\psi_3\left(\frac{s}{n^{\frac{1}{\alpha}}}\right)\right)\bigg|ds\\
&\leq C|\varphi_\infty(t)|\int_0^t \frac{\left|\varphi^{n-1}_1\left(\frac{s}{n^{\frac{1}{\alpha}}}\right)\right|}{|\varphi_\infty(s)|}\left(\frac{s}{n^{\frac{2}{\alpha}-1}}+\frac{s^2}{n^{\frac{3}{\alpha}-1}}+\frac{s^{2\alpha-1}}{n}+\frac{s^{\alpha+1}}{n^{\frac{2}{\alpha}}}\right)ds, 
\end{align*}
and so
\begin{align*}
|\varepsilon_n(n^{\frac{1}{\alpha}}t)|\leq Cn|\varphi_\infty(n^{\frac{1}{\alpha}}t)|\int_0^{t} \frac{|\varphi^{n-1}_1(u)|}{|\varphi_\infty(n^{\frac{1}{\alpha}}u)|}\left(u+u^2+u^{2\alpha-1}+u^{\alpha+1}\right)du.
\end{align*}
Let us now detail how to bound the ratio $|\varphi^{n-1}_1(u)|/|\varphi_\infty\left(n^{\frac{1}{\alpha}}u\right)|$. 
For $0<u\leq t$
\begin{align*}
\frac{|\varphi^{n-1}_1(u)|}{|\varphi_\infty\left(n^{\frac{1}{\alpha}}u\right)|}&\leq e^{nu^\alpha+(n-1)\ln \varphi_1(u)}\\
&\leq e^{nu^\alpha+(n-1)(-u^\alpha+\psi_1(u)+\psi_2(u)+\psi_3(u))}\\
%&\leq e^{u^\alpha-\psi_1(u)-\psi_2(u)-\psi_3(u)+n(\psi_1(u)+\psi_2(u)+\psi_3(u))}\\
&\leq ee^{n(\psi_1(u)+\psi_2(u)+\psi_3(u))}. 
\end{align*}
By \eqref{ineq:Psi123}, we can choose $\varepsilon\in (0,1)$ such that 
$0<C(\varepsilon)=\underset{u\in (0,\varepsilon)}{\max} (|\psi_1(u)|+|\psi_2(u)|+|\psi_3(u)|)/u^\alpha<1$, 
since $\alpha\in (1,2)$. Then, for $0<u\leq t\leq \varepsilon$,
\begin{align*}
\frac{|\varphi^{n-1}_1(u)|}{|\varphi_\infty\left(n^{\frac{1}{\alpha}}u\right)|}&\leq e e^{n C(\varepsilon) u^{\alpha}}, 
\end{align*}
which implies that, for $0<t\leq \varepsilon <1$,
\begin{align*}
|\varepsilon_n(n^{\frac{1}{\alpha}}t)|\leq Cne^{-n(1-C(\varepsilon))t^{\alpha}}\left(t^2+t^3+t^{2\alpha}+t^{\alpha+2}\right).  
\end{align*}
A similar bound can also be obtained for $-\varepsilon \leq t< 0$. 
Setting $T:=\varepsilon n^{{1}/{\alpha}}$, applying Esseen's inequality, and if $h_\alpha$ denotes the density of the 
$S\alpha S$-law, we finally get
\begin{align*}
d_K(X_n,X_\infty)&\leq C'_1\int_{-\varepsilon n^{\frac{1}{\alpha}}}^{+\varepsilon n^{\frac{1}{\alpha}}}\dfrac{|\varepsilon_n(t)|}{|t|}dt+C_2\dfrac{\|h_\alpha\|_\infty}{\varepsilon n^{\frac{1}{\alpha}}}\\
&\leq C'_1\int_{-\varepsilon}^{+\varepsilon}\dfrac{|\varepsilon_n(n^{\frac{1}{\alpha}}t)|}{|t|}dt+C_2\dfrac{\|h_\alpha\|_\infty}{\varepsilon n^{\frac{1}{\alpha}}}\\
&\leq C'_1\int_{0}^{+\varepsilon}Cne^{-n(1-C(\varepsilon))t^{\alpha}}\left(t+t^2+t^{2\alpha-1}+t^{\alpha+1}\right)dt+C_2\dfrac{\|h_\alpha\|_\infty}{\varepsilon n^{\frac{1}{\alpha}}}\\
&\leq C'_1\int_{0}^{n^{\frac{1}{\alpha}}\varepsilon}e^{-(1-C(\varepsilon))t^{\alpha}}\left(\frac{t}{n^{\frac{2}{\alpha}-1}}+\frac{t^2}{n^{\frac{3}{\alpha}-1}}+\frac{t^{2\alpha-1}}{n}+\frac{t^{\alpha+1}}{n^{\frac{2}{\alpha}}}\right)dt+C_2\dfrac{\|h_\alpha\|_\infty}{\varepsilon n^{\frac{1}{\alpha}}}\\
&\leq C_{\varepsilon,\alpha}\left(\frac{1}{n^{\frac{2}{\alpha}-1}}+\frac{1}{n^{\frac{3}{\alpha}-1}}+\frac{1}{n}+\frac{1}{n^{\frac{2}{\alpha}}}\right)+C_2\dfrac{\|h_\alpha\|_\infty}{\varepsilon n^{\frac{1}{\alpha}}}, 
\end{align*}
for some $C_{\varepsilon,\alpha}>0$ depending only on $\varepsilon$ and on $\alpha$. 
This concludes the proof of the proposition.

%&\leq C \dfrac{t}{1+t^\alpha}\left(\frac{t}{n^{\frac{2}{\alpha}-1}}+\frac{t^2}{n^{\frac{3}{\alpha}-1}}+\frac{t^{2\alpha-1}}{n}+\frac{t^{\alpha+1}}{n^{\frac{2}{\alpha}}}\right)
%\end{align}
%and similarly for $t\leq 0$. Thus, for any $T>0$,
%\begin{align}
%d_K(X_n,X_\infty)&\leq C'_1\int_{-T}^{+T}\dfrac{|\varepsilon_n(t)|}{|t|}dt+C_2\dfrac{\|h_\alpha\|_\infty}{T}\\
%&\leq C'_1 \left(\frac{T^{\frac{2}{\alpha+1}}}{n^{\frac{2}{\alpha}-1}}+\frac{T^{3-\alpha}}{n^{\frac{3}{\alpha}-1}}+\frac{T^\alpha}{n}+\frac{T^2}{n^{\frac{2}{\alpha}}}\right)+C_2\dfrac{\|h_\alpha\|_\infty}{T}
%\end{align}
%Setting $T=n^{\frac{\alpha+1}{\alpha+3}(\frac{2}{\alpha}-1)}$ concludes the proof of the proposition.
\end{proof}

\begin{rem}\label{rem:DNA}
The above result has to be compared with the ones available in the literature but for other types of Pareto laws. 
Very recently, and via Stein's method, a rate of convergence in 1-Wasserstein and for stable limiting laws is obtained in \cite{X17}. When specialized  to the Pareto law with density $f(x):=\alpha\lambda^{\alpha}/(2|x|^{\alpha+1})\bbone_{|x|>\lambda}$, described in the previous remark, 
this rate is of order $n^{-(2/\alpha-1)}$ (see \cite{X17}), 
which via the inequality \eqref{ineq:KW1} provides a rate of the order $n^{-(1/\alpha-1/2)}$ in Kolmogorov distance.  
%Note that $(\alpha+1)({2}/{\alpha}-1)/(\alpha + 3)>1/\alpha -1/2$, for $1<\alpha<2$. 
Moreover, a convergence rate of order $n^{-(2/\alpha-1)}$ in Kolmogorov distance is known to hold for the same Pareto law (see e.g., \cite{CW93} 
and references therein). The results of \cite{Hall81} also imply a rate of convergence of order $n^{-(2/\alpha-1)}$ in Kolmogorov distance for the Pareto law considered in Proposition \ref{prop:DNA} (see Corollary $1$ in \cite{Hall81}). At a different level, the rate $n^{-(2/\alpha-1)}$ also appears 
when one considers the convergence, in supremum norm, of the 
corresponding densities towards the stable density, see \cite{KK00}.  
\end{rem}

\begin{rem}\label{rem:DNAbis}
Analyzing the proof of Proposition~\ref{prop:DNA}, it is clearly possible to generalize the previous result,  
beyond the Pareto case, to more general distributions pertaining to the domain of normal attraction 
of the symmetric $\alpha$--stable distribution. 
Indeed, consider distribution functions of the form
\begin{align*}
&\forall x>0,\ F(x)=1-\dfrac{(c+a(x))}{x^{\alpha}},\\
&\forall x<0,\ F(x)=\dfrac{(c+a(-x))}{(-x)^{\alpha}},
\end{align*}
where the function $a$ is defined on $(0, +\infty)$ and such that $\underset{x\rightarrow+\infty}{\lim}a(x)=0$.   Moreover, 
let $a$ be bounded and continuous on $(0, +\infty)$ and such that $\underset{x\rightarrow+\infty}{\lim}xa(x)<+\infty$. 
Then, by straightforward computations
\begin{align*}
\int_{-\infty}^{+\infty}e^{it x}dF(x)=1-t^\alpha+\psi(t), 
\end{align*}
where $\psi$ is a real-valued function satisfying $|\psi(t)|\leq C_1|t|^2$ and 
$|\psi^\prime(t)|\leq C_2|t|$, for some $C_1>0$ and $C_2>0$, two constants depending only on $\alpha$ and $a$. Assuming as well that the probability measure associated with the distribution function $F$ is infinitely divisible, it follows that 
\begin{align*}
d_K(X_n,X_\infty)\leq \frac{C}{n^{\frac{2}{\alpha}-1}}, 
\end{align*}
where $X_n=\sum_{i=1}^n\xi_i/n^{\frac{1}{\alpha}}$, with $\xi_i$ iid random variables such that 
$\xi_1\sim F$, where $X_\infty\sim S\alpha S$ and where $C>0$ only depends on $\alpha$ and $a$.
\end{rem}

A further simple adaptation of the proofs of the previous results leads to explicit rates of 
convergence for the compound Poisson approximation of some classes of infinitely divisible distributions. 
First, the next result gives a Berry-Esseen type bound.

\begin{thm}\label{thm4.3}
Let $X\sim ID(b,0,\nu)$ nondegenerate be such that $\bbe |X|<\infty$ and such that
\begin{align}\label{HP:Reg}
\int_{-1}^{+1}|u|\nu(du)<\infty.  
\end{align}
Let $\varphi$, its characteristic function, be such that, for all $t\in\mathbb{R}$,
\begin{align}\label{HP:PolynBound2}
|\varphi(t)|\int_0^{|t|}\frac{ds}{|\varphi(s)|}\leq C |t|^{p},
\end{align}
where $C>0$ and $p\geq 1$, and let the law of $X$ be absolutely continuous with 
respect to the Lebesgue measure and have a bounded density. Let $X_n$, $n\geq 1$, be compound Poisson random variables each with characteristic function,
\begin{align}\label{CPA}
\varphi_n(t):=\exp\left(n\left((\varphi(t))^{\frac{1}{n}}-1\right)\right).
\end{align}
Then,
\begin{align}
d_K(X_n,X)
\leq C'\left(\frac{1}{n}\right)^{\frac{1}{p+2}}\left(|b_0|+\int_{-\infty}^{+\infty}|u|\nu(du)\right)^{\frac{2}{p+2}}.
\end{align}
\end{thm}

\begin{proof}
Clearly, $\varphi_n(t)\rightarrow \varphi(t)$, for all $t\in \mathbb{R}$, and 
so the sequence $(X_n)_{n\geq 1}$ converges in distribution towards $X$. 
Then, adopting the notations of the proof of Theorem~\ref{thm4.1},
\begin{align}
\frac{1}{i}\dfrac{d}{dt}\big(\varphi_n(t)\big)=\varphi_n(t)S(t)+\varphi_n(t)\Delta_n^+(t)-\varphi_n(t)\Delta_n^-(t)\nonumber.
\end{align}
Moreover, thanks to (\ref{CPA}),
\begin{align*}
\dfrac{d}{dt}\big(\varphi_n(t)\big)&=\frac{\varphi'(t)}{\varphi(t)}\big(\varphi(t)\big)^{\frac{1}{n}}\varphi_n(t)
=iS(t)\big(\varphi(t)\big)^{\frac{1}{n}}\varphi_n(t),
\end{align*}
since
\begin{align*}
S(t)=\frac{1}{i\varphi(t)}\frac{d}{dt}(\varphi(t)).
\end{align*}
Thus,
\begin{align*}
\Delta_n^+(t)-\Delta_n^-(t)=S(t)\left(\big(\varphi(t)\big)^{\frac{1}{n}}-1\right),
\end{align*}
which implies that:
\begin{align}
\varepsilon_n(t)=i\varphi(t)\int_0^t \frac{\varphi_n(s)}{\varphi(s)}S(s)\left(\big(\varphi(s)\big)^{\frac{1}{n}}-1\right)ds.
\end{align}
Therefore,
\begin{align}\label{ineq:Eps}
|\varepsilon_n(t)|\leq |\varphi(t)|\int_0^t\frac{1}{|\varphi(s)|}|S(s)|\big(\varphi(s)\big)^{\frac{1}{n}}-1|ds.
\end{align}
Next, by the very definition of $S$, Corollary \ref{cor:ExtSizeBias}, and straightforward computations,
\begin{align*}
S(s)&=m_0^+\varphi_{Y^+}(s)-m_0^-\varphi_{Y^-}(s)\\
&=\left(b_0^++\int_{-\infty}^{+\infty} e^{i s u}\tilde{\nu}_{+}(du)\right)
-\left(b_0^-+\int_{-\infty}^{+\infty} e^{i s u}\tilde{\nu}_{-}(du)\right)\\
&=b_0+\int_{-\infty}^{+\infty}e^{i s u}\tilde{\nu}(du).
\end{align*}
Thus,
\begin{align}\label{ineq:S}
|S(s)|\leq |b_0|+\int_{-\infty}^{+\infty}|u|\nu(du).
\end{align}
Moreover, further straightforward computations lead to:
\begin{align}\label{ineq:Ch}
|\big(\varphi(s)\big)^{\frac{1}{n}}-1|\leq \frac{|s|}{n}\left(|b_0|+\int_{-\infty}^{+\infty}|u|\nu(du)\right).
\end{align}
Combining (\ref{HP:PolynBound2}) and (\ref{ineq:Eps})--(\ref{ineq:Ch}) gives
\begin{align}\label{ineq:pointdiffcharPC}
|\varepsilon_n(t)|\leq C \frac{1}{n}\left(|b_0|+\int_{-\infty}^{+\infty}|u|\nu(du)\right)^2 |t|^{p+1}.
\end{align}
To conclude the proof of this theorem, proceed as in the end of the proof of Theorem~\ref{thm4.1}.
\end{proof}

\begin{prop}\label{prop4.2}
Let $X\sim ID(b,0,\nu)$ nondegenerate be such that $\bbe |X|^2<\infty$.
%\begin{align}\label{HP:Reg}
%\int_{|u|\leq 1}|u|\nu(du)<\infty.  
%\end{align}
Let $\varphi$, its characteristic function, be such that, for all $t\in\mathbb{R}$,
\begin{align}\label{HP:PolynBound3}
|\varphi(t)|\int_0^{|t|}\frac{ds}{|\varphi(s)|}\leq C |t|^{p},
\end{align}
where $C>0$ and $p\geq 1$. Further, let the law of $X$ be absolutely continuous with respect 
to the Lebesgue measure and have a bounded density. Let $X_n$, $n\geq 1$, be compound Poisson random variables each with characteristic function,
\begin{align}\label{CPA2}
\varphi_n(t):=\exp\left(n\left((\varphi(t))^{\frac{1}{n}}-1\right)\right).
\end{align}
Then,
\begin{align}
d_K(X_n,X)
\leq C'\left(\frac{1}{n}\right)^{\frac{1}{p+4}}\left(|\bbe X|+\int_{-\infty}^{+\infty}u^2\nu(du)\right)^{\frac{2}{p+4}}.
\end{align}
\end{prop}
\begin{proof}
The proof is similar to the proof of Theorem \ref{thm4.3} and so it is only sketched. 
Clearly, $\varphi_n(t)\rightarrow \varphi(t)$, for all $t\in \mathbb{R}$ and so 
the sequence $(X_n)_{n\geq 1}$ converges in distribution towards $X$. 
Then, with the previous notations,
\begin{align}
\dfrac{d}{dt}\big(\varphi_n(t)\big)=R(t)\varphi_n(t)+\Delta_n(t)\varphi_n(t)\nonumber.
\end{align}
Moreover, thanks to \eqref{CPA2},
\begin{align*}
\dfrac{d}{dt}\big(\varphi_n(t)\big)&=\frac{\varphi'(t)}{\varphi(t)}\big(\varphi(t)\big)^{\frac{1}{n}}\varphi_n(t)
=R(t)\big(\varphi(t)\big)^{\frac{1}{n}}\varphi_n(t),
\end{align*}
since
\begin{align*}
R(t)=\frac{1}{\varphi(t)}\frac{d}{dt}(\varphi(t)).
\end{align*}
Thus,
\begin{align*}
\Delta_n(t)=R(t)\left(\big(\varphi(t)\big)^{\frac{1}{n}}-1\right),
\end{align*}
which implies that:
\begin{align*}
\varepsilon_n(t)=\varphi(t)\int_0^t \frac{\varphi_n(s)}{\varphi(s)}R(s)\left(\big(\varphi(s)\big)^{\frac{1}{n}}-1\right)ds.
\end{align*}
Therefore,
\begin{align}\label{ineq:Eps2}
|\varepsilon_n(t)|\leq |\varphi(t)|\int_0^t\frac{1}{|\varphi(s)|}|R(s)|\big(\varphi(s)\big)^{\frac{1}{n}}-1|ds.
\end{align}
By the very definition of $R$, Proposition \ref{local}, and straightforward computations, we have
\begin{align*}
R(s)&=-s\left(\int_{-\infty}^{+\infty} u^2\nu(du)\right) \varphi_Y(s)+i\bbe X\\
&=i\int_{-\infty}^{+\infty}\left(e^{isu}-1\right)u\nu(du)+i\bbe X.\\
\end{align*}
Thus,
\begin{align}\label{ineq:R}
|R(s)|\leq |\bbe X|+ |s|\int_{-\infty}^{+\infty}u^2\nu(du).
\end{align}
Moreover, further straightforward computations lead to:
\begin{align}\label{ineq:Ch2}
|\big(\varphi(s)\big)^{\frac{1}{n}}-1|\leq \frac{|s|}{n}\left(|\bbe X|+ |s|\int_{-\infty}^{+\infty}v^2\nu(dv)\right).
\end{align}
Combining \eqref{HP:PolynBound3} and \eqref{ineq:Eps2}--\eqref{ineq:Ch2} gives
\begin{align*}
|\varepsilon_n(t)|&\leq C \frac{1}{n}\left(|\bbe X|+ |t|\int_{-\infty}^{+\infty}u^2\nu(du)\right)^2 |t|^{p+1},\\
&\leq C \frac{1}{n}\left(|\bbe X|+ \int_{-\infty}^{+\infty}u^2\nu(du)\right)^2(|t|^{p+1}+|t|^{p+2}+|t|^{p+3}).
\end{align*}
To conclude the proof of this proposition, proceed as in the end of the proof of Proposition~\ref{thm4.3}.
\end{proof}
\noindent

\begin{rem}\label{rem:CPA}
(i) Under the condition
\begin{align}
A:=\underset{s\in\mathbb{R}}\sup\,\left|\int_{-\infty}^{+\infty}(e^{isu}-1)u\nu(du)\right|<\infty,
\end{align}
the upper bound on the Kolmogorov distance in Proposition \ref{prop4.2} becomes 
\begin{align}
d_K(X_n,X)
\leq C'\left(\frac{1}{n}\right)^{\frac{1}{p+2}}\left(|\bbe X|+A\right)^{\frac{2}{p+2}},
\end{align}
which is comparable to the one obtained in Theorem \ref{thm4.3}, and is, for instance, verified in case the L\'evy measure of X satisfies the assumptions of Theorem \ref{thm4.3}.

(ii) Once more, versions of Theorem \ref{thm4.3} and of Proposition~\ref{prop4.2} can be derived for the 
smooth Wasserstein distance. 
Let us develop this point a bit more:  assume 
that $X$ has finite exponential moment, namely, that $\bbe e^{\lambda |X|}$ is finite for some $\lambda>0$. 
This condition implies that the characteristic function $\varphi$ is analytic in a horizontal strip of the complex 
plane containing 
the real axis.  Then, by the very definition of $\varphi_n$ and the use of the  L\'evy-Raikov Theorem (see e.g. Theorem $10.1.1$ in \cite{Lu}), it follows 
that $\varphi_n$ is analytic at least in the same horizontal strip. Moreover, in this strip, still by its very definition $(\varphi_n)_{n\geq 1}$ converges pointwise towards $\varphi$.  Then, $e^{\eta |X_n|}$ (for some $\eta>0$) are uniformly integrable.
Therefore, if $\bbe e^{\lambda |X|}$ is finite and if the assumptions \eqref{HP:Reg} and \eqref{HP:PolynBound2} hold true,  \eqref{ineq:pointdiffcharPC} and  \cite[Theorem 1]{AMPS17} lead to,
\begin{align}
d_{W_{p+2}}(X_n,X)\leq C \frac{\sqrt{\ln n}}{n}\left(|b_0|+\int_{-\infty}^{+\infty}|u|\nu(du)\right)^2,
\end{align}
for some constant $C$ only depending on the limiting distribution.
\end{rem}
\noindent
We now present some examples illustrating the applicability of our methods.

(i)  Let  $X$ be a compound Poisson $X\sim {\rm CP}(\nu(\mathbb{R}),\nu_0)$, then for all $t\in \mathbb{R}$,  
\begin{align*}
L\varphi(t)\leq e^{\nu(\mathbb{R})\int_{-\infty}^{+\infty}(\cos(ut)-1)\nu_0(du)}
\int_0^{|t|}e^{\nu(\mathbb{R})\int_{-\infty}^{+\infty}(1-\cos(us))\nu_0(du)}ds
\leq e^{2\nu(\mathbb{R})} |t|.  
\end{align*}

(ii) The gamma random variable with parameters $\alpha\geq 1$ and $\beta>0$ satisfies the 
assumptions of Theorem \ref{thm4.3}. Indeed, \eqref{HP:Reg} and the boundedness of the 
density are automatically verified, and moreover 
\begin{align*}
|\varphi(t)|\int_0^{|t|}\frac{ds}{|\varphi(s)|}\leq |t|,
\end{align*}
for all $t\in \mathbb{R}$.

(iii) Let $q\geq 3$ and let $(\lambda_1,...,\lambda_q)$ be $q$ non-zero distinct reals. 
Let $X:= \sum_{k=1}^q\lambda_k (Z_k^2-1)$, where the $\{Z_i,\ i=1,...,q\}$ are iid standard 
normal random variables. Clearly $X$ is infinitely divisible and its L\'evy measure is given by 
\begin{align*}
\frac{\nu(du)}{du}:=\left(\sum_{\lambda\in \Lambda_+}\dfrac{e^{-u/(2\lambda)}}{2u}\right)\bbone_{(0,+\infty)}(u)
+\left(\sum_{\lambda\in \Lambda_-}\dfrac{e^{-u/(2\lambda)}}{2(-u)}\right)\bbone_{(-\infty,0)}(u), 
\end{align*}
where $\Lambda_+=\{\lambda_k: \lambda_k> 0\}$ and $\Lambda_-=\{\lambda_k: \lambda_k< 0\}$ have finite 
cardinality, and so the condition \eqref{HP:Reg} is verified. 
Moreover $\varphi(t):=\prod_{j=1}^q e^{-it\lambda_j}/(1-2it\lambda_j)^{1/2}$ and thus
\begin{align*}
\dfrac{1}{\left(1+4\lambda^2_{\max}t^2\right)^{\frac{q}{4}}}
\leq |\varphi(t)|\leq \dfrac{1}{\left(1+4\lambda^2_{\min}t^2\right)^{\frac{q}{4}}},
\end{align*}
with $\lambda_{\max}=\underset{k\geq 1}{\max}\, |\lambda_k|$ and 
$\lambda_{\min}=\underset{k\geq 1}{\min}\, |\lambda_k|$. 
This readily implies that $X$ has a bounded density and that, for all $t\in\mathbb{R}$,
\begin{align*}
|\varphi(t)|\int_0^{|t|}\frac{ds}{|\varphi(s)|}\leq C |t|,
\end{align*}
where
\begin{align*}
C:=\underset{t\in \mathbb{R}}{\sup}\left(\dfrac{(1+4\lambda^2_{\max}t^2)^{\frac{q}{4}}}
{(1+4\lambda^2_{\min}t^2)^{\frac{q}{4}}}\right) =\left( \frac{\lambda_{\max}}{\lambda_{\min}}\right)^{q/2}.
\end{align*}\\

(iv) More generally, let $(\lambda_k)_{k\geq 1}$ be an absolutely summable sequence such that $|\lambda_k|\ne 0$, 
for all $k\geq 1$. Let $X:= \sum_{k=1}^{+\infty}\lambda_k (Z_k^2-1)$ where $(Z_k)_{k\geq 1}$ is a sequence 
of iid standard normal random variables. Since $(\lambda_k)_{k\geq 1}$ is absolutely summable, 
the condition \eqref{HP:Reg} is verified. Let us now fix $N\geq 3$ and assume that the absolute values 
of the eigenvalues $(\lambda_k)_{k\geq 1}$ are indexed in decreasing 
order, i.e. $|\lambda_1| \geq |\lambda_2|\geq ...\geq |\lambda_N|\geq ...$. Then,
\begin{align*}
\dfrac{\psi_N(t)}{\left(1+4t^2|\lambda_1|^2\right)^{\frac{N}{4}}}\leq |\varphi(t)|
\leq \dfrac{\psi_N(t)}{\left(1+4t^2|\lambda_N|^2\right)^{\frac{N}{4}}},
\end{align*}
where
\begin{align*}
\psi_N(t):=\prod_{k=N+1}^{+\infty}\frac{1}{(1+4t^2\lambda_k^2)^{\frac{1}{4}}}.
\end{align*}
Since $0\leq \psi_N(t)\leq 1$, it is clear that $X$ has a bounded density. Moreover, for each $N$, 
$\psi_N$ is a decreasing function, thus,
\begin{align}\label{bound:SecChaos}
|\varphi(t)|\int_0^{|t|}\frac{ds}{|\varphi(s)|}\leq C |t|,
\end{align}
with
\begin{align*}
C:=\underset{t\in \mathbb{R}}{\sup}\left(\dfrac{(1+4\lambda^2_{1}t^2)^{\frac{q}{4}}}
{(1+4\lambda^2_{N}t^2)^{\frac{q}{4}}}\right) =\left( \frac{\lambda_1}{\lambda_N}\right)^{q/2}.
\end{align*}
\\
The next theorem pertains to quantitative convergence results inside the second Wiener chaos. 
Theorem $3.1$ of \cite{NP12} puts forward the fact that a sequence of second order Wiener chaos random variables 
converging in law, necessarily converges towards a random variable which is the sum of a centered Gaussian 
random variable (possibly degenerate) and of an independent second order Wiener chaos random variable.  
It is, therefore, natural to consider convergence in law for the following random variables
\begin{align*}
\sum_{k=1}^{+\infty}\lambda_{n,k}(Z^2_k-1)\underset{n\rightarrow +\infty}{\Longrightarrow}
\sum_{k=1}^{+\infty}\lambda_{\infty,k}(Z^2_k-1).
\end{align*}
To study this issue, below, $\ell^1$ denotes the space of absolutely summable real-valued sequences and 
for any such sequence $\lambda=(\lambda_k)_{k\geq 1}\in \ell^1$,  let 
$\|\lambda\|_{\ell^1}:=\sum_{k=1}^{+\infty}|\lambda_k|$. 

\begin{thm}\label{thm4.5} Let $(\lambda_n)_{n\geq 1}$ be a sequence of elements of $\ell^1$, 
converging (in $\|\cdot\|_{\ell_1}$) 
towards $\lambda_\infty\in \ell^1$. Moreover, let $|\lambda_{\infty,k}|\ne 0$ and 
$|\lambda_{n,k}|\ne 0$, for all $k\geq 1$, $n\geq 1$, and  
further, let $\sum_{k=1}^{+\infty}\lambda^2_{n,k}=\sum_{k=1}^{+\infty}\lambda^2_{\infty,k}=1/2$. 
Next, set $X_n=\sum_{k=1}^{+\infty}\lambda_{n,k}(Z^2_k-1)$, $n\geq 1$, $X_\infty
=\sum_{k=1}^{+\infty}\lambda_{\infty,k}(Z^2_k-1)$, and let   
\begin{align*}
\Delta_n&:=|\|\lambda^+_n\|_{\ell^1}-\|\lambda^+_\infty\|_{\ell^1}|
+|\|\lambda^-_n\|_{\ell^1}-\|\lambda^-_\infty\|_{\ell^1}| \\ 
&\qquad \qquad \qquad \qquad \quad +\|\lambda^+_n-\lambda^+_\infty\|_{\ell^1}
+\|\lambda^-_n-\lambda^-_\infty\|_{\ell^1},
\end{align*}
where $\Lambda_n^\pm:=\{\lambda^\pm_{n,k}, k\geq 1\}=\{\lambda_{n,k},\, \lambda_{n,k} > 0\, (< 0)\}$, 
%$\Lambda_n^-:=\{\lambda^-_{n,k}, k\geq 1\}=\{\lambda_{n,k},\, \lambda_{n,k}<0\}$ 
and similarly for $\Lambda_\infty^\pm$. Then,
\begin{align}
d_K(X_n,X_\infty)\leq C'_\infty \sqrt{\Delta_n}.
\end{align}
and
\begin{align}
d_{{W}_{2}}(X_n,X_\infty) \leq C_{\infty}'' \Delta_n \sqrt{|\ln \Delta_n|},
\end{align}
for some positive constants $C'_\infty, C''_\infty$ depending only on $X_\infty$.
\end{thm}

\begin{proof}
Since, $\lambda_n$, $n\geq 1$, and $\lambda_\infty$ are absolutely summable, 
the conditions \eqref{HP:Reg0} of Theorem \ref{thm4.1} are satisfied. 
Then, from the proof of Theorem~\ref{thm4.1} and, for all $t\geq 0$,
\begin{align}\label{eq:epsi}
\varepsilon_n(t)=i\varphi_\infty(t)\int_0^t \frac{\varphi_n(s)}{\varphi_\infty(s)}\big(\Delta_n^+(s)
-\Delta_n^-(s)\big)ds.
\end{align}
Let us compute the quantities $\Delta_n^+$ and $\Delta_n^-$. By definition,
\begin{align*}
b^n_0=-\int_{-\infty}^{+\infty}u\nu_n(du)=-\sum_{\lambda\in \Lambda^+_{n}}\lambda
+\sum_{\lambda\in \Lambda^-_{n}}(-\lambda).
\end{align*}
Thus,
\begin{align*}
(b^n_0)^+&=\sum_{\lambda\in \Lambda^-_{n}}(-\lambda),\quad \quad\quad\quad\quad (b^n_0)^-
=\sum_{\lambda\in \Lambda^+_{n}}\lambda,\\
\tilde{\nu}^+_n(du)&=\frac{1}{2}\sum_{\lambda\in \Lambda^+_{n}}
e^{-\frac{u}{2\lambda}}\bbone_{(0,+\infty)}(u)du,\quad 
\tilde{\nu}^-_n(du)=\frac{1}{2}\sum_{\lambda\in \Lambda^-_{n}}
e^{-\frac{u}{2\lambda}}\bbone_{(-\infty,0)}(u)du.
\end{align*}
This implies that:
\begin{align*}
(m_0^n)^+&=(m_0^n)^-=\|\lambda_n\|_{\ell^1},\\
\varphi_{Y_n^+}(t)&=\frac{1}{(m_0^n)^+}\left((b^n_0)^++\sum_{\lambda\in \Lambda^+_{n}}
\dfrac{\lambda}{1-2it\lambda}\right),\\
\varphi_{Y_n^-}(t)&=\frac{1}{(m_0^n)^-}\left((b^n_0)^-+\sum_{\lambda\in \Lambda^-_{n}}
\dfrac{-\lambda}{1-2it\lambda}\right).
\end{align*}
Then, after straightforward computations
\begin{align*}
\Delta_n^+(s)-\Delta_n^-(s)=\,&b_0^n-b_0^{\infty}+\sum_{\lambda\in \Lambda^+_{n}}
\dfrac{\lambda}{1-2is\lambda}
-\sum_{\lambda\in \Lambda^+_{\infty}}\dfrac{\lambda}{1-2is\lambda}\\
&\quad+\sum_{\lambda\in \Lambda^-_{\infty}}\dfrac{-\lambda}{1-2is\lambda}
-\sum_{\lambda\in \Lambda^-_{n}}\dfrac{-\lambda}{1-2is\lambda}.
\end{align*}
Therefore,
\begin{align*}
|\Delta_n^+(s)-\Delta_n^-(s)|\leq &|\|\lambda^+_n\|_{\ell^1}-\|\lambda^+_\infty\|_{\ell^1}|
+|\|\lambda^-_n\|_{\ell^1}-\|\lambda^-_\infty\|_{\ell^1}|\\
&\quad+\|\lambda^+_n-\lambda^+_\infty\|_{\ell^1}+\|\lambda^-_n-\lambda^-_\infty\|_{\ell^1}.
\end{align*}
Combining the previous bound together with \eqref{bound:SecChaos} and \eqref{eq:epsi} entails,
\begin{align*}
|\varepsilon_n(t)|\leq C_\infty |t| \Delta_n.
\end{align*}
Finally, proceeding as in the proof of Theorem \ref{thm4.1},
\begin{align*}
d_K(X_n,X_\infty)\leq C'_\infty \sqrt{\Delta_n}.
\end{align*}
As for the upper bound on the smooth Wasserstein distance, recall the following tail property of 
second Wiener chaoses: there exists $K>0$ such that for all $X$ in the second Wiener chaos 
having unit variance and for all $x>2$:
\begin{align*}
\mathbb{P}(|X|>x)\leq \exp\left(-K x\right),
\end{align*}
(see e.g. \cite[Theorem 6.7]{SJ97}). This tail estimate implies that 
\begin{align*}
\underset{n\geq 1}{\sup}\, \bbe\, e^{\eta |X_n|}<\infty,\quad\quad \bbe\, e^{\eta_\infty |X_\infty|}<\infty, 
\end{align*}
for some $\eta,\eta_\infty>0$, and \cite[Theorem 1]{AMPS17} finishes the proof of the theorem.
\end{proof}

\begin{rem}\label{rem:chaos}
In the previous theorem, one could also consider $(\lambda_n)_{n\geq 1}$ such that, for all $n\geq 1$, there exists $k_n\geq 1$ (converging to $+\infty$ with $n$)  such that for all $k= 1,...,k_n$ $|\lambda_{n,k}|\ne 0$ and for all $k\geq k_{n}+1$, $\lambda_{n,k}=0$.  The quantity $\Delta_n$ 
would then depend on the remainder term, $R_n:=\sum_{k=k_n+1}^{+\infty}|\lambda_{\infty,k}|$.  
\end{rem}
%}

To conclude this discussion on the compound Poisson approximation of infinitely divisible distributions  let us 
consider the stable case. Clearly, and as already indicated, an $\alpha$-stable random variable satisfies neither the hypotheses 
of Corollary \ref{cor:ExtSizeBias} nor those of Proposition \ref{local}.  Nevertheless, the 
identities \eqref{eq:StableFrac} and \eqref{characstable} lead to our next result.
\begin{thm}\label{thm4.4}
Let $\alpha\in (1,2)$ and let  $X$ be an $\alpha$--stable random variable 
with L\'evy measure given by \eqref{StableLM} where $c_1,c_2\geq 0$ are such that $c_1+c_2>0$,   
and with characteristic function $\varphi$.  Let $X_n$, $n\geq 1$, 
be compound Poisson random variables each having a characteristic function
\begin{align}\label{CPA4}
\varphi_n(t):=\exp\left(n\left((\varphi(t))^{\frac{1}{n}}-1\right)\right).
\end{align}
Then,
\begin{align}
d_K(X_n,X) \leq C \frac{1}{n^{\frac{1}{\alpha+1}}},
\end{align}
where $C>0$ depends only on $\alpha$, $c_1$ and $c_2$.
\end{thm}
\begin{proof} 
Thanks to \eqref{characstable} with $f(\cdot)=e^{it.}$,
\begin{align}\label{eq:diffsta}
\dfrac{d}{dt}(\varphi(t))=i\varphi(t)\left(c_2\int_0^{+\infty}(1-e^{-itu})\frac{du}{u^\alpha}-c_1\int_0^{+\infty}
(1-e^{itu})\frac{du}{u^\alpha}+\frac{c_1-c_2}{\alpha-1}\right).
\end{align}
Next, set
\begin{align*}
S(t):=\left(c_2\int_0^{+\infty}(1-e^{-itu})\frac{du}{u^\alpha}-c_1\int_0^{+\infty}(1-e^{itu})
\frac{du}{u^\alpha} +  \frac{c_1-c_2}{\alpha-1}\right).
\end{align*}
Moreover, from \eqref{CPA4},
\begin{align*}
\dfrac{d}{dt}(\varphi_n(t))=\frac{\varphi'(t)}{\varphi(t)}(\varphi(t))^{\frac{1}{n}}\varphi_n(t)
=iS(t)(\varphi(t))^{\frac{1}{n}}\varphi_n(t).
\end{align*}
Introducing the quantity $\Delta_n(t):=S(t)((\varphi(t))^{\frac{1}{n}}-1)$, 
\begin{align}\label{eq:diffCPAsta}
\dfrac{d}{dt}(\varphi_n(t))=iS(t)\varphi_n(t)+i\varphi_n(t)\Delta_n(t).
\end{align}
subtracting \eqref{eq:diffsta} from \eqref{eq:diffCPAsta} and setting $\varepsilon_n(t)
=\varphi_n(t)-\varphi(t)$, lead to
\begin{align*}
\dfrac{d}{dt}(\varepsilon_n(t))=iS(t)\varepsilon_n(t)+i\varphi_n(t)\Delta_n(t).
\end{align*}
Thus, for $t\geq 0$,
\begin{align*}
\varepsilon_n(t)=i\varphi(t)\int_0^t\frac{\varphi_n(s)}{\varphi(s)}\Delta_n(s)ds,
\end{align*}
and similarly for $t\leq 0$. Let us bound $S$. For $s> 0$,
\begin{align*}
|S(s)|&\leq c_2\int_0^{+\infty}|1-e^{-isu}|\frac{du}{u^\alpha}+c_1\int_0^{+\infty}|1-e^{isu}|
\frac{du}{u^\alpha}+\frac{|c_1-c_2|}{\alpha-1}\\
&\leq (c_1+c_2)s^{\alpha-1}\dfrac{2^{2-\alpha}}{(2-\alpha)(\alpha-1)}+\frac{|c_1-c_2|}{\alpha-1},
\end{align*}
using
\begin{align*}
\int_0^{+\infty}|1-e^{isu}|\frac{du}{u^\alpha}&\leq s^{\alpha-1}\bigg(\int_0^2\frac{du}{u^{\alpha-1}}
+2\int_2^{+\infty}\frac{du}{u^\alpha}\bigg)\\
&\leq |s|^{\alpha-1} \frac{2^{2-\alpha}}{(2-\alpha)(\alpha-1)}.
\end{align*}
Moreover,
\begin{align*}
\left|(\varphi(s))^{\frac{1}{n}}-1\right|\leq \frac{|s|}{n}\left((c_1+c_2)|s|^{\alpha-1}\dfrac{2^{2-\alpha}}{(2-\alpha)(\alpha-1)}
+\frac{|c_1-c_2|}{\alpha-1}\right),
\end{align*}
and
\begin{align*}
\varphi(s):=\exp\left(is\bbe X-c|s|^\alpha\left(1-i\beta \tan \frac{\pi\alpha}{2} \operatorname{sgn}(s)\right)\right),
\end{align*}
with $c=c_1+c_2$ and $\beta=(c_1-c_2)/(c_1+c_2)$. This implies that, for $t>0$,
\begin{align*}
|\varepsilon_n(t)|&\leq \frac{1}{n}|\varphi(t)|\int_0^t\frac{1}{|\varphi(s)|}
\left((c_1+c_2)s^{\alpha-1}\dfrac{2^{2-\alpha}}{(2-\alpha)(\alpha-1)}
+\frac{|c_1-c_2|}{\alpha-1}\right)^2|s|ds\\
&\leq \frac{1}{n}e^{-ct^\alpha}\int_0^t e^{cs^\alpha}s \left((c_1+c_2)s^{\alpha-1}\dfrac{2^{2-\alpha}}{(2-\alpha)(\alpha-1)}
+\frac{|c_1-c_2|}{\alpha-1}\right)^2ds\\
&\leq \frac{2}{n}\!\left(\!(c_1+c_2)\dfrac{2^{2-\alpha}}{(2-\alpha)(\alpha-1)}
+\frac{|c_1-c_2|}{\alpha-1}\right)^2\!(t+t^{2\alpha-1})\!\left(\!e^{-ct^\alpha}\!\!\!\int_0^t \!e^{cs^\alpha}ds\!\right)\!\\
&\leq \frac{2C}{n}\!\left(\!(c_1+c_2)\dfrac{2^{2-\alpha}}{(2-\alpha)(\alpha-1)}
+\frac{|c_1-c_2|}{\alpha-1}\right)^2 \dfrac{\!(t^2+t^{2\alpha})\!}{1+t^{\alpha}},
\end{align*}
where \eqref{ineq:FineBoundStable} is used to obtain the last inequality and where $C>0$ only 
depends on $\alpha, c_1$ and $c_2$. Therefore, for all $t\in\mathbb{R}$
\begin{align*}
|\varepsilon_n(t)|\leq \frac{C'}{n}\dfrac{\!(t^2+|t|^{2\alpha})\!}{1+|t|^{\alpha}},
\end{align*}
for some $C'>0$ depending only on $\alpha$ and $c$. To conclude the proof of this theorem, 
we proceed as in the end of the proof of Theorem~\ref{thm4.3}: by Esseen inequality,
\begin{align*}
d_K(X_n,X)\leq C_1\int_{-T}^T \dfrac{|\varepsilon_n(t)|}{|t|}dt+C_2\dfrac{\|h_\alpha\|_\infty}{T},
\end{align*}
where $C_1>0$, $C_2>0$, while $h_\alpha$ is the density of the stable distribution. Thus,
\begin{align*}
d_K(X_n,X)\leq \frac{C'_1}{n}\left(\int_0^{T}\dfrac{t}{1+t^{\alpha}}dt+\int_0^{T}
\dfrac{t^{2\alpha-1}}{1+t^{\alpha}}dt\right)+C_2\dfrac{\|h_\alpha\|_\infty}{T}.
\end{align*}
Now, the idea is to exploit the different behaviors of the functions 
$t\rightarrow t/(1+t^{\alpha})$ and $t\rightarrow t^{2\alpha-1}/(1+t^{\alpha})$ 
at $0$ and at infinity to optimize the powers of $T$ appearing on the right-hand side of the previous inequality. 
Then, for $T\geq1$,
\begin{align*}
\int_0^{T}\dfrac{t}{1+t^{\alpha}}dt&\leq \int_{0}^\varepsilon  \dfrac{t}{1+t^{\alpha}}dt
+ \int_{\varepsilon}^T  \dfrac{t}{1+t^{\alpha}}dt\\
&\leq \frac{\varepsilon^2}{2}+\varepsilon^{1-\alpha}T\\
&\leq C T^{\frac{2}{\alpha+1}}, 
\end{align*}
where we optimized in $\varepsilon$ in the last line and where $C>0$ only depends on $\alpha$. Similarly,
\begin{align*}
\int_0^{T}\dfrac{t^{2\alpha-1}}{1+t^{\alpha}}dt\leq C T^{\alpha}.
\end{align*}
Thus,
\begin{align*}
d_K(X_n,X)\leq \frac{C}{n}\left(T^{\alpha}+T^{\frac{2}{\alpha+1}}\right)+C_2\dfrac{\|h_\alpha\|_\infty}{T}.
\end{align*}
Finally, choosing $T=n^{\frac{1}{\alpha+1}}$ finishes the proof of the theorem.
\end{proof}

For symmetric $\alpha$-stable distribution, and in view of the proof of 
Proposition~\ref{prop:DNA}, the rate of convergence obtained above can be improved to 
$n^{1/\alpha}$. This is the content of the next result.

\begin{thm}\label{thm4.5}
Let $\alpha\in (1,2)$ and let $X\sim S\alpha S$ 
with characteristic function given by $\varphi(t)=\exp(-|t|^\alpha)$.   
Let $X_n$, $n\geq 1$, be compound Poisson random variables each having characteristic function 
\begin{align*}
\varphi_n(t):=\exp\left(n\left((\varphi(t))^{\frac{1}{n}}-1\right)\right).
\end{align*}
Then,
\begin{align}
d_K(X_n,X) \leq  \frac{C}{n^{\frac{1}{\alpha}}},
\end{align}
where $C>0$ only depends on $\alpha$.
\end{thm}
\begin{proof}
From the proof of Proposition~\ref{prop:DNA} (with its notations), for $t\geq 0$,
\begin{align*}
\varepsilon_n(t)=i\varphi(t)\int_0^t\frac{\varphi_n(s)}{\varphi(s)}\Delta_n(s)ds.
\end{align*}
Moreover, for all $0\leq s\leq t$,
\begin{align*}
|\Delta_n(s)|\leq \frac{s^{2\alpha-1}}{n},
\end{align*}
hence,
\begin{align*}
\left|\varepsilon_n(n^\frac{1}{\alpha}t)\right| \leq C n \left|\varphi\left(n^\frac{1}{\alpha}t\right)\right|\int_0^t
\left|\frac{\varphi_n(n^\frac{1}{\alpha}s)}{\varphi(n^\frac{1}{\alpha}s)}\right| |s|^{2\alpha-1}ds.  
\end{align*}
We next detail how to bound the ratio $\left|\varphi_n\left(n^\frac{1}{\alpha}s\right)/ \varphi\left(n^\frac{1}{\alpha}s\right)\right|$.  
For $0\leq s\leq t\leq 1$,
\begin{align*}
\left|\frac{\varphi_n\left(n^\frac{1}{\alpha}s\right)}{\varphi\left(n^\frac{1}{\alpha}s\right)}\right|
\leq \exp \left(n (\exp(-s^\alpha)-1+s^\alpha)\right).  
\end{align*}
Now, pick $\varepsilon \in (0,1)$ such that 
$0<C(\varepsilon)=\underset{s\in (0,\varepsilon)}{\max} (\exp(-s^\alpha)-1+s^\alpha)/s^{\alpha}<1$. Then,
\begin{align*}
\left|\varepsilon_n\left(n^\frac{1}{\alpha}t\right)\right| \leq C n e^{-n(1-C(\varepsilon))t^\alpha}t^{2\alpha}.
\end{align*}
A similar bound can also be obtained for $ t\leq 0$. 
Setting $T:=\varepsilon n^{{1}/{\alpha}}$, and applying Esseen's inequality, we finally get
\begin{align*}
d_K(X_n,X_\infty)&\leq C'_1\int_{-\varepsilon n^{\frac{1}{\alpha}}}^{+\varepsilon n^{\frac{1}{\alpha}}}
\dfrac{|\varepsilon_n(t)|}{|t|}dt+C_2\dfrac{\|h_\alpha\|_\infty}{\varepsilon n^{\frac{1}{\alpha}}}\\
&\leq C'_1\int_{-\varepsilon}^{+\varepsilon}\dfrac{|\varepsilon_n(n^{\frac{1}{\alpha}}t)|}{|t|}dt+C_2\dfrac{\|h_\alpha\|_\infty}{\varepsilon n^{\frac{1}{\alpha}}}\\
&\leq C'_1\int_{0}^{+\varepsilon}C n e^{-n(1-C(\varepsilon))t^\alpha}t^{2\alpha-1}dt
+C_2\dfrac{\|h_\alpha\|_\infty}{\varepsilon n^{\frac{1}{\alpha}}}\\
&\leq C'_1\int_{0}^{n^{\frac{1}{\alpha}}\varepsilon}e^{-(1-C(\varepsilon))t^{\alpha}}\frac{t^{2\alpha-1}}{n}dt
+C_2\dfrac{\|h_\alpha\|_\infty}{\varepsilon n^{\frac{1}{\alpha}}}\\
&\leq \frac{C_{\varepsilon,\alpha}}{n}+C_2\dfrac{\|h_\alpha\|_\infty}{\varepsilon n^{\frac{1}{\alpha}}}, 
\end{align*}
for some $C_{\varepsilon,\alpha}>0$, depending only on $\varepsilon$ and $\alpha$, and where again $h_\alpha$ is the 
(bounded) density of the $S\alpha S$-law.  
This concludes the proof of the theorem.
\end{proof}

\section{Stein Equation}\label{sec:SE}
\noindent
Having found in Section \ref{sec:CC} that the operator $\cala_{\operatorname{gen}}$ given for all $f\in BLip(\bbr)$, by
$$\cala_{\operatorname{gen}} f(x)=xf(x)-bf(x)-\int^{+\infty}_{-\infty} (f(x+u)-f(x)
\bbone_{|u|\le 1})u\nu (du),$$
characterizes $X\sim ID(b,0,\nu)$, 
the usual next step in Stein's method is to show that for any 
$h\in \calh$ (a class of nice functions),
\begin{align}\label{eq:SteinGen}
\cala_{\operatorname{gen}} f(x)=h(x)-\bbe h(X),
\end{align}
has a solution $f_h$ which also belongs to a class of nice functions.  
(Of course for
$X\sim ID(b,\sigma^2,\nu)$, the integral operator $\cala_{gen}$ becomes an
integro-differential operator given by
$$\cala_{\operatorname{gen}} f(x)=xf(x)-\sigma^2f'(x)-bf(x)-
\int^{+\infty}_{-\infty}(f(x+u)-f(x)
\bbone_{|u|\le 1})u\nu (du).)$$

Then, when interested in comparing the law of some
random variable  $Y$ to the law of $X$, one needs to estimate 
$$\sup_{h\in\calh}|\bbe h(Y)-\bbe h(X)|=\sup_{h\in\calh}
|\bbe\cala_{\operatorname{gen}} f_h(Y)|.$$
\noindent
In the sequel, we develop a semigroup methodology to solve a corresponding Stein equation for non-degenerate 
self-decomposable laws 
on $\mathbb{R}$. Semigroup methods have been initiated in \cite{Bar90,Go1991} and mainly developed for multivariate normal approximation or for approximation of diffusions. To start with, recall  that by definition $X$,  having characteristic function 
$\varphi$, is self-decomposable 
if for any $\gamma\in (0,1)$, 
\begin{align}
\varphi_\gamma(t):=\frac{\varphi(t)}{\varphi(\gamma t)}, 
\end{align}
$t\in \mathbb{R}$, is itself a characteristic function (\cite[Definition $15.1$]{S}).  
Moreover, recall also that non-degenerate self-decomposable laws are infinitely divisible and absolutely continuous (see \cite[Proposition $15.5$]{S} and \cite[Chapter V, Section $6$, Theorem $6.14$]{SV03}). The class of self-decomposable distributions comprise many of the infinitely divisible ones. To name but a few, the stable distributions, the gamma distributions, the second Wiener chaos type distributions, the Laplace distribution, the Dickman distribution, the Pareto distribution, the log-normal distribution, the logistic distribution, the Student's t distribution are all self-decomposable. We refer the reader to \cite{S,SV03} for more examples and properties of self-decomposable distributions. Next, as usual denote by ${\cal S}(\mathbb{R})$ the Schwartz space of infinitely differentiable rapidly decreasing 
real-valued functions defined on $\mathbb R$, and by $\mathcal{F}$ the Fourier transform operator 
given, for $f\in {\cal S}(\mathbb{R})$, by 
\begin{align}
\mathcal{F}(f)(\xi)=\int_{-\infty}^{+\infty}f(x)e^{-i x\xi}dx.  
\end{align}
For $X\sim ID(b,0,\nu)$, let also 
\begin{align}
\beta^* :=\sup \left\{ \beta \geq 1:\, \int_{|u|>1}|u|^\beta d\nu(u)<+\infty \right\}. 
\end{align}
For $X\sim ID(b,0,\nu)$ non-degenerate and self-decomposable with law $\mu_X$, let $f_X$ be its Radon-Nikodym derivative with respect to the Lebesgue measure, $\operatorname{S}(f_X)=\{x\in \mathbb{R},\, 0\leq f_X(x)<+\infty\}$ and $N(f_X)=\{x\in\mathbb{R},\, f_X(x)=0\}$. Thanks to \cite[Theorem $28.4$]{S}, $f_X$ is continuous on $\operatorname{S}(f_X)$ (actually on $\mathbb{R}$ or on $\mathbb{R}\setminus \{b_0\}$ if $b_0$ exists). The integrability properties of the measure $|u|\nu(du)$ on $\{|u| \leq 1\}$, ensure that the following alternatives hold true, e.g. see \cite[Chapter $5$, Section $24$]{S}
\begin{itemize}
\item If $\int_{|u|\leq 1}|u|\nu(du)<+\infty$, then the support of $\mu_X$ (denoted by $\operatorname{Supp}(\mu_X)$) is $[b_0,+\infty)$ or $(-\infty,b_0]$ or $\mathbb{R}$.
\item If $\int_{|u|\leq 1}|u|\nu(du)=+\infty$, then the support of $\mu_X$ is $\mathbb{R}$.\noindent
\end{itemize}
Before solving the Stein equation \eqref{eq:SteinGen}, we start with:

\begin{prop}\label{lem:SG}
Let $X\sim ID(b,0,\nu)$ be self-decomposable 
with law $\mu_X$, characteristic function $\varphi$ and such that $\bbe|X|<\infty$.  
Let $(P^{\nu}_t)_{t\geq 0}$ be the family of operators defined, for all $t\geq 0$ and 
for all $f\in{\cal S}(\mathbb{R})$, via
\begin{align}
P^{\nu}_t(f)(x)=\frac{1}{2\pi}\int_{-\infty}^{+\infty} \mathcal{F}(f)(\xi)e^{i\xi x e^{-t}}
\frac{\varphi(\xi)}{\varphi(e^{-t} \xi)}d\xi.
\end{align}
Then, $\mu_X$ is invariant for $(P^{\nu}_t)_{t\geq 0}$ and $(P^{\nu}_t)_{t\geq 0}$ extends to a $C_0$-semigroup on $L^p(\mu_X)$, with $1\leq p\leq \beta^*$.   
Its generator $\mathcal{A}$, is defined for all $f\in {\cal S}(\mathbb{R})$ and for all $x\in \mathbb{R}$, by
\begin{align}
\mathcal{A}(f)(x)&=\frac{1}{2\pi}\int_{-\infty}^{+\infty}\mathcal{F}(f)(\xi)e^{i\xi x}(i\xi)
\left(-x+\bbe X+\int_{-\infty}^{+\infty}\left(e^{i u \xi}-1\right)u\nu(du)\right)d\xi\\
&=(\bbe X-x) f'(x)+\int_{-\infty}^{+\infty}\left(f'(x+u)-f'(x)\right)u\nu(du).
\end{align}
\end{prop}

\begin{proof}
First, it is easy to see that for any $f\in {\cal S}(\mathbb{R})$, 
\begin{align*}
&P^\nu_0(f)(x)=f(x), \qquad \underset{t\rightarrow +\infty}{\lim}P^\nu_t(f)(x)
= \int_{-\infty}^{+\infty} f(x)\mu_X(dx), \\
&\int_{-\infty}^{+\infty} P^\nu_t(f)(x)d\mu_X(x)=\int_{-\infty}^{+\infty} f(x)\mu_X(dx).
\end{align*}
Next, let $s,t\geq 0$ and $f\in {\cal S}(\mathbb{R})$.  Then, on the one hand,
\begin{align*}
P_{t+s}^\nu(f)(x)=\frac{1}{2\pi}\int_{-\infty}^{+\infty}\mathcal{F}(f)(\xi)e^{i\xi e^{-(t+s)}x}
\frac{\varphi(\xi)}{\varphi(e^{-(t+s)}\xi)}d\xi, 
\end{align*}
while on the other hand
\begin{align*}
P_{t}^\nu(P_s^\nu(f))(x)&=\frac{1}{2\pi}\int_{-\infty}^{+\infty}\mathcal{F}(P_s^\nu(f))(\xi)e^{i\xi e^{-t}x}\frac{\varphi(\xi)}{\varphi(e^{-t}\xi)}d\xi\\
&=\frac{1}{2\pi}\int_{-\infty}^{+\infty}e^{s}\mathcal{F}(f)(e^s\xi)\dfrac{\varphi(e^s \xi)}{\varphi(\xi)}e^{i\xi e^{-t}x}\frac{\varphi(\xi)}{\varphi(e^{-t}\xi)}d\xi\\
&=\frac{1}{2\pi}\int_{-\infty}^{+\infty}\mathcal{F}(f)(\xi)e^{i\xi e^{-(t+s)}x}
\frac{\varphi(\xi)}{\varphi(e^{-(t+s)}\xi)}d\xi, 
\end{align*}
since $\mathcal{F}(P^\nu_t(f))(\xi)=e^{t}\mathcal{F}(f)(e^t\xi)\dfrac{\varphi(e^t \xi)}{\varphi(\xi)}$.  
The semigroup property is therefore verified on ${\cal S}(\mathbb{R})$.   Now, let $t\in (0,1)$ and let 
$f\in {\cal S}(\mathbb{R})$. Then,
\begin{align*}
\frac{1}{t}\left(P^\nu_t(f)(x)-f(x)\right)=\int_{-\infty}^{+\infty}\mathcal{F}(f)(\xi)e^{i\xi x}\frac{1}{t}\left(e^{i\xi x (e^{-t}-1)}\frac{\varphi(\xi)}{\varphi(e^{-t}\xi)}-1\right)\frac{d\xi}{2\pi}. 
\end{align*}
But by Lemma \ref{lem:PointConv} of the Appendix,
\begin{align*}
\underset{t\rightarrow 0^+}{\lim}\frac{1}{t}\left(e^{i\xi x (e^{-t}-1)}\frac{\varphi(\xi)}{\varphi(e^{-t}\xi)}-1\right)
=\left(-x+\bbe X+\int_{-\infty}^{+\infty}\left(e^{i u \xi}-1\right)u\nu(du)\right)(i\xi).
\end{align*}
Moreover, applying Lemma \ref{lem:PointBound} of the Appendix, for $t\in (0,1)$,
\begin{align*}
\left|\frac{1}{t}\left(e^{i\xi x (e^{-t}-1)}\frac{\varphi(\xi)}{\varphi(e^{-t}\xi)}-1\right)\right|\leq C (1+|\xi|)(\bbe |X|+|x|+|\xi|+1), 
\end{align*}
for some numerical constant $C>0$ independent of $t$. Thus, 
\begin{align*}
\underset{t\rightarrow 0^+}{\lim}\frac{1}{t}\left(P^\nu_t(f)(x)-f(x)\right)=\mathcal{A}(f)(x), 
\end{align*}
and, therefore, the generator of $(P^\nu_t)_{t\geq 0}$ on ${\cal S}(\mathbb{R})$ is indeed $\mathcal{A}$.  
Let $t\geq 0$, $f\in 
{\cal S}(\mathbb{R})$ and $1\leq p\leq \beta^*$. Since $X$ is self-decomposable, there exists a 
probability measure $\mu_t$ such that 
\begin{align}\label{eq:defmut}
\frac{\varphi(\xi)}{\varphi(e^{-t}\xi)}=\int_{-\infty}^{+\infty}e^{i u \xi }\mu_t(du), 
\end{align}
and therefore, 
\begin{align}\label{eq:SpaceRep}
P^\nu_t(f)(x)=\int_{-\infty}^{+\infty} f(u+e^{-t}x)\mu_t(du).  
\end{align}
The previous representation allows to extend the semigroup to $C_b(\mathbb{R})$, the space of bounded continuous 
functions on $\mathbb{R}$ endowed with the supremum norm, and $P^\nu_t(C_b(\mathbb{R}))\subset C_b(\mathbb{R})$. Therefore, $P_t^\nu$ is a contraction semigroup on $C_b(\mathbb{R})$ such that for all $f\in C_b(\mathbb{R})$ and for all $t>0$,
\begin{align}\label{eq:invCb}
\int_{\mathbb{R}} P^\nu_t(f)(x)d\mu_X(x)=\int_{\mathbb{R}} f(x)d\mu_X(x).
\end{align}
and such that for all $f\in C_b(\mathbb{R})$ and for all $x\in\mathbb{R}$
\begin{align}\label{eq:Pointconv}
\underset{t\rightarrow 0^+}{\lim}P^\nu_t(f)(x)=f(x).
\end{align}
Indeed, one can check the invariance property \eqref{eq:invCb} on $C_b(\mathbb{R})$ by noting that the probability measure $\left(\mu_X\otimes \mu_t\right)\circ \psi_t^{-1}$, with $\psi_t(x,y)=e^{-t}x+y$, is equal to $\mu_X$. Similarly, one can check the pointwise convergence property \eqref{eq:Pointconv} by using the fact that, by the L\'evy continuity theorem, $\mu_t\circ \varphi_{t,x}^{-1}$, where $\varphi_{t,x}(u)=e^{-t}x+u$, for all $u\in \mathbb{R}$, converges weakly towards $\delta_x$ as $t\rightarrow 0^+$. The semigroup property of $(P^\nu_t)$ on $C_b(\mathbb{R})$ follows by similar arguments. Finally, 
\begin{align*}
\int_{-\infty}^{+\infty}\left| P^\nu_t(f)(x)\right|^p\mu_X(dx)&\leq \int_{-\infty}^{+\infty} 
P^\nu_t(|f|^p)(x)\mu_X(dx)\\
&\leq \int_{-\infty}^{+\infty} |f(x)|^p\mu_X(dx).  
\end{align*}
A standard approximation argument concludes the proof of the proposition.
\end{proof}
\noindent
\begin{rem}\label{rem:SG}
(i) It is important to note that the representation \eqref{eq:SpaceRep} also allows to extend the semigroup to the space of continuous functions on $\mathbb{R}$ vanishing at $\pm$ infinity. Moreover, this extension is a Feller semigroup (e.g. one can apply \cite[Proposition 2.4]{RY99}).\\
(ii) Self-decomposable distributions are naturally associated with Ornstein-Uhlenbeck type Markov processes. Indeed, thanks to \cite[Theorem 17.5]{S}, for any self-decomposable distributions $\mu$, one can find an Ornstein-Uhlenbeck type Markov process such that $\mu$ is its invariant measure. Hence, from a heuristic point of view, it seems legitimate to implement a semigroup method to solve a Stein equation associated with a self-decomposable law.\\
(iii) When $\nu$ is the L\'evy measure of a symmetric $\alpha$ stable distribution, the generator of the semigroup $(P^\nu_t)_{t\geq 0}$ boils down to
\begin{align*}
\mathcal{A}(f)(x)&=-xf'(x)+\int_{-\infty}^{+\infty}\left(f'(x+u)-f'(x)\right)u\nu(du)\\
&=-xf'(x)+\int_{0}^{+\infty}\left(f'(x+u)-f'(x-u)\right)\frac{cdu}{u^\alpha},
\end{align*}
which is, thanks to \eqref{rep:FracLap}, proportional to the one considered in \cite{X17}.\\
(iv) Recall that since $\mu_X$ is non-degenerate and self-decomposable, its L\'evy measure admits the following representation $\nu(du)=\psi(u)/|u|$, where $\psi$ is a positive function increasing on $(-\infty,0)$ and decreasing on $(0,+\infty)$ (\cite[Corollary $15.11$]{S}). Then, the probability measure $\mu_t$ (defined by \eqref{eq:defmut}) is infinitely divisible. Denoting by $\nu_t$ the L\'evy measure corresponding to $\mu_t$, one easily checks that
\begin{align*}
\nu_t(du)=\frac{\psi(u)-\psi(e^t u)}{|u|}du.
\end{align*}
\end{rem}

With the help of the previous proposition, we now wish to solve the Stein equation associated with 
the operator $\cala$.  More precisely, for any Lipschitz function or bounded Lipschitz function $h$, we wish  
to solve the following integro-differential equation
\begin{align}\label{eq:SteinEqua}
(\bbe X-x) f'(x)+\int_{-\infty}^{+\infty}\left(f'(x+u)-f'(x)\right)u\nu(du)=h(x)-\bbe h(X).
\end{align}
Using
 classical semigroup theory (\cite[Chapter $2$]{Par04} or \cite[Chapter $1$]{EK09}), the next step is to prove that
\begin{align*}
\int_0^{+\infty} \left(P^\nu_t(h)(x)-\bbe h(X)\right)dt
\end{align*}
is well defined when $h$ is a (bounded) Lipschitz function. Let $h$ be a continuously differentiable function on $\mathbb{R}$ such that $\|h\|_\infty\leq 1$ and $\|h'\|_\infty\leq 1$. Since
\begin{align*}
|P^\nu_t(h)(x)-\bbe h(X)|&=\left|\int_{-\infty}^{+\infty} h(y+e^{-t}x)d\mu_t(dy)-\int_{-\infty}^{+\infty} h(y)d\mu_X(y) \right|\\
&\leq \int_{-\infty}^{+\infty} |h(y+e^{-t}x)-h(y)|d\mu_t(y)+d_{W_1}(\mu_t,\mu_X)\\
&\leq e^{-t}|x|+d_{W_1}(\mu_t,\mu_X),
\end{align*}
we need to estimate the rate at which $\mu_t$ converges towards $\mu_X$ in smooth Wasserstein-1 distance. Based on \cite[Theorem $2$]{AMPS17}, we begin by estimating the rate at which $\mu_t$ converges towards $\mu_X$ in smooth Wasserstein-2 distance.

\begin{prop}\label{prop:estimate}
Let $X\sim ID(b,0,\nu)$ be non-degenerate self-decomposable with law $\mu_X$, characteristic function $\varphi$ and such that $\bbe|X|<\infty$. Let $X_t$, $t\geq 0$, be random variables each having a characteristic function
\begin{align}
\varphi_t(\xi)=\dfrac{\varphi(\xi)}{\varphi(e^{-t}\xi)},\quad\quad \xi\in\mathbb{R}.
\end{align}
Then,
\begin{align}\label{ineq:dw2exp}
d_{W_2}(X_t,X) \leq C e^{-\frac{2}{3}t},
\end{align}
for $t$ large enough and for some $C>0$ independent of $t$.
\end{prop}

\begin{proof}
Let $X_t$ be a random variable with law $\mu_t$ given by $\eqref{eq:defmut}$. Then, $X_t$ admits the following representation
\begin{align}
X_t=_{d}(1-e^{-t})\bbe X+X^1_t+X^2_t,
\end{align}
where $X^1_t$ and $X^2_t$ are independent and respectively defined, for all $\xi\in\mathbb{R}$ and for all $t>0$, via
\begin{align*}
\bbe e^{i \xi X^1_t}=\exp\int_{|u|\leq 1}\left(e^{i u \xi}-1-iu\xi\right)\nu_t(du),\quad\quad\bbe e^{i \xi X^2_t}=\exp\int_{|u|> 1}\left(e^{i u \xi}-1-iu\xi\right)\nu_t(du)
\end{align*}
with $\nu_t$ being the L\'evy measure of $X_t$. Moreover, from \cite[inequality $13$]{MR01}
\begin{align}
\bbe \left| X_t\right|\leq (1-e^{-t})\bbe \left|X\right|+\left(\int_{|u|\leq 1}|u|^2\nu_t(du)\right)^{\frac{1}{2}}+2\int_{|u|\geq 1}|u|\nu_t(du),
\end{align}
which implies, in particular, that $\sup_{t>0} \bbe |X_t|<+\infty$, since $\nu_t(B)\leq \nu(B)$ for all Borel sets $B$. Then, hypothesis (H2') of \cite[Theorem $2$]{AMPS17} is satisfied for $\gamma=1$. Let us next estimate the difference between $\varphi_t$ and $\varphi$ the respective characteristic functions of $\mu_t$ and $\mu_X$. Since
\begin{align*}
|\varphi_t(\xi)-\varphi(\xi)|&\leq \left|\frac{\varphi(\xi)}{\varphi(e^{-t}\xi)}\right| |1-\varphi(e^{-t}\xi)|\\
&\leq \bbe|X| |\xi| e^{-t},
\end{align*}
for $t$ large enough, \eqref{ineq:dw2exp} is a straightforward application of \cite[Theorem $2$]{AMPS17}.
\end{proof}
\noindent
To link the smooth-Wasserstein-$2$ distance to the smooth Wasserstein-$1$ distance, a further technical lemma is needed.

\begin{lem}\label{lem:D1D2}
Let $X$ and $Y$ be two random variables such that $d_{W_2}(X,Y)<1$. Then,
\begin{align}
d_{W_1}(X,Y)\leq C\sqrt{d_{W_2}(X,Y)}
\end{align}
for some constant $C>0$.
\end{lem}

\begin{proof}
Let $h$ be a bounded Lipschitz function such that $\|h\|_\infty \leq 1$ and $\|h'\|_\infty\leq 1$. Let $h_\varepsilon$ be a regularization by convolution of $h$ with
\begin{align*}
&\|h_\varepsilon\|_\infty\leq 1,\quad\quad\quad \| h-h_\varepsilon\|_\infty \leq \varepsilon,\\
&\| h_\varepsilon \|_{Lip(1)} \leq 1,\quad\quad\quad \| h''_\varepsilon \|_{\infty} \leq C \varepsilon^{-1},
\end{align*}
for some constant $C>0$ only depending on the convoluting function. Taking $\varepsilon\in (0,C)$,
\begin{align*}
|\bbe h(X)-\bbe h(Y)| &\leq 2\varepsilon+ |\bbe h_\varepsilon(X)-\bbe h_\varepsilon(Y)|\\
&\leq \max(2,C)(\varepsilon+\varepsilon^{-1}d_{W_2}(X,Y)).
\end{align*}
Choosing $\varepsilon=C\sqrt{d_{W_2}(X,Y)}/(1+C)$ concludes the proof of the lemma.
%\begin{align}
%|\bbeh(X)-\bbeh(Y)| \leq C(\varepsilon+\varepsilon^{-1}d_{W_2}(X,Y))+d_{W_2}(X,Y)
%\end{align}
%for some $C>0$ independent on $\varepsilon$. Optimizing in $\varepsilon$ concludes the proof of the proposition.
\end{proof}

\begin{rem}\label{rem:D1BLip}
It is clear from the previous arguments, that
\begin{align*}
\underset{\underset{\|h\|_\infty\leq 1,\ \|h'\|_\infty\leq 1}{h\in C^1(\mathbb{R})}}{\sup}|\bbe h(X)-\bbe h(Y)|=\underset{h\in BLip(1)}{\sup}|\bbe h(X)-\bbe h(Y)|.
\end{align*}
\end{rem}
\noindent
Combining Proposition \ref{prop:estimate} together with Lemma \ref{lem:D1D2} yields,
\begin{align}\label{eq:EST}
|P^\nu_t(h)(x)-\bbe h(X)| \leq e^{-t}|x|+Ce^{-\frac{1}{3}t},
\end{align}
which implies that $\int_0^{+\infty} |P^\nu_t(h)(x)-\bbe h(X)|dt<+\infty$ and which ensures the existence of the function 
\begin{align}\label{eq:SteinSol}
f_h(x)=-\int_0^{+\infty} \left(P^\nu_t(h)(x)-\bbe h(X)\right)dt,\quad x\in\mathbb{R}.
\end{align}
Let us now study the regularity of $f_h$.

\begin{lem}\label{lem:regStein}
Let $h$ be a continuously differentiable function such that $\|h\|_\infty\leq 1$ and $\|h'\|_\infty\leq 1$. Then, $f_h$ is differentiable on $\mathbb{R}$ and $\| f'_h\|_\infty \leq 1$.
\end{lem}

\begin{proof}
Since
\begin{align*}
\frac{d}{dx}\left(P^\nu_t(h)(x)\right)=e^{-t}\int_{-\infty}^{+\infty}h'(xe^{-t}+y)\mu_t(dy),
\end{align*}
it is clear that $f_h$ is differentiable and that $\| f'_h\|_\infty \leq 1$.
\end{proof}
\noindent
When $h$ is a bounded Lipschitz function, the existence of higher order derivatives for $f_h$ is linked to the behavior at $\pm\infty$ of the characteristic function of the probability measure $\mu_t$ as well as to heat kernel estimates for the density of $\mu_t$. To illustrate these ideas, let us provide some examples.\\
\\
\noindent
(i) Let $X$ be a gamma random variable with parameters $(\alpha,1)$. Then, for all $\xi\in \mathbb{R}$,
\begin{align*}
\left|\frac{\varphi(\xi)}{\varphi(e^{-t}\xi)}\right|= \left(\dfrac{1+e^{-2t}\xi^2}{1+\xi^2}\right)^{\frac{\alpha}{2}},
\end{align*}
is a decreasing function of $\xi$ on $\mathbb{R}_+$ and thus,
\begin{align*}
e^{-\alpha t}\leq \left|\frac{\varphi(\xi)}{\varphi(e^{-t}\xi)}\right|\leq 1.
\end{align*}
Similar upper and lower bounds hold for more general probability laws pertaining to the second Wiener chaos.\\
\noindent
(ii) Let $X$ be a Dickman random variable. Then, for all $\xi\in \mathbb{R}$,
\begin{align*}
\frac{\varphi(\xi)}{\varphi(e^{-t}\xi)}=\exp\left(\int_0^1 e^{iu \xi}\dfrac{1-e^{i u\xi(e^{-t}-1)}}{u}du\right).
\end{align*}
Using standard asymptotic expansion for the cosine integral \cite[Formulae $6.12.3$, $6.12.4$ and $6.2.20$, Chapter $6$]{OLBC10}, for all $\xi\in \mathbb{R}$,
\begin{align*}
C_1 \dfrac{1+|\xi| e^{-t}}{1+|\xi|} \leq\left|\frac{\varphi(\xi)}{\varphi(e^{-t}\xi)}\right| \leq C_2 \dfrac{1+|\xi| e^{-t}}{1+|\xi|},
\end{align*}
for some $C_1>0,C_2>0$ not depending on $t$.\\
\noindent
(iii) Let $X$ be a $S\alpha S$ random variable with $\alpha\in (1,2)$. Then, for all $\xi\in \mathbb{R}$,
\begin{align*}
\frac{\varphi(\xi)}{\varphi(e^{-t}\xi)}= e^{-(1-e^{-\alpha t})|\xi|^\alpha}.
\end{align*}
By Fourier inversion, $\mu_t$ admits a smooth density $q$ such that, for all $k\geq 0$,
\begin{align*}
q^{(k)}(t,y)=\frac{1}{2\pi}\int_{-\infty}^{+\infty} \left(\overline{\frac{\varphi(\xi)}{\varphi(e^{-t}\xi)}}\right)e^{iy \xi}(i \xi)^kd\xi.
\end{align*}
Moreover, denoting by $p_\alpha(t,y)$ the probability transition density function of a one dimensional $\alpha$-stable L\'evy process and noting that $q(t,y)=p_\alpha(1-e^{-\alpha t},y)$,
\begin{align*}
|q'(t,y)| \leq C \dfrac{1-e^{-\alpha t}}{\left((1-e^{-\alpha t})^{\frac{1}{\alpha}}+|y|\right)^{\alpha+2}},
\end{align*}
thanks to Lemma 2.2 of \cite{CZ16}. Following the proof of Proposition 4.3 in \cite{X17}, this implies that the function $f_h$ admits a second order derivative uniformly bounded such that
\begin{align*}
\|f''_h\|_\infty\leq C_\alpha \|h'\|_\infty,
\end{align*}
for some constant $C_\alpha>0$, only depending on $\alpha$.\\
\\
The previous examples point out that for some specific target laws the semigroup solution to the Stein equation \eqref{eq:SteinEqua} with $h$ a bounded Lipschitz function might not reach second order differentiability. Nevertheless, if $h$ is a $C^2(\mathbb{R})$ function such that $\|h\|_\infty+\|h'\|_\infty+\|h''\|_\infty\leq 1$, then $f_h$ is twice differentiable with 
\begin{align*}
\|f''_h\|_\infty\leq \frac{1}{2}.
\end{align*}
In this situation, it is possible to partially transform the non-local part of the Stein operator into an operator acting on the second derivative. This is the purpose of the next lemma whose proof is very similar to the one of Lemma $4.6$ of \cite{X17} and, as such, only sketched. 

\begin{lem}\label{lem:Ker}
Let $\nu$ be a L\'evy measure such that $\int_{|u|>1}|u|\nu(du)<+\infty$. Let $f$ be a real-valued 
differentiable function with Lipschitzian first derivative.  Then, for all $N>0$, and all $x\in\mathbb{R}$, 
\begin{align}
\int_{-\infty}^{+\infty} (f'(x+u)-f'(x))u\nu(du)=\int_{-N}^{+N} K_\nu(t,N) f''(x+t)dt+R_N(x), 
\end{align}
where $K_\nu(t,N)$ and $R_N(x)$ are respectively given by
\begin{align*}
K_\nu(t,N)&=\bbone_{[0, N]}(t)\int_{t}^N u\nu(du)+\bbone_{[-N,0]}(t)\int_{-N}^t (-u)\nu(du),\\
R_N(x)&=\int_{|u|>N} (f'(x+u)-f'(x))u\nu(du).
\end{align*}
\end{lem}

\begin{proof}
For $N>0$ and $x\in\mathbb{R}$,
\begin{align}\label{eq:decomp}
\int_{-\infty}^{+\infty} (f'(x+u)-f'(x))u\nu(du)=&\int_0^N(f'(x+u)-f'(x))u\nu(du) \nonumber \\ 
&\qquad +\int_{-N}^0(f'(x+u)-f'(x))u\nu(du) +R_N(x).
\end{align}
For the first term on the right-hand side of \eqref{eq:decomp}, 
\begin{align}\label{eq:pospa}
\int_0^N(f'(x+u)-f'(x))u\nu(du)&=\int_0^N \left(\int_0^u f''(x+t)dt \right)u\nu(du)\nonumber\\
&=\int_0^N f''(x+t)\left(\int_{t}^Nu\nu(du)\right)dt,  
\end{align}
while similar computations for the second term lead to
\begin{align}\label{eq:negpa}
\int_{-N}^0(f'(x+u)-f'(x))u\nu(du)=\int_{-N}^0 f''(x+t)\left(\int_{-N}^t(-u)\nu(du)\right)dt.
\end{align}
Combining \eqref{eq:pospa} and \eqref{eq:negpa} gives
\begin{align*}
\int_{-\infty}^{+\infty} (f'(x+u)-f'(x))u\nu(du)=\int_{-N}^{+N} K_\nu(t,N) f''(x+t)dt+R_N(x).
\end{align*}
\end{proof}
\noindent
Next, we study the regularity properties of the non-local part of the Stein operator. For this purpose, denote it by
\begin{align}
\mathcal{T}(f)(x):=\int_{-\infty}^{+\infty} \left(f'(x+u)-f'(x)\right)u\nu(du).
\end{align}

\begin{prop}\label{lem:RegNL}
Let $\nu$ be a L\'evy measure such that $\int_{|u|>1}|u|\nu(du)<\infty$. Let $f$ be a twice continuously differentiable function such that $\|f'\|_\infty<+\infty$ and $\|f''\|_\infty<+\infty$.\\
(i) If $\int_{|u|\leq 1}|u|\nu(du)<+\infty$, then
\begin{align*}
\|\mathcal{T}(f)\|_{Lip}<+\infty.
\end{align*}
(ii) If $\int_{|u|\leq 1}|u|\nu(du)=+\infty$ and if there exist $\gamma,\beta>0$ and $C_1,C_2>0$ such that for $R>0$,
\begin{align*}
\int_{|u|>R}|u|\nu(du)\leq \frac{C_1}{R^\gamma},\quad\quad\quad \int_{|u|\leq R}|u|^2\nu(du)\leq C_2 R^\beta,
\end{align*}
then,
\begin{align*}
\underset{x\ne y}{\sup}\frac{\left|\mathcal{T}(f)(x)-\mathcal{T}(f)(y)\right|}{\left|x-y\right|^{\frac{\beta}{\beta+\gamma}}}<+\infty.
\end{align*}
\end{prop}

\begin{proof}
Let us start with (i). Let $x,y\in\mathbb{R}$, $x\ne y$. Then,
\begin{align*}
\left|\mathcal{T}(f)(x)-\mathcal{T}(f)(y)\right| &\leq \left|\int_{-\infty}^{+\infty}\left(f'(x+u)-f'(y+u)-f'(x)+f'(y)\right)u\nu(du)\right|\\
&\leq 2 \|f''\|_{\infty}|x-y| \int_{-\infty}^{+\infty} |u|\nu(du).
\end{align*}
which shows that 
\begin{align}
\|\mathcal{T}(f)\|_{Lip}\leq 2\int_{-\infty}^{+\infty}|u|\nu(du)\|f''\|_\infty.
\end{align}
Next, let us prove $(ii)$. Let $R>0$ and $x,y\in\mathbb{R}$, $x\ne y$. Then,
\begin{align*}
\left|\mathcal{T}(f)(x)-\mathcal{T}(f)(y)\right| &\leq \left|\int_{-\infty}^{+\infty}\left(f'(x+u)-f'(y+u)-f'(x)+f'(y)\right)u\nu(du)\right|\\
&\leq \int_{|u|\leq R} \left|f'(x+u)-f'(y+u)-f'(x)+f'(y)\right| |u|\nu(du)\\
&\quad+ \int_{|u|> R}\left|f'(x+u)-f'(y+u)-f'(x)+f'(y)\right| |u|\nu(du).\\
&\leq 2\|f''\|_\infty \left(\int_{|u|\leq R}|u|^2\nu(du)+|x-y| \int_{|u|> R}|u|\nu(du)\right)\\
&\leq 2C\|f''\|_\infty \left(R^\beta+\frac{|x-y|}{R^\gamma}\right).
\end{align*}
Choosing $R=|x-y|^{\frac{1}{\gamma+\beta}}$,
\begin{align*}
\left|\mathcal{T}(f)(x)-\mathcal{T}(f)(y)\right| &\leq 2C\|f''\|_\infty |x-y|^{\frac{\beta}{\beta+\gamma}}.
\end{align*}
This concludes the proof of the proposition.
\end{proof}

\begin{rem}\label{rem:RegNonLoc}
When $\nu$ is the L\'evy measure of a $S\alpha S$ probability distribution, the exponents $\gamma$
and $\beta$ are respectively equal to $\alpha-1$ and $2-\alpha$ and so $\beta/(\gamma+\beta)=2-\alpha$ which is exactly the right order of H\"olderian regularity needed for optimal approximation results as in \cite{X17}.
\end{rem}
\noindent
To end this section, we solve the Stein equation \eqref{eq:SteinEqua} for $h$ a real valued function infinitely differentiable with compact support and with $\|h\|_\infty\leq 1$, $\|h'\|_\infty\leq 1$ and $\|h''\|_\infty\leq 1$. This solution ensures, via the representation \eqref{eq:repcinc}, the existence of  quantitative bounds on the smooth Wasserstein-2 distance as will be shown in the results of the next section. Note also that, thanks to the Fourier approach in defining the semigroup $(P^\nu_t)_{t\geq 0}$, the function $f_h$ is a solution to the Stein equation \eqref{eq:SteinEqua} on the whole real line.

\begin{lem}\label{lem:SteinEqu}
Let $X\sim ID(b,0,\nu)$ be nondegenerate, self-decomposable and such that $\bbe |X|<\infty$. Let $h$ be a real valued function on $\mathbb{R}$ belonging to $C_c^{\infty}(\mathbb{R})$ such that $\|h\|_\infty\leq 1$, $\|h'\|_\infty\leq 1$ and $\|h''\|_\infty\leq 1$. Let $f_h$ be the function given by \eqref{eq:SteinSol}. Then, for all $x\in \mathbb{R}$,
\begin{align}\label{eq:SteinEq2}
(\bbe X-x) f_h'(x)+\int_{-\infty}^{+\infty}\left(f_h'(x+u)-f_h'(x)\right)u\nu(du)=h(x)-\bbe h(X).
\end{align}
\end{lem}

\begin{proof}
Let $h\in C^{\infty}_c(\mathbb{R})$ such that $\|h\|_\infty\leq 1$, $\|h'\|_\infty\leq 1$ and $\|h''\|_\infty\leq 1$ and let $f_h$ be given by \eqref{eq:SteinSol}. Let $\hat{h}$ be defined by $\hat{h}=h-\bbe h(X)$. Let $\psi \in \mathcal{S}(\mathbb{R})$. Now, by Fourier arguments as in the proof of Proposition \ref{lem:SG}
\begin{align}\label{eq:weakheat}
\dfrac{d}{dt}\langle P^{\nu}_t(h);\psi\rangle=\langle \mathcal{A}(P^{\nu}_t(h));\psi\rangle,
\end{align}
where $\langle f;g \rangle=\int_{-\infty}^{+\infty}f(x)g(x)dx$. Then, integrating from $0$ to $+\infty$,
\begin{align}\label{eq:intweakform}
\int_0^{+\infty}\dfrac{d}{dt}\langle P^{\nu}_t(h);\psi\rangle dt=\int_0^{+\infty}\langle \mathcal{A}(P^{\nu}_t(h));\psi\rangle dt.
\end{align}
Let us deal with the left-hand side of equality \eqref{eq:intweakform}. By definition,
\begin{align*}
\int_0^{+\infty}\dfrac{d}{dt}\langle P^{\nu}_t(h);\psi\rangle dt=\underset{t\rightarrow+\infty}{\lim}\langle P^{\nu}_t(h);\psi\rangle-\underset{t\rightarrow 0^+}{\lim}\langle P^{\nu}_t(h);\psi\rangle.
\end{align*}
Straightforward applications of the dominated convergence theorem imply
\begin{align*}
\int_0^{+\infty}\dfrac{d}{dt}\langle P^{\nu}_t(h);\psi\rangle dt=\langle \bbe\,h(X)-h;\psi\rangle.
\end{align*}
To conclude we need to deal with the right hand side of \eqref{eq:intweakform}. First, note that $\mathcal{A}(P^{\nu}_t(h))=\mathcal{A}(P^{\nu}_t(\hat{h}))$. Then,
\begin{align*}
\int_0^{+\infty}\langle \mathcal{A}(P^{\nu}_t(h));\psi\rangle dt&=\langle \int_0^{+\infty}\mathcal{A}(P^{\nu}_t(\hat{h}))dt ;\psi\rangle\\
&=-\langle \mathcal{A}(f_h);\psi\rangle,
\end{align*}
where the interchange of integrals is justified by Fubini theorem. Then, for all $\psi\in \mathcal{S}(\mathbb{R})$,
\begin{align*}
\langle \mathcal{A}(f_h)+\bbe\,h(X)-h;\psi\rangle=0.
\end{align*}
Now, for $x\in \mathbb{R}$ and for $0<\varepsilon \leq 1$, consider the function $\psi_\varepsilon$ belonging to $\mathcal{S} (\mathbb{R})$ and defined by, for all $y\in \mathbb{R}$
\begin{align}
\psi_\varepsilon(y)=\dfrac{1}{\sqrt{2\pi}\varepsilon}\exp\left(-\frac{(x-y)^2}{2\varepsilon^2}\right).
\end{align}
Thanks to the dominated convergence theorem and the regularity of $f_h$ and $h$, it follows that
\begin{align*}
\underset{\varepsilon\rightarrow 0^+}{\lim}\langle \mathcal{A}(f_h)+\bbe\,h(X)-h;\psi_\varepsilon\rangle=\mathcal{A}(f_h)(x)+\bbe\,h(X)-h(x)=0,
\end{align*} 
finishing the proof of the lemma.

%Then, from \cite[Chapter $1$, Proposition $1.5$]{EK09}),
%\begin{align}
%P^\nu_t(\hat{h})-\hat{h}=\mathcal{A}\left(\int_0^t P^\nu_s(\hat{h})ds\right).
%\end{align}
% in $L^1(\mu_X)$. By \eqref{eq:EST},
%\begin{align}
%\underset{t\rightarrow +\infty}{\lim} \|P^\nu_t(\hat{h})\|_{L^1(\mu_X)}=0,\quad\quad \int_0^{+\infty} \|P^\nu_t(\hat{h})\|_{L^1(\mu_X)}dt<+\infty.
%\end{align}
%Therefore, $\int_0^t P^\nu_s(\hat{h})ds$ converges towards $\int_0^{+\infty} P^\nu_s(\hat{h})ds$ in $L^1(\mu_X)$ as $t\rightarrow +\infty$. But, since $\mathcal{A}$ is a closed operator \cite[Chapter $1$, Corollary $1.6$]{EK09}),
%\begin{align}
%\hat{h}=\mathcal{A}(f_h).
%\end{align}
 %in $L^1(\mu_X)$. The end of the proof of the lemma follows from the fact that $\mu_X$ is non-degenerate and self-decomposable and by a continuity argument.
\end{proof}

\section{Applications}\label{sec:AP}
Self-decomposable distributions appear naturally as limiting laws of sum of rather general independent random variables (see \cite{L37,Kh38}). In order to provide quantitative convergence estimates in such general weak limit theorems, let us introduce some notations. Let $(Z_k)_{k\geq 1}$ be a sequence of independent real-valued random variables, $(b_n)_{n\geq 1}$ be a sequence of strictly positive reals such that $b_n\rightarrow 0$ and $b_{n+1}/b_n\rightarrow 1$, $n\rightarrow +\infty$ and finally, let $(c_n)_{n\geq 1}$ be a sequence of reals. Set
\begin{align}\label{eq:Seq}
S_n=b_n\sum_{k=1}^n Z_k+c_n,
\end{align}
where $\{b_nZ_k,\, k=1,...,n,\ n\geq 1\}$ is assumed to be a null array, namely, if $\varphi_{b_nZ_k}$ is the characteristic function of $b_nZ_k$, then
\begin{align}
\underset{n\rightarrow +\infty}{\lim}\underset{1\leq k\leq n}{\max}\left| \varphi_{b_nZ_k}(t)-1\right|=0,
\end{align}
uniformly in $t$ on any compact subset of $\mathbb{R}$. Then, if $(S_n)_{n\geq 1}$ converges in law, the limit is self-decomposable (see \cite[Theorem $15.3$]{S}). Conversely, if the limiting law is self-decomposable, one can always find $(Z_k)_{k\geq 1}$, $(b_n)_{n\geq 1}$ and $(c_n)_{n\geq 1}$ as above such that $(S_n)_{n\geq 1}$ converges in law towards this limit. In the sequel, we quantitatively revisit this result with the help of the Stein methodology developed in the previous sections. First, we need the following straightforward lemma.

\begin{lem}\label{lem:Dec}
Let $(Z_k)_{k\geq 1}$, $(b_n)_{n\geq 1}$ and $(c_n)_{n\geq 1}$ be as above, let $\bbe|Z_k|<\infty$ for all $k\geq 1$ and let $f$ be Lipschitz. Then, for all $n\geq 1$,
\begin{align}
\bbe S_nf'(S_n)=\left(c_n+b_n\sum_{k=1}^n \bbe Z_k\right)\bbe f'(S_n)&+b_n\sum_{k=1}^n \bbe \bbone_{b_n|Z_k|\leq N}\tilde{Z}_k(f'(S_n)-f'(S_{n,k}))\nonumber\\
&+b_n\sum_{k=1}^n \bbe \bbone_{b_n|Z_k|> N}\tilde{Z}_k(f'(S_n)-f'(S_{n,k})),
\end{align}
with $S_{n,k}=S_n-b_nZ_k$, $\tilde{Z}_k=Z_k-\bbe Z_k$ and $N\geq 1$. Moreover, if $f$ is twice continuously differentiable with $\|f'\|_\infty< +\infty$ and $\|f''\|_\infty< +\infty$, then
\begin{align}
\bbe S_nf'(S_n)=\left(c_n+b_n\sum_{k=1}^n \bbe Z_k\right)\bbe f'(S_n)&+b_n\sum_{k=1}^n \bbe \bbone_{b_n|Z_k|> N}Z_k(f'(S_n)-f'(S_{n,k}))\nonumber\\
&+\sum_{k=1}^n\int_{-\infty}^{+\infty}\bbe K_k(t,N)f''(S_{n,k}+t)dt\nonumber\\
&-b_n\sum_{k=1}^n \bbe Z_k \bbe (f'(S_n)-f'(S_{n,k})),
\end{align}
with 
\begin{align}
K_k(t,N)=\bbe b_nZ_k\bbone_{b_n|Z_k|\leq N}(\bbone_{0\leq t\leq b_nZ_k}-\bbone_{b_nZ_k\leq t\leq 0}).
\end{align}
\end{lem}

\begin{proof}
For $f$ Lipschitz, the first part of the lemma follows from
\begin{align}
\bbe \tilde{Z}_kf'(S_n)=\bbe \tilde{Z}_k(f'(S_n)-f'(S_{n,k})),
\end{align}
which is valid for all $k\geq 1$, from the independence of  $\tilde{Z}_k$ and $S_{n,k}$, and since $\bbe \tilde{Z}_k=0$. For $f$ twice continuously differentiable with $\|f'\|_\infty< +\infty$ and $\|f''\|_\infty< +\infty$, the result follows by computations similar to the proof of \cite[Lemma $4.5$]{X17} using also Lemma \ref{lem:Ker}.
\end{proof}
\noindent
In the next theorem, let $X\sim ID (b,0,\nu)$ be non-degenerate, self-decomposable with law $\mu_X$ and such that $\bbe |X|<\infty$. Let further $(S_n)_{n\geq 1}$ be as in \eqref{eq:Seq} with $\bbe |Z_k|<+\infty$, for all $k\geq 1$. %Moreover, the following standing assumptions are in play:
%\begin{itemize}
%\item If $\int_{|u|\leq 1}|u|\nu(du)< +\infty$, then a.s. for all $n\geq 1$, $S_n$ belongs to the support of $\mu_X$ which is $[b_0,+\infty)$ or $(-\infty,b_0]$ or $\mathbb{R}$. 
%\item If $\int_{|u|\leq 1}|u|\nu(du)=+\infty$, then no further assumption is required.
%\end{itemize}
From the results of the previous section (Lemma \ref{lem:SteinEqu}), for any $h\in C^{\infty}_c(\mathbb{R})$ with $\|h\|_\infty\leq 1$, $\|h'\|_\infty\leq 1$ and $\|h''\|_\infty\leq 1$, 
\begin{align}\label{eq:STEIN}
\left|\bbe h(S_n)-\bbe h(X)\right|=\left|\bbe \left((\bbe X-S_n)f_h'(S_n)+\int_{-\infty}^{+\infty}\left(f_h'(S_n+u)-f_h'(S_n)\right)u\nu(du)\right)\right|.
\end{align}
The next theorem upperbounds \eqref{eq:STEIN}.

\begin{thm}\label{thm:GWLT}
(i) Let $\int_{|u|\leq 1}|u|\nu(du)< +\infty$. Then, for $n,N\geq 1$,
\begin{align*}
d_{W_2}(S_n,X)&\leq  \left|\bbe X-(c_n+b_n\sum_{k=1}^n \bbe Z_k)\right|+\frac{b_n}{n}\left(\sum_{k=1}^n\bbe |Z_k|\right)\left(\int_{-\infty}^{+\infty}|u|\nu(du)\right) +\frac{1}{2}b_n^2\sum_{k=1}^n|\bbe Z_k|\bbe |Z_k|\nonumber\\
&\quad+2\int_{|u|>N}|u|\nu(du)+2b_n\sum_{k=1}^n \bbe |Z_k|\bbone_{|b_nZ_k|>N}+\frac{1}{2}\sum_{k=1}^n\int_{-N}^{+N} \left| \frac{K_\nu(t,N)}{n}-K_k(t,N)\right|dt.
\end{align*}
(ii) Let $\int_{|u|\leq 1}|u|\nu(du)=+\infty$ and let there exist $\gamma>0,\beta>0$ and $C_1>0,C_2>0$ such that for $R>0$
\begin{align}
\int_{|u|>R}|u|\nu(du)\leq \frac{C_1}{R^\gamma},\quad\quad\quad \int_{|u|\leq R}|u|^2\nu(du)\leq C_2 R^\beta.
\end{align}
Then, for $n,N\geq 1$,
\begin{align*}
d_{W_2}(S_n,X)&\leq  \left|\bbe X-(c_n+b_n\sum_{k=1}^n \bbe Z_k)\right|+C_{\gamma,\beta}\frac{b_n^{\frac{\beta}{\beta+\gamma}}}{n}\left(\sum_{k=1}^n\bbe |Z_k|^{\frac{\beta}{\beta+\gamma}}\right) +\frac{1}{2}b_n^2\sum_{k=1}^n|\bbe Z_k|\bbe |Z_k|\nonumber\\
&\quad+2\int_{|u|>N}|u|\nu(du)+2b_n\sum_{k=1}^n \bbe |Z_k|\bbone_{|b_nZ_k|>N}+\frac{1}{2}\sum_{k=1}^n\int_{-N}^{+N} \left| \frac{K_\nu(t,N)}{n}-K_k(t,N)\right|dt,
\end{align*}
for some $C_{\gamma,\beta}>0$ only depending on $\gamma$ and $\beta$. 
\end{thm}
\begin{proof}
Let us start with the proof of $(i)$. Assume that $\int_{|u|\leq 1}|u|\nu(du)< +\infty$. Let $h\in C^{\infty}_c(\mathbb{R})$ be such that $\|h\|_\infty\leq 1$, $\|h'\|_\infty\leq 1$ and $\|h''\|_\infty\leq 1$. Then,
\begin{align}\label{eq:FIRST}
\left|\bbe h(S_n)-\bbe h(X)\right|&=\left|\bbe \left((\bbe X-S_n)f_h'(S_n)+\int_{-\infty}^{+\infty}\left(f_h'(S_n+u)-f_h'(S_n)\right)u\nu(du)\right)\right|\nonumber\\
&\leq \left|\bbe \left(\bbe\, X -(c_n+b_n\sum_{k=1}^n \bbe Z_k)\right)f_h'(S_n) \right|\nonumber\\
&\quad+\left| \bbe\, b_n\sum_{k=1}^n \tilde{Z}_k(f'_h(S_{n,k}+Z_kb_n)-f'_h(S_{n,k}))-\mathcal{T}(f_h)(S_n)\right|\nonumber\\
&\leq \left|\bbe\, X-\left(c_n+b_n\sum_{k=1}^n \bbe Z_k\right)\right|+\left|\bbe \frac{1}{n}\sum_{k=1}^n \mathcal{T}(f_h)(S_{n,k})-\mathcal{T}(f_h)(S_n)\right|\nonumber\\
&\quad+\left| \bbe\, b_n\sum_{k=1}^n \tilde{Z}_k(f'_h(S_{n,k}+Z_kb_n)-f'_h(S_{n,k}))-\frac{1}{n}\sum_{k=1}^n \mathcal{T}(f_h)(S_{n,k})\right|\nonumber\\
&\leq \left|\bbe\, X-\left(c_n+b_n\sum_{k=1}^n \bbe Z_k\right)\right|+\frac{b_n}{n}\left(\int_{-\infty}^{+\infty}|u|\nu(du)\right)\sum_{k=1}^n\bbe |Z_k|\nonumber\\
&\quad+\left| \bbe\, b_n\sum_{k=1}^n \tilde{Z}_k(f'_h(S_{n,k}+Z_kb_n)-f'_h(S_{n,k}))-\frac{1}{n}\sum_{k=1}^n \mathcal{T}(f_h)(S_{n,k})\right|, 
\end{align}
where we have successively used Lemma \ref{lem:regStein} and Proposition \ref{lem:RegNL} (i). Next, we need to bound
\begin{align*}
I:=\left| \bbe b_n\sum_{k=1}^n \tilde{Z}_k(f'_h(S_{n,k}+Z_kb_n)-f'_h(S_{n,k}))-\frac{1}{n}\sum_{k=1}^n \mathcal{T}(f_h)(S_{n,k})\right|.
\end{align*}
First,
\begin{align*}
I&\leq \left| \bbe\, b_n\sum_{k=1}^n \bbe Z_k (f'_h(S_{n,k}+Z_kb_n)-f'_h(S_{n,k}))\right|\\
&\quad+\left| \bbe\, b_n\sum_{k=1}^n Z_k(f'_h(S_{n,k}+Z_kb_n)-f'_h(S_{n,k}))-\frac{1}{n}\sum_{k=1}^n \mathcal{T}(f_h)(S_{n,k})\right|\\
&\leq \frac{1}{2}b_n^2\sum_{k=1}^n|\bbe Z_k|\bbe |Z_k|+\left| \bbe b_n\sum_{k=1}^n Z_k(f'_h(S_{n,k}+Z_kb_n)-f'_h(S_{n,k}))-\frac{1}{n}\sum_{k=1}^n \mathcal{T}(f_h)(S_{n,k})\right|.
\end{align*}
Let $N\geq 1$. Now, from Lemma \ref{lem:Ker},
\begin{align*}
\frac{1}{n}\sum_{k=1}^n \mathcal{T}(f_h)(S_{n,k})=\frac{1}{n}\sum_{k=1}^n\left(\int_{-N}^{+N} K_\nu(t,N) f_h''(S_{n,k}+t)dt+R_N(S_{n,k})\right), 
\end{align*}
and moreover,
\begin{align*}
\frac{1}{n}\sum_{k=1}^n\bbe |R_N(S_{n,k})|\leq 2\int_{|u|>N}|u|\nu(du).  
\end{align*}
Therefore,
\begin{align*}
I&\leq \frac{1}{2}b_n^2\sum_{k=1}^n|\bbe Z_k|\bbe |Z_k|+2\int_{|u|>N}|u|\nu(du)+\bigg| \bbe b_n\sum_{k=1}^n Z_k(f'_h(S_{n,k}+Z_kb_n)-f'_h(S_{n,k}))\\
&\quad\qquad \qquad \qquad \qquad \qquad \qquad\quad -\frac{1}{n}\sum_{k=1}^n\int_{-N}^{+N} K_\nu(t,N) f_h''(S_{n,k}+t)dt \bigg|,
\end{align*}
and from Lemma \ref{lem:Dec},
\begin{align*}
\bbe b_n\sum_{k=1}^n Z_k(f'_h(S_{n,k}+Z_kb_n)-f'_h(S_{n,k}))&=b_n\sum_{k=1}^n \bbe \bbone_{b_n|Z_k|> N}Z_k(f_h'(S_{n,k}+Z_kb_n)-f_h'(S_{n,k}))\nonumber\\
&\qquad \qquad +\sum_{k=1}^n\int_{-\infty}^{+\infty}\bbe K_k(t,N)f_h''(S_{n,k}+t)dt.
\end{align*}
Then,
\begin{align}\label{eq:LAST}
I&\leq \frac{1}{2}b_n^2\sum_{k=1}^n|\bbe Z_k|\bbe |Z_k|+2\int_{|u|>N}|u|\nu(du)+2b_n\sum_{k=1}^n \bbe |Z_k|\bbone_{|b_nZ_k|>N}\nonumber\\
&\quad\quad\quad \qquad \qquad +\frac{1}{2}\sum_{k=1}^n\int_{-N}^{+N} \left| \frac{K_\nu(t,N)}{n}-K_k(t,N)\right|dt.
\end{align}
Combining \eqref{eq:FIRST} and \eqref{eq:LAST} proves part $(i)$. For the proof of $(ii)$, proceed 
in a similar way.
\end{proof}

\begin{rem}\label{rem:Bound}
(i) Using Lemma \ref{lem:D1D2}, it is possible to transfer the bounds on $d_{W_2}(S_n,X)$ to bounds on $d_{W_1}(S_n,X)$.\\
(ii) The approach just developed generalizes the methodology developed for the $S\alpha S$ distribution in \cite{X17}. Note that in this case, due to regularizing properties of the probability transition density function of the one dimensional $\alpha$-stable L\'evy process, it is possible to obtain quantitative upper bounds in Wasserstein-1 distance. At the level of the semigroup solution to the Stein equation, this can be viewed by a gain of one order of differentiability with a test function $h$ only Lipschitz. In this regard, Theorem \ref{thm:GWLT} can be compared with Theorem 2.1 of \cite{X17} (from which several quantitative convergence results follow).\\
(iii) As first noticed in \cite{X17}, the quantity 
\begin{align*}
\frac{1}{2}\sum_{k=1}^n\int_{-N}^{+N} \left| \frac{K_\nu(t,N)}{n}-K_k(t,N)\right|dt,
\end{align*}
is the $L^1$ analog of the $L^2$ Stein discrepancy considered e.g. in \cite{Ch08,Ch09,NP09,LNP15}. As such, it seems to be the appropriate Stein quantity to deal with sums of independent summands.
\end{rem}
\noindent
To illustrate the abstract bounds obtained in the previous theorem, we consider the important case where the sequence of random variables, $(Z_k)_{k\geq 1}$, is a sequence of independent and identically distributed random variables such that $\bbe |Z_k|<+\infty$, $k\geq 1$. In this situation, thanks to \cite[Theorem 15.7]{S}, the limiting self-decomposable law is actually stable. Recall that an infinitely divisible probability measure, $\mu$, is stable if, for any $a>0$, there are $b>0$ and $c\in \mathbb{R}$ such that, for all $t\in \mathbb{R}$
\begin{align}
\varphi(t)^a=\varphi(b t)e^{i c t},
\end{align}
where $\varphi$ is the characteristic function of $\mu$. Also, by \cite[Theorem 14.3(ii)]{S}, any non-degenerate stable distribution with index $\alpha\in (0,2)$ has a L\'evy measure given by \eqref{StableLM}. Note finally that, when $\nu$ is given by \eqref{StableLM} with $\alpha\in (1,2)$, $\int_{|u|\leq 1}|u|\nu(du)=+\infty$ and that, for all $R>0$,
\begin{align}
\int_{|u|>R}|u|\nu(du)=\frac{c_1+c_2}{\alpha-1}\frac{1}{R^{\alpha-1}},\quad\quad\quad \int_{|u|\leq R}|u|^2\nu(du)=\frac{c_1+c_2}{2-\alpha}R^{2-\alpha}.
\end{align}

\begin{cor}\label{cor:generalstable}
Let $X$ be a non-degenerate stable random variable with index $\alpha\in (1,2)$ and L\'evy measure given by \eqref{StableLM}. Let $(S_n)_{n\geq 1}$ be as in \eqref{eq:Seq} with $(Z_k)_{k\geq 1}$ independent and identically distributed and with $\bbe |Z_1|<+\infty$. Then, for $n,N\geq 1$,
\begin{align}\label{eq:genstablecase}
d_{W_2}(S_n,X)&\leq  \left|\bbe X-(c_n+n b_n \bbe Z_1)\right|+C_{\alpha}(b_n)^{2-\alpha}\bbe\, |Z_1|^{2-\alpha} +\frac{n}{2}b_n^2|\bbe Z_1|\bbe |Z_1|\nonumber\\
&\quad+2\frac{c_1+c_2}{\alpha-1}\frac{1}{N^{\alpha-1}}+2nb_n \bbe |Z_1|\bbone_{|b_nZ_1|>N}+\frac{1}{2}\int_{-N}^{+N} \left| K_\nu(t,N)-nK_1(t,N)\right|dt,
\end{align}
where $C_{\alpha}>0$ only depends on $\alpha$, $c_1$ and $c_2$, and, for all $t\in \mathbb{R}$, $t\ne 0$ and $N\geq 1$, $n\geq 1$
\begin{align*}
&K_\nu(t,N)=c_1\bbone_{[0, N]}(t)\dfrac{t^{\alpha-1} -N^{\alpha-1}}{\alpha-1}+c_2\bbone_{[-N,0]}(t)\dfrac{N^{1-\alpha}-(-t)^{1-\alpha}}{1-\alpha},\\
&K_1(t,N)=\bbe b_nZ_1\bbone_{b_n|Z_1|\leq N}(\bbone_{0\leq t\leq b_nZ_1}-\bbone_{b_nZ_1\leq t\leq 0}).
\end{align*}
\end{cor}

\begin{proof}
This is a direct application of (ii) Theorem \ref{thm:GWLT} ($\beta=2-\alpha$ and $\gamma=\alpha-1$) together with the fact that the random variables $Z_k$ are identically distributed. Moreover, thanks to \eqref{StableLM} with $\alpha\in (1,2)$ and to Lemma \ref{lem:Ker}, for all $t\in \mathbb{R}$, $t\ne 0$, and for all $N\geq 1$
\begin{align}
K_\nu(t,N)=c_1\bbone_{[0, N]}(t)\dfrac{t^{\alpha-1} -N^{\alpha-1}}{\alpha-1}+c_2\bbone_{[-N,0]}(t)\dfrac{N^{1-\alpha}-(-t)^{1-\alpha}}{1-\alpha}.
\end{align}
This concludes the proof of the corollary.
\end{proof}

\begin{rem}\label{rem:domainattraction}
By assuming further properties on the tails behavior of the law of $Z_1$ as done, for example in \cite[Theorem $2.6$ and Corollary $2.7$]{X17}, it is possible to extract explicit rates of convergence from the right-hand side of \eqref{eq:genstablecase}. For further details, we refer the reader to the proofs of \cite[Theorem $2.6$ and Corollary $2.7$]{X17}.
\end{rem}
\noindent
Finally, as a last example, let $(Z_{n,k})_{1\leq k\leq n,\, n\geq 1}$ be a null array (see e.g. \cite[Definition $9.2$]{S}), namely, for each fixed $n\geq 1$, $Z_{n,1},...,Z_{n,n}$ are independent random variables and, for all $\epsilon>0$
\begin{align}
\underset{n\rightarrow +\infty}{\lim}\underset{1\leq k\leq n}{\max} \mathbb{P}\left(|Z_{n,k}|>\epsilon \right)=0.
\end{align}
Let $(S_n)_{n\geq 1}$ be the sequence of row sums associated with this triangular array, namely, for all $n\geq 1$
\begin{align}\label{eq:rowsum}
S_n=\sum_{k=1}^n Z_{n,k}.
\end{align}
By a result of Khintchine (see e.g. \cite[Theorem $9.3$]{S}), if for some $c_n\in \mathbb{R}$, for all $n\geq 1$, the distribution of $S_n+c_n$ converges in distribution to $X$ with distribution $\mu$, then $\mu$ is infinitely divisible. Thus, by a straightforward adaptation of the proofs of Lemma \ref{lem:Dec} and Theorem \ref{thm:GWLT}, the following result holds.

\begin{thm}\label{thm:triangulararray}
Let $X\sim ID (b,0,\nu)$ be non-degenerate, self-decomposable with law $\mu_X$ and such that $\bbe |X|<\infty$. Let $(S_n)_{n\geq 1}$ be given by \eqref{eq:rowsum} such that $\bbe |Z_{n,k}|<\infty$, for all $n\geq 1$, for all $k=1,...,n$ and let $c_n\in \mathbb{R}$, for all $n\geq 1$.\\
(i) Let $\int_{|u|\leq 1}|u|\nu(du)< +\infty$. Then, for $n,N\geq 1$,
\begin{align*}
d_{W_2}(S_n+c_n,X)&\leq  \left|\bbe X-(c_n+\sum_{k=1}^n \bbe Z_{n,k})\right|+\frac{1}{n}\left(\sum_{k=1}^n\bbe |Z_{n,k}|\right)\left(\int_{-\infty}^{+\infty}|u|\nu(du)\right) +\frac{1}{2}\sum_{k=1}^n|\bbe Z_{n,k}|\bbe |Z_{n,k}|\nonumber\\
&\quad+2\int_{|u|>N}|u|\nu(du)+2\sum_{k=1}^n \bbe |Z_{n,k}|\bbone_{|Z_{n,k}|>N}+\frac{1}{2}\sum_{k=1}^n\int_{-N}^{+N} \left| \frac{K_\nu(t,N)}{n}-K_k(t,N)\right|dt,
\end{align*}
where, for all $t\in \mathbb{R}$, for all $N\geq 1$ and for all $1\leq k\leq n$
\begin{align}
K_k(t,N)=\bbe Z_{n,k}\bbone_{|Z_{n,k}|\leq N}(\bbone_{0\leq t\leq Z_{n,k}}-\bbone_{Z_{n,k}\leq t\leq 0}).
\end{align}
(ii) Let $\int_{|u|\leq 1}|u|\nu(du)=+\infty$ and let there exist $\gamma>0,\beta>0$ and $C_1>0,C_2>0$ such that for $R>0$
\begin{align}
\int_{|u|>R}|u|\nu(du)\leq \frac{C_1}{R^\gamma},\quad\quad\quad \int_{|u|\leq R}|u|^2\nu(du)\leq C_2 R^\beta.
\end{align}
Then, for $n,N\geq 1$,
\begin{align*}
d_{W_2}(S_n+c_n,X)&\leq  \left|\bbe X-(c_n+\sum_{k=1}^n \bbe Z_{n,k})\right|+\frac{C_{\gamma,\beta}}{n}\left(\sum_{k=1}^n\bbe |Z_{n,k}|^{\frac{\beta}{\beta+\gamma}}\right) +\frac{1}{2}\sum_{k=1}^n|\bbe Z_{n,k}|\bbe |Z_{n,k}|\nonumber\\
&\quad+2\int_{|u|>N}|u|\nu(du)+2\sum_{k=1}^n \bbe |Z_{n,k}|\bbone_{|Z_{n,k}|>N}+\frac{1}{2}\sum_{k=1}^n\int_{-N}^{+N} \left| \frac{K_\nu(t,N)}{n}-K_k(t,N)\right|dt,
\end{align*}
for some $C_{\gamma,\beta}>0$ only depending on $\gamma$ and $\beta$. 
\end{thm}

%\section{Concluding Remarks}\label{sec:CR}
We finish our manuscript by briefly addressing, among many others, three possible extensions and generalizations of our current work which will be presented elsewhere. A first possible direction of future research is to solve the Stein equation associated with general infinitely divisible distributions with finite first moment (not only the self-decomposable target laws). %Recall that their characterizing operator is
%\begin{align}
%\cala_{\operatorname{gen}} f(x)=xf(x)-\sigma^2f'(x)-bf(x)-
%\int^{+\infty}_{-\infty}(f(x+u)-f(x))
%\bbone_{|u|\le 1})u\nu (du).
%\end{align}
%One step further would be to find characterizing identities (and in particular characterizing operators) for infinitely divisible distributions without finite first moment. Note that this class of ID distributions comprises the $\alpha$-stable distributions with parameter $\alpha\in (0,1]$ and in particular, the Cauchy distribution.\\
A second possible direction of research to which our methods are amenable is the study of extensions to multivariate  
(and even infinite dimensional) settings of the results presented here. (In particular, see \cite{HPAS} for the validity of the 
multivariate covariance representation.)  
A third direction would be to attempt at removing the finite first moment assumption which is present throughout our hypotheses.

\appendix
\section{Appendix}
\label{sec:appendix}
This section is devoted to the proof of two technical results used in the previous sections. 
\begin{lem}\label{lem:PointConv}
Let $X\sim ID(b,0,\nu)$ be self-decomposable with characteristic function $\varphi$ and such that $\bbe |X|<\infty$. Let $x,\xi\ne 0$. Then,
\begin{align*}
\underset{t\rightarrow 0^+}{\lim}\frac{1}{t}\left(e^{i\xi x(e^{-t}-1)}\dfrac{\varphi(\xi)}{\varphi(e^{-t}\xi)}-1\right)=\left(-x+\bbe X+\int_{-\infty}^{+\infty}\left(e^{iu\xi}-1\right)u\nu(du)\right)(i\xi).
\end{align*}
\end{lem}

\begin{proof}
Let $x,\xi\ne 0$. First, it is clear that 
\begin{align}\label{eq:firsterm}
\underset{t\rightarrow 0^+}{\lim}\frac{1}{t}(e^{i\xi x (e^{-t}-1)}-1)=-i\xi x.
\end{align}
Next, since $X$ is infinitely divisible with finite first moment, for all $t\geq 0$,
\begin{align*}
\dfrac{\varphi(\xi)}{\varphi(e^{-t}\xi)}=e^{i\xi \bbe X(1-e^{-t})}e^{\int_{-\infty}^{+\infty}\left(e^{i u \xi}-e^{i u \xi e^{-t}}-iu\xi(1-e^{-t})\right)\nu(du)}.
\end{align*}
Now, from \eqref{eq:firsterm},
\begin{align*}
\underset{t\rightarrow 0^+}{\lim}\frac{1}{t}(e^{i\xi \bbe X(1-e^{-t})}-1)=i \xi \bbe X,
\end{align*}
moreover,
\begin{align*}
\underset{t\rightarrow 0^+}{\lim}e^{\int_{-\infty}^{+\infty}\left(e^{i u \xi}-e^{i u \xi e^{-t}}-iu\xi(1-e^{-t})\right)\nu(du)}=1.
\end{align*}
So, to conclude the proof, one needs to show that
\begin{align*}
\underset{t\rightarrow 0^+}{\lim}\frac{1}{t}\left(e^{\int_{-\infty}^{+\infty}\left(e^{i u \xi}-e^{i u \xi e^{-t}}-iu\xi(1-e^{-t})\right)\nu(du)}-1\right)=\int_{-\infty}^{+\infty}\left(e^{iu\xi}-1\right)u\nu(du)(i\xi).
\end{align*}
For this purpose, let us show that 
\begin{align}\label{eq:fourterm}
\underset{t\rightarrow 0^+}{\lim}\frac{1}{t}\int_{-\infty}^{+\infty}\left(e^{i u \xi}-e^{i u \xi e^{-t}}-iu\xi(1-e^{-t})\right)\nu(du)=\int_{-\infty}^{+\infty}\left(e^{iu\xi}-1\right)u\nu(du)(i\xi).
\end{align}
For the real part of \eqref{eq:fourterm}, one wants to prove
\begin{align*}
\underset{t\rightarrow 0^+}{\lim}\frac{1}{t}\int_{-\infty}^{+\infty}\left(\cos(u\xi)-\cos( u \xi e^{-t})\right)\nu(du)=-\int_{-\infty}^{+\infty}\sin(u\xi)u\nu(du)\xi.
\end{align*}
Noting that for all $u\ne 0$
\begin{align}
\underset{t\rightarrow 0^+}{\lim}\frac{1}{t}\left(\cos(u\xi)-\cos( u \xi e^{-t})\right)=-\xi u\sin(u\xi),
\end{align}
this is a straightforward application of the dominated convergence theorem. The imaginary part of \eqref{eq:fourterm} is treated in a similar fashion.
\end{proof}

\begin{lem}\label{lem:PointBound}
For any $t\geq 0$, let $Y_t$ be an ID random variable such that $\bbe |Y_t|<+\infty$, $\bbe Y_t=0$ and with L\'evy measure
\begin{align*}
\nu_t(du)=\frac{\psi(u)-\psi(e^t u)}{|u|}du,
\end{align*}
where $\psi$ is a nonnegative function increasing on $(-\infty,0)$ and decreasing on $(0,+\infty)$ and such that
\begin{align*}
\int_{|u|\leq 1}|u|\psi(u)du<\infty,\quad\quad \int_{|u|>1} \psi(u)du<\infty.
\end{align*}
Then, for all $\xi\in\mathbb{R}$ and for all $t\in (0,1)$
\begin{align*}
\frac{1}{t}\left|\bbe e^{i\xi Y_t}-1\right|\leq C_\psi (|\xi|+|\xi|^2)
\end{align*}
for some $C_\psi>0$ only depending on $\psi$.
\end{lem}

\begin{proof}
Let $\xi\ne 0$ and $t\in (0,1)$. First, since $\bbe|Y_t|<\infty$ and since $Y_t$ has zero mean,
\begin{align*}
\bbe e^{i\xi Y_t}=e^{\int_{-\infty}^{+\infty}\left(e^{i u \xi}-1-iu\xi\right)\nu_t(du)}.
\end{align*}
Then,
\begin{align*}
\left|e^{\int_{-\infty}^{+\infty}\left(e^{i u \xi}-1-iu\xi\right)\nu_t(du)}-1\right|\leq |\xi| \underset{\omega\in [0,|\xi|]}{\max}\left|\int_{-\infty}^{+\infty}u\left(e^{i u \omega}-1\right)\nu_t(du) \right|.
\end{align*}
Moreover, for $\omega\in [0,|\xi|]$
\begin{align*}
\left|\int_{-\infty}^{+\infty}u\left(e^{i u \omega}-1\right)\nu_t(du) \right|\leq |\omega|\int_{|u|\leq 1}|u|^2\nu_t(du)+2\int_{|u|>1}|u|\nu_t(du).
\end{align*}
Let us bound the two terms $\int_{|u|\leq 1}|u|^2\nu_t(du)$ and $\int_{|u|>1}|u|\nu_t(du)$. It is sufficient for our purpose to consider these integrals on $(0,+\infty)$ since similar arguments provide the same type of bounds on $(-\infty, 0)$. For the first term,
\begin{align}
\int_0^{1}u\left(\psi(u)-\psi(e^t u)\right)du&=\int_0^1u\psi(u)du-\int_0^1u\psi(e^t u)du\nonumber\\
&=\int_0^1u\psi(u)du-e^{-2t}\int_0^{e^{t}}u\psi(u)du\nonumber\\
&=\int_1^{e^t}u\psi(u)du+(1-e^{-2t})\int_0^{e^t} u\psi(u)du\nonumber\\
&\leq \underset{1<u<e}{\sup}|u \psi(u)|(e^t-1)+(1-e^{-2t})\int_0^e u\psi(u)du.\nonumber
\end{align}
Thus,
\begin{align*}
\int_{|u|\leq 1}|u|^2\nu_t(du)\leq C_\psi ((e^t-1)+(1-e^{-2t})).
\end{align*}
for some constant $C_\psi>0$ only depending on $\psi$. For the second term
\begin{align}
\int_1^{+\infty} (\psi(u)-\psi(e^t u))du&=\int_1^{+\infty}\psi(u)du-e^{-t}\int_{e^t}^\infty \psi( u)du\nonumber\\
&\leq \int_1^{e^t}\psi(u)du+(1-e^{-t})\int_{1}^\infty \psi(u)du\nonumber\\
&\leq C_\psi \left((e^t-1)+(1-e^{-t})\right).\nonumber
\end{align}
The conclusion of the lemma then follows.
\end{proof}

\begin{lem}\label{lem: repsmwa	s}
Let $r\geq 1$, and let $X$ and $Y$ be two random variables. Then,
\begin{align}
d_{W_r}(X,Y)=\underset{h\in C^{\infty}_c(\mathbb{R})\cap\mathcal{H}_r}{\sup}\left|\bbe\, h(X)-\bbe\, h(Y)\right|.
\end{align}
\end{lem}

\begin{proof}
Let $r\geq 1$ and $X$ and $Y$ be two random variables with respective laws $\mu_X$ and $\mu_Y$. First, note that
\begin{align*}
d_{W_r}(X,Y)\geq \underset{h\in C^{\infty}_c(\mathbb{R})\cap\mathcal{H}_r}{\sup}\left|\bbe\, h(X)-\bbe\, h(Y)\right|.
\end{align*}
Now, let $h$ be in $\mathcal{H}_r$. Let $(h_\varepsilon)_{\varepsilon>0}$ be a regularization of $h$ by convolution such that
\begin{align*}
\|h-h_\varepsilon\|_{\infty}\leq \varepsilon,\quad \|h^{(k)}_\varepsilon\|_\infty\leq 1,\quad 0\leq k\leq r.
\end{align*}
Let $\phi$ be a compactly supported infinitely differentiable even function with values in $[0,1]$ and such that $\phi(x)=1$, for $x\in [-1,1]$. Let $M\geq 1$ and set $\phi_M(x)=\phi(x/M)$, for all $x\in \mathbb{R}$. Next, denote by $h_{M,\varepsilon}$ the $C^{\infty}_c(\mathbb{R})$ function defined, for all $x\in \mathbb{R}$, via
\begin{align*}
h_{M,\varepsilon}(x)=\phi_M(x)h_\varepsilon(x).
\end{align*}
Then,
\begin{align*}
\left| \bbe\, h(X)-\bbe\, h(Y)\right| &\leq \left| \bbe\, h_{M,\varepsilon}(X)-\bbe\, h_{M,\varepsilon}(Y)\right|+\bbe\left|h(X)-h_{M,\varepsilon}(X) \right|+\bbe\left|h(Y)-h_{M,\varepsilon}(Y) \right|\\
&\leq \left| \bbe h_{M,\varepsilon}(X)-\bbe h_{M,\varepsilon}(Y)\right|+\bbe\left|h(X)-h_{\varepsilon}(X) \right|+\bbe\left|h_\varepsilon(X)-h_{M,\varepsilon}(X) \right|\\
&\quad\quad+\bbe\left|h(Y)-h_{\varepsilon}(Y) \right|+\bbe\left|h_\varepsilon(Y)-h_{M,\varepsilon}(Y) \right|\\
&\leq \left| \bbe\, h_{M,\varepsilon}(X)-\bbe\, h_{M,\varepsilon}(Y)\right|+2\varepsilon+\bbe\left|h_\varepsilon(X)-h_{M,\varepsilon}(X) \right|+\bbe\left|h_\varepsilon(Y)-h_{M,\varepsilon}(Y) \right|\\
&\leq \left| \bbe\, h_{M,\varepsilon}(X)-\bbe\, h_{M,\varepsilon}(Y)\right|+2\varepsilon+\int_{\mathbb{R}} \left|1-\phi_M(x)\right|d\mu_X(x)+\int_{\mathbb{R}} \left|1-\phi_M(y)\right|d\mu_Y(y).
\end{align*}
Now, choosing $M\geq 1$ large enough so that 
\begin{align*}
\int_{\mathbb{R}} \left|1-\phi_M(x)\right|d\mu_X(x)+\int_{\mathbb{R}} \left|1-\phi_M(y)\right|d\mu_Y(y)\leq 2\varepsilon,
\end{align*}
it follows that, for such $M\geq 1$,
\begin{align*}
\left| \bbe h(X)-\bbe h(Y)\right| &\leq \left| \bbe h_{M,\varepsilon}(X)-\bbe h_{M,\varepsilon}(Y)\right|+4\varepsilon.
\end{align*}
Moreover, for all $x\in \mathbb{R}$, for all $M\geq 1$ and for all $\varepsilon>0$
\begin{align*}
|h_{M,\varepsilon}(x)| \leq 1,
\end{align*}
while, for all $x\in \mathbb{R}$ and for all $1\leq k\leq r$,
\begin{align*}
|h^{(k)}_{M,\varepsilon}(x)|&\leq \sum_{p=0}^k \binom{k}{p} |h^{(k-p)}_\varepsilon(x)| |\phi^{(p)}_M(x)|\\
& \leq  \left(1+C_k\sum_{p=1}^k\frac{1}{M^p}\right),
\end{align*}
for some $C_k>0$ which only depends on $k$ and on $\phi$. Thus, 
\begin{align*}
\left| \bbe h(X)-\bbe h(Y)\right| \leq  \left(1+C_r\sum_{p=1}^r \frac{1}{M^p}\right)\underset{h\in C^{\infty}_c(\mathbb{R})\cap\mathcal{H}_r}{\sup}\left|\bbe h(X)-\bbe h(Y)\right|+4\varepsilon,
\end{align*}
for some appropriate constant $C_r>0$ only depending on $r>0$ and on $\phi$. Letting first $M\rightarrow +\infty$ and then, $\varepsilon\rightarrow 0^+$ gives the result.
\end{proof}

\end{document}